\documentclass[reqno]{amsart}
\usepackage[utf8]{inputenc}
\usepackage{amssymb, amsthm}
\usepackage{mathrsfs}
\usepackage{bm,bbm}
\usepackage{microtype} 
\usepackage{comment} 

\usepackage{amsfonts}
\DeclareMathAlphabet{\mathpgoth}{OT1}{pgoth}{m}{n}
\DeclareMathAlphabet{\mathesstixfrak}{U}{esstixfrak}{m}{n}
\DeclareMathAlphabet{\mathboondoxfrak}{U}{BOONDOX-frak}{m}{n}
\usepackage{amsmath,amssymb}
\usepackage[bbgreekl]{mathbbol}
\usepackage{yfonts,mathtools}

\usepackage{tikz-cd}
\usepackage[all,cmtip]{xy}

\numberwithin{equation}{section}
\usepackage{soul}

\usepackage{hyperref}
\hypersetup{colorlinks}
\definecolor{darkred}{rgb}{0.5,0,0}
\definecolor{darkgreen}{rgb}{0,0.5,0}
\definecolor{darkblue}{rgb}{0,0,0.5}
\hypersetup{colorlinks, linkcolor=darkblue, filecolor=darkgreen, urlcolor=darkred, citecolor=darkblue}
\makeatletter 
\@addtoreset{equation}{section}
\makeatother  

\numberwithin{equation}{section}

\newtheorem{thma}{Theorem}

\newtheorem{thm}{Theorem}[section]
\newtheorem{cor}[thm]{Corollary}

\newtheorem{prop}[thm]{Proposition}
 
\newtheorem{lemma}[thm]{Lemma}
\theoremstyle{definition}
\newtheorem{defn}[thm]{Definition}
\theoremstyle{remark}
\newtheorem{rem}[thm]{Remark}

\usepackage{xcolor}
\newtheorem{example}[thm]{Example}

\newcommand{\beq}{\begin{equation}}
\newcommand{\eeq}{\end{equation}}
\newcommand{\beqn}{\begin{equation*}}
\newcommand{\eeqn}{\end{equation*}}
\newcommand{\ov}{\overline}
\newcommand{\mb}{\mathbb}
\newcommand{\mc}{\mathcal}
\newcommand{\mf}{\mathfrak}

\newcommand{\dep}{{\rm depth}}

\newcommand{\wt}{\widetilde}
\newcommand{\wh}{\widehat}

\newcommand{\uds}[1]{\underline{\smash{#1}}}
\renewcommand{\outer}{{}_{\lfloor}}
\newcommand{\pman}{\uds{\bf \Psi Man}}
\newcommand{\regpos}{\uds{\bf RegPoset}}

\newcommand{\ev}{{\rm ev}}

\newcommand{\om}{\omega}
\newcommand{\la}{\lambda}
\newcommand{\cl}{\mathcal}
\newcommand{\ol}{\overline}

\newcommand{\F}{{\mathbb F}}
\newcommand{\R}{{\mathbb R}}
\newcommand{\C}{{\mathbb C}}
\newcommand{\Z}{{\mathbb Z}}

\title{A proof of the Arnold--Givental Conjecture}

\author{Shaoyun Bai}
\address{MIT, 77 Massachusetts Avenue, Cambridge, MA 02139, USA}
\email{shaoyunb@mit.edu}

\author{Egor Shelukhin}
\address{D\'epartement de Math\'ematiques et de Statistique, Universit\'e de
Montr\'eal, C.P. 6128 Succ. Centre-Ville, Montreal (Québec), H3C 3J7, Canada} 
\email{egor.shelukhin@umontreal.ca}

\author{Yi Wang}
\address{Hetao Institute of Mathematics and Interdisciplinary Sciences (HIMIS), Shenzhen, China} \email{wangyi@himis-sz.cn}

\author{Guangbo Xu}
\address{Department of Mathematics, Rutgers University, Hill Center--Busch Campus, 110 Frelinghuysen Road, Piscataway, NJ 08854-8019, USA}
\email{guangbo.xu@rutgers.edu}

\date{\today}

\begin{document}

\begin{abstract}
We prove the Arnold--Givental conjecture in full generality: given a closed symplectic manifold $(X, \omega)$, an anti-symplectic involution $\tau_X: X \to X$ with fixed point set $L={\rm Fix}(\tau_X)$, and a Hamiltonian diffeomorphism $\phi: X \to X$ such that $\phi(L)$ intersects transversely with $L$, the following inequality holds: 
\beqn
\# \big( \phi(L) \cap L \big) \geq {\rm dim}_{{\mb F}_2} H_*(L; {\mb F}_2).
\eeqn

The proof combines the methods of integral Floer theory of the first and fourth authors, a reduction to Hamiltonian Floer cohomology due to Lu, and a new idea related to localization in a $\mathbb Z/2$-equivariant Floer theory tailored to the problem.

\end{abstract}

\maketitle

\setcounter{tocdepth}{1}
\tableofcontents

\section{Introduction}\label{sec:intro}

A conjecture of Arnold and Givental\footnote{A more phonetic transliteration from Russian would be Arnol'd and Givental'.} \cite{Givental-per} from 1989 states that if $(X,\om)$ is a closed symplectic manifold and $\tau_X: X \to X$ is an anti-symplectic involution, then its fixed point manifold $L = {\rm Fix}(\tau_X),$ which is easily seen to be a possibly disconnected Lagrangian submanifold if non-empty, is symplectically rigid. More precisely, if $\phi$ is a Hamiltonian diffeomorphism of $X$ such that $\phi(L)$ is transverse to $L,$ then \[ \#(\phi(L) \cap L) \geq \dim_{\F_2} H_*(L;\F_2)\] where $\F_2$ is the field with two elements. It is convenient to note that $\dim_{\F_2} H_*(L;\F_2) = \dim_{\F_2} H^*(L;\F_2),$ as it is the latter quantity that we will eventually work with. We may also freely assume that $L$ is non-empty.

Prior results about this question were obtained by Givental \cite{Givental-per} for $X = \C P^n,$ $L = \R P^n,$ where $\tau_X$ is the standard complex conjugation; by Oh \cite{Oh-AG} for certain real forms $L$ of compact Hermitian spaces $X$ (see also \cite{another-ag} for extensions along these lines); by Fukaya--Oh--Ohta--Ono in \cite[unpublished]{FOOO_Chap82009ref} for $X$ satisfying a special positivity assumption for holomorphic spheres and in \cite{FO3-inv} for $X$ symplectically Calabi-Yau, where in the latter case only $\phi(L)\cap L\neq \varnothing$ is proved; by Lazzarini \cite{lazzarini} for certain $L = {\rm Fix}(\tau_X)$ in negative monotone manifolds $X$; and by Frauenfelder in \cite{Frauenfelder-inv} for certain symplectic quotients.

We note that already Givental's result shows that $\F_2$ is the right coefficient field to impose stronger rigidity from the viewpoint of the number of intersection points; indeed, the cohomology of $\R P^n$ is much smaller over coefficients of characteristic away from $2.$ 

In this paper, we establish the Arnold--Givental conjecture in full generality.

\begin{thma}\label{thm:main}
    Let $(X, \omega)$ be a closed symplectic manifold and suppose that $\tau_X: X \to X$ is an anti-symplectic involution. Denote by $L$ the fixed point locus of $\tau_X$, which is a Lagrangian submanifold. Then if $\phi: X \to X$ is a Hamiltonian diffeomorphism such that $\phi(L)$ and $L$ intersect transversely, we have
    \beqn
    \# (\phi(L) \cap L) \geq \dim_{\mathbb{F}_2} H_*(L;\mathbb{F}_2).
    \eeqn
\end{thma}

\subsection{Overall strategy}\label{section:overall strategy}
The overall theme of the proof of Theorem \ref{thm:main} is based on developing what can be considered to be a far-reaching Floer-theoretic generalization of Smith's estimates on the cohomology of real-algebraic varieties (see \cite{Smith-ra} for example) by means of the Smith inequality for $\Z/2$ actions in cohomology over $\F_2.$ 

To explain the strategy, we recall an observation due to Guangcun Lu \cite{Lu_2008} that allows one to reduce the Arnold--Givental conjecture to counting certain one-periodic orbits of a $\tau_X$-symmetric Hamiltonian $H \in C^\infty(\R/\Z \times X, \R)$. Here, the $\tau_X$-symmetry means that \[H(t,x) = H(-t, \tau_X(x))\] for all $(t,x) \in \R/\Z \times X$. 

Namely, given a Hamiltonian diffeomorphism $\phi:X\to X$ with $\phi(L)\pitchfork L$ and a Hamiltonian $K(t,x)$ with $\phi^1_K=\phi$, by a reparametrization we can assume $K(t,x) = 0$ for all $t$ near $0$ in $\R/\Z$. We define the ``double" of $K$ by $H(t,x) = 2K(2t,x)$ for $t \in [0,1/2]$ and $H(t,x) = 2K(-2t,\tau_X(x))$ for $t \in [-1/2,0]$. Then, first of all, $H$ is $\tau_X$-symmetric and $\phi_H^{1/2}=\phi$. Second, let $\cl LX$ denote the space of contractible loops in $X$, and consider the involution 
\begin{equation}\label{eqn:T-action}
\begin{aligned}
T:\cl LX &\to \cl LX, \\
z(t) &\mapsto \tau_X(z(-t)).
\end{aligned}
\end{equation}
Denoting by $\cl O(H) \subset \cl LX$ the set of contractible one-periodic orbits of $H$, note that $T$ preserves $\cl O(H)$ as $H$ is $\tau_X$-symmetric. Let  \[\cl O(H)^{\Z/2} = \{z \in \cl O(H)\;|\; Tz = z\}\] denote the fixed point set of the $\Z/2$ action on $\cl O(H)$ induced by $T.$ Then there is an injection \begin{equation}\label{eq:injection} \cl O(H)^{\Z/2} \to \phi(L) \cap L,\end{equation} given by $z \mapsto z(1/2).$\footnote{This map becomes an isomorphism in general if we consider instead of $\cl LM$ the space of {\em all} loops in $X$; however for our purposes only the component of contractible loops is relevant.}

Moreover, we may perturb $H$ slightly to $H_1$ in $C^2$-topology within the class of $\tau_X$-symmetric Hamiltonians to make all 1-periodic orbits of $H_1$ non-degenerate while keeping $\phi_{H_1}^{1/2}(L)\cap L=\phi(L)\cap L$. Henceforth we will thus simply assume that $H$ is non-degenerate. Therefore, given any Hamiltonian diffeomorphism $\phi: X \to X$ such that $\phi(L) \pitchfork L$, to find a lower bound for $\# (\phi(L) \cap L)$, it suffices to find a lower bound for $\# \cl O(H)^{\Z/2}$ for every nondegenerate $\tau_X$-symmetric Hamiltonian $H.$ 

\begin{rem}
Write $\tau_X = \tau$. It is easy to see that if $K$ generates $\phi = \phi^1_K$ then the double $H$ of $K$ as defined above generates $\psi = \phi^1_H = \tau \phi^{-1} \tau \phi.$ Then the $T$ action on $\cl O(H)$ corresponds to the $\tau$ action on $\mathrm{Fix}(\psi),$ namely $x \mapsto \tau x.$ Indeed, for $y \in X,$ $\psi y = y$ if and only if $\phi y = \tau \phi \tau y.$ Hence if $\psi x = x,$ that is $\phi x = \tau \phi \tau x,$ then $\phi(\tau x) = \tau (\tau \phi \tau x) = \tau \phi x = \tau \phi \tau (\tau x)$ so $\psi \tau x = \tau x.$
\end{rem}

Building on the above observation, we provide an alternative description from the variational perspective. Let now $H$ be a non-degenerate $\tau_X$-symmetric Hamiltonian on $X.$ We consider a suitable cover \[\tilde{\cl L} X \to \cl LX,\] for instance the one corresponding to \[\ker([\om]) \subset \pi_2(X).\] Here, $\tilde{\cl L} X$ consists of equivalence classes of pairs of the form $(u, \uds \gamma)$ where $\uds \gamma: {\mb R}/{\mb Z} \to X$ is a smooth loop and $u: {\mb D}^2 \to X$ such that $u(e^{2\pi t i}) = \uds \gamma(t),$ where $(u, \uds \gamma) \sim (u', \uds \gamma')$ if $\uds \gamma = \uds \gamma'$ and $u$ concatenated with $u'$ with reversed orientation pairs trivially with $\omega$. Crucially, the involution $T$ from \eqref{eqn:T-action} lifts to such a cover: for a capping disk $u: {\mb D}^2 \to X$ with $u(e^{2\pi t i}) = \uds \gamma(t),$ $t \in \R/\Z,$ we set \[(Tu)(w) = \tau(u(\overline{w})),\] where $\overline{w}$ denotes the complex conjugate of $w \in {\mb D}^2 \subset \C.$

Then, the action functional
\beqn
\cl A_{H}([(u, \uds \gamma)]) := - \int_{{\mb D}^2} u^* \omega + \int_0^1 H(t,\uds \gamma(t))dt
\eeqn
satisfies the symmetry 
\beqn
T^* \cl A_{H} = \cl A_{H}.
\eeqn 
The critical points of $\cl A_{H}$, after projecting down to $\cl LX$, correspond to $\cl O(H)$, and the ones invariant under $T$-action correspond to $\cl O(H)^{\Z/2}$. Finally note that if $J=\{J_t\}_{t\in \R/\Z}$ is an $\R/\Z$-dependent $\om$-compatible almost complex structure, then $T_* J = \{ -(\tau_X)_* J_{-t} \}_{t \in \R/\Z}$ is also an $\R/\Z$-dependent $\om$-compatible almost complex structure. 

This indicates that we may define a (Tate-)equivariant Floer cohomology  \[\widehat{HF}^*_{\Z/2}(H;\F_2)\] of $H$ with respect to the $\Z/2$ action provided by $T.$ This cohomology is a vector space over the universal Novikov field \[\Lambda_{\cl K} = \{ \sum a_i q^{\la_i}\;|\; a_i \in \cl K, \; \la_i \to \infty\},\] where we recall that $\cl K = \F_2[\theta^{-1},\theta]]$ is the field of Laurent power series over $\F_2.$ Assuming the putative existence of $\widehat{HF}^*_{\Z/2}(H;\F_2)$, general considerations based on equivariant Morse inequalities should show that 
\begin{equation}\label{eqn:lower-bound}
\# \cl O(H)^{\Z/2} \geq \dim_{\Lambda_{\cl K}} \widehat{HF}^*_{\Z/2}(H;\F_2).
\end{equation}
Indeed, we observe that the $\Z/2$-orbits in $\cl O(H)$ are of two types: a free pair $O=\{x,T(x)\}$ where $T(x)\neq x$, and a fixed point $O=\{y\}=\{T(y)\}.$ We can define the local Tate cohomology \[\widehat{HF}^*_{{\rm loc},\Z/2}(O;\F_2)\] of each $\Z/2$-orbit $O$ and observe that it vanishes if $O$ is free and \[\widehat{HF}^*_{{\rm loc},\Z/2}(O;\F_2) \cong \cl K\] if $O=\{y\} \subset \cl O(H)^{\Z/2}.$ This implies by a homological perturbation argument that the desired inequality \eqref{eqn:lower-bound} holds.

What is left to bridge with the Arnold--Givental conjecture is a computation of $\widehat{HF}^*_{\Z/2}(H;\F_2)$. To this end, we may prove by means of an equivariant PSS map \cite{PSS} that we have an isomorphism 
\begin{equation}\label{PSS-iso}
\widehat{H}^*_{\Z/2}(X;\Lambda_{\cl K}) \xrightarrow{\cong}  \widehat{HF}^*_{\Z/2}(H;\F_2).
\end{equation}
As for $\widehat{H}^*_{\Z/2}(X;\Lambda_{\cl K})$, we recall the following fact from classical Smith inequalities, as in \cite{Borel}. Consider the Tate $\Z/2$ equivariant cohomology \[\widehat{H}^*_{\Z/2}(X;\F_2) = \theta^{-1} {H}^*_{\Z/2}(X;\F_2)\] with respect to the $\Z/2$-action provided by $\tau_X$. Then, the localization theorem in equivariant cohomology implies that the inclusion $L = {\rm Fix}(\tau_X) \to X$ induces an isomorphism\footnote{ It is not hard to see by well-established techniques that this isomorphism implies the following Smith inequality, which we will not use explicitly: $\dim_{\F_2} H^*(L;\F_2) \leq \dim_{\F_2} H^*(X;\F_2)^{\Z/2}.$ }, \begin{equation}\label{eq:iso} \widehat{H}^*_{\Z/2}(X;\F_2) \xrightarrow{\cong} H^*(L;\F_2) \otimes_{\F_2} \cl K.\end{equation} 

Therefore, for instance, by choosing a $\Z/2$-invariant Morse function with positive definite Hessian in the normal direction to $L$, one can prove that 
\begin{align}\label{eq:formula}\displaystyle \dim_{{\Lambda_{\cl K}} }  \widehat{HF}^*_{\Z/2}(H;\F_2) = \dim_{\Lambda_{\cl K}}\widehat{H}^*_{\Z/2}(X;\Lambda_{\cl K})  = \dim_{\Lambda_{\cl K}}\widehat{H}^*_{\Z/2}(X;\cl K) \otimes_{\cl K} \Lambda_{\cl K} \\ \nonumber =  \dim_{\cl K}\widehat{H}^*_{\Z/2}(X;{\cl K})   = \dim_{\cl K} H^*(L;\cl K) = \dim_{\F_2} H^*(L;\F_2),\end{align} where we used \eqref{eq:iso} in the last line. Therefore, by combining \eqref{eqn:lower-bound}, \eqref{eq:formula}, we obtain a proof of Theorem \ref{thm:main}. Technically, it is easier to only prove the injectivity of the map \[\widehat{H}^*_{\Z/2}(X;\Lambda_{\cl K}) \xrightarrow{}  \widehat{HF}^*_{\Z/2}(H;\F_2),\] and hence obtain an inequality which is still sufficient for our argument.

\subsection{Outlook} We list the following two directions of research which we plan to investigate in future work:
\begin{enumerate}
\item First, it is interesting to see whether $L = {\rm Fix}(\tau_X)$ is symplectically rigid from other perspectives in the field of symplectic topology, in addition to the existence of Lagrangian intersections under Hamiltonian deformation as proved in Theorem \ref{thm:main}. Namely, it is interesting to see whether $L$ is heavy or superheavy in the terminology of \cite{Entov_Polterovich_3} with base field $\F_2$ or equivalently \cite{MSV-jtop} whether it is $SH$-heavy in the sense of \cite{Varolgunes-GT, VT, DGPZ}; finally, it would be interesting to see whether $L$ is unobstructed in a suitable sense over base field $\F_2$ (see \cite{FO3-book, FOOO_2013, Rabah, Rabah_thesis} for example) and has non-vanishing resulting Floer cohomology.

\item Second, it is interesting to continue studying the symmetric $1$-periodic orbits of iterations $H^{(k)}(t,x) = kH(kt,x),$ $k\geq 2,$ of a symmetric Hamiltonian $H(t,x).$ Note that such $H^{(k)}$ are also symmetric and hence their symmetric one-periodic orbits correspond to intersection points $\phi^{1/2}_{H^{(k)}}(L) \cap L.$ Note that $\phi^{1/2}_{H^{(k)}} = \phi^{1}_{H^{(\frac{k}{2})}}$ if $k$ is even and $\phi^{1/2}_{H^{(k)}} = \phi^{1/2}_{H} \phi^{1}_{H^{(\frac{k-1}{2})}}$ if $k$ is odd.

We believe that our methods, combined with arguments from Hamiltonian dynamics, such as \cite{shelukhin-22, BSWX, CGG-hypersurfaces} can provide interesting results in this direction, including new progress on the Seifert conjecture. We refer to \cite{ETDS-Seifert, AIHP-Seifert} for recent results and a review of the literature in this direction.
\end{enumerate}

\subsection{Comments on technical difficulties}
As the reader may tell, our strategy resembles the proof of the integral Arnold conjecture of the first and fourth authors \cite{Bai_Xu_Arnold}. Namely, we 1) define the corresponding Floer cohomology, 2) define the PSS and SSP maps between the Floer cohomology and the Morse cohomology, 3) use gluing to obtain a homotopy between the composition of PSS and SSP maps and an automorphism on the Morse cohomology, 4) we conclude that the rank of Morse cohomology is a lower bound on the number of generators of the Floer cohomology. In the case of the Arnold--Givental conjecture, the additional difficulties include the following:
\begin{enumerate}
    \item \emph{Regularizing the moduli spaces in the equivariant context}. The moduli spaces involved in the construction, which are spelled out in detail in Section \ref{sec:moduli}, are regularized using global Kuranishi charts \cite{AMS,AMS2} in the Floer-theory setting \cite{Bai_Xu_Arnold,Bai_Xu_foundation}. The additional feature in the current setting is the presence of an involution which conjugates the stable complex structures on the moduli spaces. This is distinctive from the equivariant setup in \cite[Part 4]{Bai_Xu_foundation}.
    \item \emph{FOP perturbations with additional ${\mb Z}/2$-symmetry.} To define counts over ${\mb F}_2$, we use the Fukaya--Ono--Parker (FOP) perturbation scheme developed by the first and fourth authors \cite{Bai_Xu_2022}. Besides performing an abstract Borel construction following \cite[Part 4]{Bai_Xu_foundation} at the level of global Kuranishi charts, we need to ensure that the FOP transversality condition is preserved under involutions on normally complex orbifolds relevant for the current problem. Luckily, this is a direct consequence of the diffeomorphism-invariance property of the canonical Whitney stratification that is built into the construction of FOP transverse perturbations.
    \item \emph{Tate cohomology via semi-infinite sphere.} We highlight our approach to Tate cohomology. Broadly, we follow the Morse-theoretic Borel construction of \cite{seidel-pants, Hendricks_Lipshitz_Sarkar, Seidel_Smith_localization}. More precisely, we ``couple'' the Floer equation with the (standard) Morse flow lines inside $S^\infty$ which has a free ${\mb Z}/2$-action. Using the construction of \cite{Bai_Xu_foundation}, we can also extend the infinite-dimensional sphere $S^\infty$ to the negative direction, i.e., defining
    \beqn
    \wh S^\infty = \Big\{ (h_0, h_{\pm 1}, h_{\pm 2},\ldots)\ |\ h_i \in {\mb R},\ \sum_i |h_i|^2 = 1,\ |i| \gg 0\Longrightarrow h_i = 0 \Big\}
    \eeqn
and extending the Morse theory to $\wh S^\infty$. This allows us to geometrically realize the Tate construction (i.e. geometrically inverting the equivariant variable).
\end{enumerate}

Finally, we remark that equivariant localization in Floer theory may not hold in general in the most naive sense (cf. \cite{large2019equivariantfloertheorydouble}). In our setting, we \emph{do not} prove a localization statement for the Tate equivariant Floer cohomology $\widehat{HF}^*_{\Z/2}(H;\F_2)$. Instead, we appeal to the equivariant PSS map and classical considerations in finite-dimensional Morse theory to deduce our result.

\subsection{Outline}
This manuscript is divided into two parts. 

Part I consists of Section \ref{sec:moduli} and Section \ref{sec:proof}, in which we respectively discuss the relevant moduli spaces for the Floer-theoretic study and provide a proof of Theorem \ref{thm:main} highlighting the formal aspects without providing details on global Kuranishi charts and FOP perturbations.

Part II is the technical core of the paper.
\begin{itemize}
    \item In Section \ref{sec:abstract}, we recall abstract notions on manifolds with corners and flow categories following \cite[Part 1]{Bai_Xu_foundation} and detail the parts where the involution matters.
    \item In Section \ref{sec:FOP}, we recall FOP perturbations for normally complex orbifolds developed in \cite{Bai_Xu_2022} and prove the necessary extension in the presence of an involution. We also state the properties satisfied by the global Kuranishi/derived orbifold charts, which will be used in Section \ref{subsec:official-proof} to implement the strategy described in Section \ref{sec:proof} rigorously.
    \item In Section \ref{section6}, we describe how to construct global Kuranishi charts for the moduli spaces discussed in Section \ref{sec:moduli}.
    \item Finally, in Section \ref{sec:nc-smoothing}, we provide details on constructing normal complex structures on the derived orbifold charts of the moduli spaces and smoothing constructions.
\end{itemize}

\begin{rem}
We learned that Amanda Hirschi has announced a proof of the Arnold--Givental conjecture in a joint work-in-progress with Noah Porcelli on August 26, 2026 during the event ``Symplectic Topology - A conference in honor of Kai Cieliebak" at the University of Augsburg.
\end{rem}

\subsection*{Acknowledgments}
S.B. was supported by the NSF standard grant DMS-2404843 and CAREER grant DMS-2540393. E.S. was supported by an NSERC Discovery Grant, an FRQNT Teams Grant, and by the Courtois chair in fundamental research; he thanks Mohammed Abouzaid, Urs Frauenfelder, and Leonid Polterovich for useful discussions on related topics. Y.W. would like to thank Kenji Fukaya for helpful discussion. G.X. was supported by NSF grants DMS-2345030 and DMS-2506403, and a Simons Foundation Travel Grant, and would like to thank Mohammad Tehrani, Penka Georgieva, and Mohamad Rabah for helpful discussions. 

\section*{Part I.}

\section{Moduli spaces}\label{sec:moduli}
In this section, we set up notation for the moduli spaces involved in our Morse-theoretic and Floer-theoretic constructions. We use notation different from that in the Introduction to align better with the discussions in Part II.

\subsection{Involution symmetric Morse complex}

Let $X$ be a compact smooth manifold and $\tau_X: X \to X$ be a smooth involution. 

\begin{prop}
There exists a $\tau_X$-equivariant Morse--Smale pair. 
\end{prop}

\begin{proof}
By the main result of \cite{Mayer_1989}, there exists an invariant stable Morse function on $X$. By a result of Bao--Lawson \cite{Bao_Lawson_2024}, one can achieve the Morse--Smale condition for stable Morse functions.
\end{proof}

Then we choose a $\tau_X$-equivariant Morse--Smale pair. In notations, we only indicate the Morse function $f_X: X \to {\mb R}$ but suppress the Riemannian metric. We use $\kappa$ to denote a general critical point of $f_X$.

To consider equivariant theory, we ``couple'' the Morse theory of $f_X$ with the Morse theory of the classifying space of ${\mb Z}/2$. Consider the Borel model for ${\mb Z}/2$-equivariant theories. One has 
\beqn
S^\infty:= \Big\{ (h_0, h_1, \ldots)\ |\ h_i \in {\mb R},\ i\gg 0 \Longrightarrow h_i = 0,\ \sum_{i=0}^\infty h_i^2 = 1 \Big\}.
\eeqn
The group ${\mb Z}/2$ acts on $S^\infty$ by flipping all signs. Then the quotient is $\mb{RP}^\infty$ whose cohomology ring over ${\mb F}_2$ is 
\beqn
H_{{\mb Z}/2}^{\rm Borel} ({\rm pt}) = H^*(\mb{RP}^\infty; {\mb F}_2) = {\mb F}_2[[\theta]].
\eeqn
One can use a special Morse--Smale pair on $S^\infty$ to give a chain model. Define
\beqn
f_{\rm Borel}: S^\infty \to {\mb R},\ f_{\rm Borel} (h_0, h_1, \ldots) = - \sum_{i=0}^\infty i |h_i|^2.
\eeqn
Then the standard metric satisfies the Morse--Smale condition (on any finite truncations). 

One can geometrically obtain a Tate model as follows. Define
\beqn
\wh S^\infty:= \Big\{ h =  (h_i)_{i \in {\mb Z}}\ |\ h_i \in {\mb R},\ |i|\gg 0 \Longrightarrow h_i = 0,\ \sum_{i\in {\mb Z}} h_i^2 = 1 \Big\}
\eeqn
which contains $S^\infty$. The function $f_{\rm Borel}$ can be extended to 
\beqn
f_{\rm Tate}(h) = - \sum_{i \in {\mb Z}} i|h_i|^2.
\eeqn
The Riemannian metric also extends in the natural way. Then the critical points of $f_{\rm Tate}$ and critical values are
\begin{align*}
&\ \mu_\pm^{(k)} = ( \pm \delta_{ik})_{i \in {\mb Z}}, &\ f_{\rm Tate}(\mu_\pm^{(k)}) = -k,
\end{align*}
where $\delta_{ik}$ is the Kronecker delta of $i$ and $k$.

Now we define the coupled Morse function
\beqn
\wt f_X: X \times \wh S^\infty \to {\mb R},\ \wt f_X(x, h) = f_X(x) + f_{\rm Tate}(h).
\eeqn
Critical points are denoted by 
\beqn
v = (\kappa_v, \mu_v)
\eeqn
where $\kappa_v \in {\rm crit} f_X$ and $\mu_v \in {\rm crit} f_{\rm Tate}$. For each pair of critical points $v, w \in {\rm crit}(\wt f_X)$, one has the moduli space
\beqn
M_{vw}^{\mb M}
\eeqn
of smooth or broken negative (downward) gradient flow lines (modulo componentwise translation) from $v$ to $w$.\footnote{Here the superscript ${\mb M}$ stands for ``Morse.'' It will also denote the Tate-Morse flow category after the abstract notion of flow categories is introduced.} Transversality implies that the moduli space is a compact manifold with corners. 

Moreover, there are two types of symmetries among these moduli spaces. First, ${\mb Z}$ acts on $\wh S^\infty$ by the translation
\beqn
a h = a((h_i)_{i \in {\mb Z}}) = (h_{i-a})_{i\in {\mb Z}},\ a \in {\mb Z},\ h \in \wh S^\infty.
\eeqn
This action preserves the gradient flow of $f_{\rm Tate}$ with 
\beqn
f_{\rm Tate}(ah) = f_{\rm Tate}(h) - a.
\eeqn
This ${\mb Z}$-action is then extended to the product $X \times \wh S^\infty$ which acts trivially on the $X$-component. The maps induce homeomorphisms
\beqn
M_{av\ aw}^{\mb M} \cong M_{vw}^{\mb M}.
\eeqn
Another symmetry, which will be called an involution later, is the action
\beqn
v = (\kappa_v, \mu_v) \mapsto  v^\dagger = (\kappa_v^\dagger, \mu_v^\dagger) :=(\tau_X (\kappa_v), - \mu_v ).
\eeqn
It preserves the action value and moduli spaces:
\beq\label{Morse_involution}
M_{vw}^{\mb M} \cong M_{v^\dagger w^\dagger}^{\mb M}.
\eeq

\subsubsection{Tate-Morse complex}

One can then define a Tate-Morse complex as follows. Define
\beqn
\wt{CM}(f_X) = \Big\{ \sum_{i=1}^\infty a_i v_i\ |\ a_i \in {\mb F}_2,\ v_i \in {\rm crit} (\wt f_X),\ \lim_{i \to \infty} \wt f_X(v_i) = - \infty \Big\}.
\eeqn
We consider infinite series instead of finite sums as we would like to have a theory over the field
\beqn
{\mc K}:= {\mb F}_2 ((\theta)) = \F_2[\theta^{-1},\theta]]
\eeqn
of formal Laurent series in $\theta$. Then define the differential
\beqn
\wt d: \wt{CM}(f_X) \to \wt{CM}(f_X)
\eeqn
by linearly extending over ${\mc K}$
\beqn
\wt d(v) = \sum_{w \in {\rm crit}(\wt f_X)} (\# M_{vw}^{\mb M}) w \in \wt{CM}(f_X)
\eeqn
where $\# M_{vw}^{\mb M} \in {\mb F}_2$ is the mod 2 count of the moduli space (when the moduli space has nonzero dimension, the count is defined to be zero). 

Moreover, the involution induces a ${\mb Z}/2$-action on the complex defined by 
\beqn
\sum_{i=1}^\infty a_i v_i \mapsto \sum_{i=1}^\infty a_i v_i^\dagger.
\eeqn
The homeomorphism \eqref{Morse_involution} implies that $\# M_{vw}^{\mb M} = \# M_{v^\dagger w^\dagger}^{\mb M}$. Hence this action commutes with the differential. Therefore, the invariant part
\beqn
\wh{CM}(f_X):= \wt{CM}(f_X)_{{\mb Z}/2}
\eeqn
is a subcomplex. We call this the {\bf Tate-Morse complex} associated to $f_X$. Its cohomology, denoted by 
\beqn
\wh{HM}(f_X)
\eeqn
is a finite-dimensional vector space over ${\mc K}$. 

\subsubsection{Computation}

One can prove that the Tate-Morse cohomology $\wh{HM}(f_X)$ is independent of choices following standard arguments. For proving the Arnold--Givental conjecture, we do not need to do so. Instead, we compute the rank of homology for a specific choice of involution-invariant Morse function. Choose $f_X$ such that its normal Hessian along ${\rm Fix}(\tau_X)$ is positive definite. A consequence is that a downward Morse flow line of $f_X$ cannot flow from an invariant fixed point to a non-invariant fixed point. Let
\beqn
f_L:= f_X|_L: L = {\rm Fix}(\tau_X) \to {\mb R}
\eeqn
be the restriction of $f$ along the fixed point locus. We view the fixed locus ${\rm Fix}(\tau_X)$ as having the trivial ${\mb Z}/2$-action. Then the corresponding Tate-Morse complex
\beqn
\wh{CM}(f_L) \subset \wh{CM}(f_X)
\eeqn
becomes a subcomplex. Using the Morse-theoretic localization (see for example \cite{Seidel_Smith_localization}), or more precisely, the group cohomology discussion for chain complexes with a ${\mb Z}/2$-action as in \cite[Section 2]{large2019equivariantfloertheorydouble}, one obtains that 

\beq\label{Morse_localization}
\wh{HM}(f_X) \cong \wh{HM}(f_L) \cong HM(f_L) \otimes {\mc K} \cong H^*(L; {\mc K}).
\eeq

Its rank is the lower bound appearing in the Arnold--Givental conjecture.

\subsection{Involution symmetric Hamiltonian}

\subsubsection{Hamiltonian Floer theory}

Let $(X, \omega)$ be a compact symplectic manifold. Define
\beqn
\Pi:= \omega(\pi_2(X)) \subset {\mb R}.
\eeqn
Let $H$ be a nondegenerate 1-periodic Hamiltonian. Let ${\mc O}(H)$ be the set of contractible 1-periodic orbits. In this paper, a capped 1-periodic orbit is an equivalence class of pairs $(u, \uds \gamma)$ where $\uds \gamma\in {\mc O}(H)$ and $u$ is a smooth extension to the disk; two pairs $(u, \uds\gamma)$ and $(u', \uds \gamma')$ are considered equivalent if $\uds \gamma = \uds \gamma'$ and the spherical class determined by $u$ and the reversion of $u'$ has zero symplectic area. Let $\tilde {\mc O}(H)$ be the set of capped 1-periodic orbits of $H$. Then $\Pi$ acts freely on $\tilde {\mc O}(H)$ by recapping with quotient $\tilde{\mc O}(H)/\Pi \cong {\mc O}(H)$. We use $\gamma$ to denote a general element of $\tilde {\mc O}(H)$ where the underlying loop is $\uds\gamma$.

Let $(X, \omega)$ be a compact symplectic manifold and $\tau_X: X \to X$ be an anti-symplectic involution and $L \subset X$ be the fixed point set of $\tau_X$. Following Lu \cite{Lu_2008}, we introduce the following notion.

\begin{defn}
A 1-periodic Hamiltonian $H_t: X \to {\mb R}$, $t \in {\mb R}$, is called {\bf involution symmetric} if for all $t\in {\mb R}$ and $x \in X$, 
\beqn
H_{-t} (x) = H_t(\tau_X(x)).
\eeqn
\end{defn}
For the next two lemmas, let ${\mc P}(H)\supset{\mc O}(H)$ be the set of {\em all} 1-periodic orbits of $H$.
If $H_t$ is involution symmetric, then if $x(t)$ is a solution to $\dot x(t) = X_{H_t}(x(t))$, so is $y(t) = \tau_X( x(-t))$. Therefore, one has an involution $\uds\gamma \mapsto \uds \gamma^\dagger$ on ${\mc P}(H)$ given by
\beqn
\uds\gamma^\dagger(t) = \tau_X( \uds \gamma (-t)) = \tau_X( \uds\gamma (1-t)),
\eeqn
which restricts to ${\mc O}(H)$.

\begin{lemma}\label{lemma_Lu}\cite{Lu_2008} There is a bijection
\beqn
{\mc P}(H)^{{\mb Z}/2} \to \phi(L) \cap L,\quad \ \uds\gamma \mapsto \phi(\uds\gamma(0)) = \uds\gamma( \frac{1}{2})
\eeqn
where $\phi$ is the time-$\frac{1}{2}$ map of the Hamiltonian $H_t$.    
\end{lemma}

\begin{proof}
We first consider the $\Z/2$-fixed point set. Suppose $\uds\gamma^\dagger = \uds\gamma\in\mc{P}(H)$. Then one has 
\beqn
\uds \gamma (0) = \uds\gamma^\dagger(0) = \tau_X( \uds\gamma (0)) \Longrightarrow \uds\gamma (0)  \in L.
\eeqn
Similarly $\uds\gamma(\frac{1}{2}) \in L$, so $\uds\gamma(\frac{1}{2})\in\phi(L) \cap L$ because $\uds\gamma(\frac{1}{2}) = \phi(\uds\gamma (0))$. Conversely, given $x_1=\phi(x_0)\in\phi(L)\cap L$, there is a unique solution $x:[0,1/2]\to X$ to $\dot{x}(t)=X_{H_t}(x(t))$ with $x(0)=x_0$ and $x(1/2)=x_1$. Define $y:[1/2,1]\to X$ by $y(t)=\tau_Xx(1-t)$, then $y(t)$ also solves $\dot{y}(t)=X_{H_t}(y(t))$ and $y(1/2)=x(1/2)$, $y(1)=x(0)$. So $x*y$ is a 1-periodic orbit of $H$ which is clearly $\Z/2$-fixed.
\end{proof}

By Lemma \ref{lemma_Lu}, there is an injection ${\mc O}(H)^{\Z/2}\hookrightarrow \phi(L)\cap L$. The next Lemma implies that without loss of generality, to prove the Arnold--Givental conjecture, we can assume that the involution symmetric Hamiltonian is nondegenerate.

\begin{lemma}\label{lemma:symmetric regularization}
Let $\phi$ be a Hamiltonian diffeomorphism on $X$ such that $\phi(L)$ and $L$ intersect transversely. Then there exists a nondegenerate 1-periodic involution-symmetric Hamiltonian $H_t$ with time-$\frac{1}{2}$ map $\phi'$ such that 
\beqn
 \phi(L) \cap L =  \phi'(L) \cap L .
\eeqn
\end{lemma}

\begin{proof}
By \cite{Lu_2008} there exists a $\tau_X$-symmetric Hamiltonian $G_t$ with $\phi_{G}^{1/2}=\phi$ (see also the beginning of Section \ref{section:overall strategy}). However, such a $G_t$ may not be nondegenerate. We can perturb $G_t$ within $\tau_X$-symmetric Hamiltonians while keeping the set $\mc{P}(G)^{\Z/2}\cong\phi_G^{1/2}(L)\cap L$ unchanged. 
We begin by identifying the nondegeneracy condition for $\Z/2$-fixed 1-periodic orbits. Let $\uds\gamma\in{\mc P}(G)^{{\mb Z}/2}$ and denote $x_i=\uds\gamma(i/2)$ for $i=0,1$. Denote the $\pm1$-eigenspaces of $d\tau_X|_{x_i}$ by $E_i^{\pm}$. Clearly $E_i^+=T_{x_i}L$ and $\dim E_i^{\pm}=\frac{1}{2}\dim X$. Write the tangent map $d\phi|_{x_0}:T_{x_0}X\to T_{x_1}X$ as 
\[
d\phi|_{x_0}=\begin{pmatrix}
    a & b \\ c & d
\end{pmatrix}: E_0^+\oplus E_0^- \to E_1^+ \oplus E_1^-.
\]
Involution symmetry of $G_t$ implies 
$\phi_G^1=\tau_X\circ\phi^{-1}\circ\tau_X\circ\phi$.
Differentiation gives
\[
d\phi_G^1|_{x_0}=d\tau_X|_{x_0}\circ (d\phi|_{x_0})^{-1}\circ d\tau_X|_{x_1}\circ d\phi|_{x_0}.
\]
Now, writing any $v\in T_{x_0}X$ as $v=v_++v_-$ with $v_\pm\in E_0^{\pm}$, it is straightforward to see 
\[
d\phi_G^1|_{x_0}(v)=v \quad \iff \quad bv_-=0, \quad cv_+=0.
\]
Consequently, 
\[\ker(d\phi_G^1|_{x_0}-1)=\ker c\oplus \ker b\subset E_0^+\oplus E_0^-.\]
The transverse intersection condition $d\phi|_{x_0}(E_0^+)+E_1^+=T_{x_1}X$ translates to $c:E_0^+\to E_1^-$ being surjective, or equivalently $\ker c=0$. Hence we need $\ker b=0$ in addition for $\uds\gamma$ to be nondegenerate.
Now we write ${\mc P}(G)^{{\mb Z}/2}=\{\uds\gamma_1,\dots,\uds\gamma_k\}$ and perturb $G$ in two steps. 

\emph{Step 1.} We perturb $G$ to $G'$ near $\uds\gamma_1,\dots,\uds\gamma_k$. The graphs of $\uds\gamma_1,\dots,\uds\gamma_k$ in $\R/\Z\times X$ are disjoint. For each $\uds\gamma_j$ we consider a candidate Hamiltonian $Q^j_t$ supported near the graph of $\uds\gamma_j$, satisfying 
\[Q^j_t(\uds\gamma_j(t))=0, \quad dQ^j_t(\uds\gamma_j(t))=0,\]
and with disjoint supports. The Hamiltonian $Q^j_t$ can be chosen $\tau_X$-symmetric by first constructing it inside a closed sub-interval of $(0,1/2)$ and then extending it by involution symmetry.
Then $\uds\gamma_1,\dots,\uds\gamma_k$ are also 1-periodic orbits of 
\[G'=G+\sum_j Q^j,\]
so $\phi(L)\cap L\subset\phi_{G'}^{1/2}(L)\cap L$.
For generic $Q^1,\dots,Q^k$, the $b$-component of $d\phi_{G'}^{\frac{1}{2}}$ at $\uds\gamma_j(0)$ is invertible for each $j$. For sufficiently $C^2$-small $Q^1,\dots, Q^j$, $\phi_{G'}^{1/2}$ is $C^1$-close to $\phi_G^{1/2}$, so the $c$-component of $d\phi_{G'}^{\frac{1}{2}}$ at $\uds\gamma_j(0)$ is invertible for each $j$ since originally $\phi(L)\pitchfork L$. Thus there are disjoint small balls $p_j\in B_j\subset L$ such that $\phi_{G'}^{1/2}(x)\in L$ ($x\in B_j$) if and only if $x=p_j$. Moreover, for $C= L\setminus \cup_{j} B_j$, we can assume ${\rm dist}(\phi_{G'}^{1/2}(C),L)>0$ since ${\rm dist}(\phi(C),L)>0$. Thus, 
\[
\phi_{G'}^{1/2}(x)\in L \quad (x\in L)\quad \iff \quad x\in\{\uds\gamma_1(0),\dots,\uds\gamma_k(0)\}.
\]
In summary, for $Q^1,\dots,Q^k$ satisfying the above assumptions, we have 
\[
\mc{P}(G')^{\Z/2}=\mc{P}(G)^{\Z/2}=\{\uds \gamma_1,\dots,\uds\gamma_k\}
\]
and $\uds \gamma_1,\dots,\uds\gamma_k$ are non-degenerate for $G'$. 

\emph{Step 2.} We perturb $G'$ to $H$ away from $\uds\gamma_1,\dots,\uds\gamma_k$.
Since $\uds \gamma_1,\dots,\uds\gamma_k$ are non-degenerate for $G'$, there are disjoint neighborhoods $\uds\gamma_j(0)\in U_j\subset X$ such that $\uds\gamma_j(0)$ is the only fixed point of $\phi_{G'}^1$ lying in $U_j$. Choose disjoint $\tau_X$-invariant neighborhoods $W_j$ of ${\rm graph}(\uds\gamma_j)$ in $\R/\Z\times X$ satisfying $W_j\cap(\{0\}\times X)\subset U_j$. Consider a candidate $\tau_X$-symmetric Hamiltonian $K$ which vanishes on $W=\cup_j W_j$, and put 
\[
H=G'+K.
\]
Then $\uds \gamma_1,\dots,\uds\gamma_k\in\mc{P}(H)^{\Z/2}$. 
We claim that if $K$ is sufficiently $C^1$-small, then $H$ has no more $\Z/2$-fixed 1-periodic orbits. Otherwise, there is a sequence $K_\nu$ of $\tau_X$-symmetric Hamiltonians and a sequence $\uds\eta_\nu\in\mc{P}(G'+K_\nu)^{\Z/2}$ such that
\[
K_\nu|_W = 0, \quad \|K_\nu\|_{C^1}\to 0\ (\nu\to\infty), \quad \uds\eta_\nu\notin\{\uds\gamma_1,\dots,\uds\gamma_k\}.
\]
Then $X_{G'+K_\nu}\xrightarrow{C^0} X_{G'}$, so $\uds\eta_\nu$ converges to some $\uds\gamma_j\in\mc{P}(G')$. 
For sufficiently large $\nu_0$, ${\rm graph}({\uds\eta_{\nu_0}})\subset W_j$, and since $K_{\nu_0}|_W=0$, ${\uds\eta_{\nu_0}}\in\mc{P}(G')$. Then since $\uds\eta_{\nu_0}(0)\in U_j$ we have $\uds\eta_{\nu_0}=\uds\gamma_j$, a contradiction. Therefore, for generic $K$ satisfying the above assumptions and with sufficiently small $C^1$-norm, we conclude that all 1-periodic orbits of $H=G'+K$ are nondegenerate and $\mc{P}(H)^{\Z/2}=\mc{P}(G')^{\Z/2}$. By Lemma \ref{lemma_Lu}, the proof is complete.
\end{proof}
\begin{rem}
    There are two conceptual understandings of Lemma \ref{lemma:symmetric regularization}. First, it is an infinite dimensional analogue of a relative $\Z/2$-equivariant Morse lemma. The condition $\phi(L)\pitchfork L$ says the Hamiltonian action functional (or 1-form) is already Morse on the fixed locus of loop space. The lemma makes it Morse on the whole loop space without changing its fixed critical points. Second, it is about $\Z/2$-equivariant transversality: consider the $\Z/2$-equivariant section $x\mapsto \dot{x}-X_{H_t}(x(t))$ on loop space, which is transverse to the zero-section on the fixed locus. The lemma perturbs it equivariantly to achieve transversality everywhere while preserving its fixed zeros.
\end{rem}

One can further extend the involution on ${\mc O}(H)$ to $\tilde {\mc O}(H)$. Suppose $\gamma \in \tilde {\mc O}(H)$ is represented by $(u, \uds\gamma)$ where $u: {\mb D}^2 \to X$ is a continuous map. Let $\tau_{{\mb D}^2}: {\mb D}^2 \to {\mb D}^2$ be the complex conjugation (whose boundary restriction is just the map $t \mapsto -t$ on the circle). Then define 
\beqn
u^\dagger:= \tau_X \circ u \circ \tau_{{\mb D}^2}.
\eeqn
Then define $\gamma^\dagger$ to be the capped 1-periodic orbit represented by $(u^\dagger, \uds\gamma^\dagger)$. Notice that this action preserves the symplectic action:
\beqn
{\mc A}_H(\gamma) = {\mc A}_H(\gamma^\dagger).
\eeqn
and commutes with the $\Pi$-action:
\beqn
(\lambda \gamma)^\dagger = \lambda \gamma^\dagger,\ \lambda \in \Pi,\ \gamma \in \tilde {\mc O}(H).
\eeqn

\subsubsection{Involution symmetric almost complex structure}

To incorporate the involution into Floer theory, one needs to have involutions on moduli spaces of Floer cylinders. In other words, the Floer equation has to respect the combination of domain involution and target involution. This requires the almost complex structure to respect them in the sense that the differential of the symplectic involution is complex-conjugate linear with respect to the almost complex structure. Therefore, we use the following statement. 

\begin{lemma}\cite[Lemma 38.3]{FOOO_Chap82009ref}\cite[Corollary 1.18]{Reiser_2014}\footnote{This paper was published with a misspelling of the author's last name. The correct spelling is Rieser.}
There exists an $\omega$-compatible almost complex structure $J$ such that $d\tau_X \circ J = - J \circ d\tau_X$.
\end{lemma}

\begin{rem}
We present a simple alternative argument proving this fact. Let $T_*: \cl J \to \cl J$ be the map taking $J \in \cl J$ to $d \tau_X \circ J \circ d\tau_X^{-1}.$ It is an involution on the space $\cl J$ of $\om$-compatible almost complex structures. It is easy to see that if $J_0, J_1$ are two $\om$-compatible almost complex structures, and $[J_0,J_1]$ is the path $[J_0,J_1]: [0,1] \to \cl J$ in $\cl J$ corresponding to fiberwise geodesics from $(J_0)_x$ to $(J_1)_x$ over each $x \in X$ (see \cite{Reznikov, S-action} for different uses of this structure) then $T_*([J_0,J_1]) = [T_*(J_0), T_*(J_1)].$ Now setting $J_1 = T_*J_0$ for an arbitrary $J_0 \in \cl J,$ our desired $J$ is defined as $J= [J_0,J_1](\frac{1}{2}).$ Indeed $T_*J = [J_1,J_0](\frac{1}{2}) = [J_0,J_1](\frac{1}{2}) = J.$
\end{rem}

Fix such a $J$ from now on. It allows us to define an involution on the moduli space of holomorphic spheres. Let $S^2 = \mb{CP}^1$ be the standard Riemann sphere, which has a standard anti-holomorphic involution
\beqn
\tau_{S^2}: \mb{CP}^1 \to \mb{CP}^1,\ [z_0: z_1] \mapsto [\bar z_0: \bar z_1].
\eeqn
Notice that $\tau_{S^2}$ induces an automorphism on $PSL(2; {\mb C})$ as
\beqn
PSL(2; {\mb C}) \ni g \mapsto g^\dagger:= \tau_{S^2} \circ g \circ \tau_{S^2}.
\eeqn
More explicitly, if $g([z_0: z_1]) = [az_0 + bz_1: cz_0 + dz_1]$, then 
\beqn
g^\dagger ([z_0: z_1]) = [\bar a z_0 + \bar b z_1: \bar c z_0 + \bar d z_1].
\eeqn
Then one has the basic fact, which follows from a direct computation.
\begin{lemma}
If $u: \mb{CP}^1 \to X$ is a $J$-holomorphic map, so is 
\beqn
u^\dagger: \mb{CP}^1 \to X,\ z\mapsto \tau_X(u(\tau_{S^2}(z))).
\eeqn
Moreover, for each $g \in PSL(2; {\mb C})$, one has 
\beqn
(u\circ g)^\dagger = u^\dagger \circ g^\dagger. \qed
\eeqn
\end{lemma}

\subsubsection{Involutions of stable Floer cylinders}

We consider Floer moduli spaces ``coupled'' with the Morse theory of the function $f_{\rm Tate}: \wh S^\infty \to {\mb R}$, similar to the case of Tate--Morse complex. It is not that straightforward to define involutions on moduli spaces as domain reparametrizations are involved. We need certain notational preparations.

\begin{defn}\label{defn_cylinder}
A {\bf prestable cylinder} is a smooth or nodal rational curve $\Sigma$ with two marked points $z_-$, $z_+$ together with cylindrical coordinate $(s, t)$ on each irreducible component (called a cylindrical component of $\Sigma$) lying between $z_-$ and $z_+$. Given two prestable cylinders $\Sigma$, $\Sigma'$, an isomorphism $\varphi: \Sigma \to \Sigma'$ is an isomorphism of marked rational curves such that on each cylindrical component $\Sigma_i \subset \Sigma$, the isomorphism with the corresponding cylindrical component of $\Sigma'$, written in cylindrical coordinates, is a translation in $s$-direction.
\end{defn}

\begin{defn}\label{defn_cylinder_conjugate}
Let $\Sigma$ be a prestable cylinder. Its {\bf conjugate} $\Sigma^\dagger$ is the prestable cylinder with the same underlying nodal surface as $\Sigma$, where the complex structure on each component is flipped (i.e., if we denote by $j$ the complex structure on an irreducible component, flipping means taking $-j$) and the cylindrical coordinate on each cylindrical component is composed with the flipping anti-holomorphic involution $(s, t) \mapsto (s, -t)$ on the standard cylinder. 
\end{defn}

\begin{defn}
Given $\gamma_p, \gamma_q \in \tilde {\mc O}(H)$, a {\bf stable Floer trajectory} from $\gamma_p$ to $\gamma_q$ consists of 
\begin{enumerate}
\item A prestable cylinder $\Sigma$.

\item For each cylindrical component $\Sigma_i \subset \Sigma$, a solution $u_i: \Sigma_i \to X$ to the Floer equation
\beqn
\partial_s u_i + J(u_i)(\partial_t u_i - X_{H}(u_i)) = 0
\eeqn
with respect to the cylindrical coordinate $(s, t)$. 

\item For each spherical component $\Sigma_\alpha \subset \Sigma$, a $J$-holomorphic map $u_\alpha:\Sigma_\alpha \to X$.
\end{enumerate}
They need to satisfy the following conditions.
\begin{enumerate}
    \item The matching conditions at nodal points and asymptotic circles.

    \item The homotopy type of the map is specified by the difference of the capped orbits $\gamma_p$ and $\gamma_q$.

    \item Two stable Floer trajectories $(\Sigma, u_i, u_\alpha)$ and $(\Sigma', u_i', u_\alpha')$ are {\bf equivalent} if there exists an isomorphism $\varphi: \Sigma \to \Sigma'$ (which sends $\Sigma_i$ to $\Sigma_i'$ and $\Sigma_\alpha$ to $\Sigma_\alpha'$) such that 
    \begin{align*}
    &\ u_i = u_i'\circ \varphi_i,\ &\ u_\alpha = u_\alpha' \circ \varphi_\alpha.
    \end{align*}
    We require that the group of self-equivalences is finite, which is what we mean by stability.\footnote{When $p=q$, the constant trajectory is regarded as stable.}
\end{enumerate}
\end{defn}

\begin{defn}
Let $u: {\mb R}\times S^1 \to X$ be a smooth map. Its {\bf conjugate} $u^\dagger$ is the map
    \beqn
    u^\dagger(s, t) = \tau_X( u(s, -t)).
    \eeqn
\end{defn}

\begin{lemma}
If $u$ is a Floer cylinder, so is $u^\dagger$.
\end{lemma}
\begin{proof}
This is a straightforward consequence of the invariance assumption on $J$ and $H$.
\end{proof}

With the above preparations, we can extend the involution on the moduli space of smooth Floer trajectories to all stable Floer trajectories.

\begin{defn}
Let $(\Sigma, u_i, u_\alpha)$ be a stable Floer cylinder. Its {\bf conjugate} is the triple $(\Sigma^\dagger, u_i^\dagger, u_\alpha^\dagger)$. 
\end{defn}

\subsubsection{Coupling with the Tate flow}

Recall that one has a Morse function $f_{\rm Tate}$ on $\wh S^\infty$ whose critical points are generally denoted by $\mu$. 

\begin{defn}
Given 
\beqn
p = (\gamma_p, \mu_p),\ q = (\gamma_q, \mu_q) \in \tilde {\mc O}(H) \times {\rm crit} f_{\rm Tate}
\eeqn
a {\bf stable Tate-Floer trajectory} from $p$ to $q$ is a quadruple $(\Sigma, u_i, u_\alpha, y_i)$ where $\Sigma$ is a prestable cylinder, $u_i: \Sigma_i \to X$ is a Floer trajectory, $u_\alpha: \Sigma_\alpha \to X$ is a $J$-holomorphic map, and $y_i: {\mb R} \to \wh S^\infty$ is a Tate flow trajectory. They are required to satisfy the following conditions.
\begin{enumerate}
    \item The matching condition at nodal points and asymptotic circles so that $(\Sigma, u_i, u_\alpha)$ is a stable Floer cylinder with asymptotics $\gamma_p$ and $\gamma_q$.

    \item The matching condition between the end points of $y_i$ and $y_{i+1}$ so that $\{y_i\}$ forms a broken Tate flow trajectory connecting $\mu_p$ and $\mu_q$.
\end{enumerate}
Two stable Tate--Floer trajectories $(\Sigma, u_i, u_\alpha, y_i)$ and $(\Sigma', u_i', u_\alpha', y_i')$ are {\bf equivalent} if there is an isomorphism $\varphi: \Sigma \to \Sigma'$ of prestable cylinders (which determines a translation by $\uds \varphi_i: {\mb R} \to {\mb R}$) such that 
\beqn
u_i = u_i' \circ \varphi_i,\ u_\alpha = u_\alpha' \circ \varphi_\alpha,\ y_i = y_i'\circ \uds\varphi_i.
\eeqn
\end{defn}

Let $M_{pq}^{\mb F}$ be the set of all equivalence classes of stable Tate--Floer trajectories from $p$ to $q$.\footnote{Here the superscript ${\mb F}$ stands for ``Floer.'' Later it will also be the notation for a general flow category.} These moduli spaces admit the following kind of symmetry. First, one has the free action by the group
\beqn
\Pi =  \omega(\pi_2(X)) \subset {\mb R}
\eeqn
via ``recapping,'' which changes the symplectic action by 
\beqn
{\mc A}_H(\lambda \gamma_p) = {\mc A}_H(\gamma_p) - \lambda,\ \forall \lambda \in \Pi,\ p \in \tilde {\mc O}(H).
\eeqn
Second, one has the free action by ${\mb Z}$ on $\wh S^\infty$, which shifts $f_{\rm Tate}$ by corresponding integers. Lastly, one has a free involution on the product $\tilde {\mc O}(H) \times {\rm crit} (f_{\rm Tate})$ by
\beqn
p = (\gamma_p, \mu_p) \mapsto p^\dagger = (\gamma_p^\dagger, \mu_p^\dagger) = (\gamma_p^\dagger, -\mu_p ).
\eeqn
Notice that one has an induced canonical homeomorphism between moduli spaces
\beq\label{Tate-Floer_involution}
M_{p^\dagger q^\dagger}^{\mb F}  \cong M_{pq}^{\mb F}.
\eeq

At this moment, we do not have a Tate--Floer complex as in the Morse case because we do not have transversality: the moduli spaces $M_{pq}^{\mb F}$ are only topological spaces in general.

\subsection{Involutive Tate-PSS and Tate-SSP moduli spaces}\label{subsection_PSS_moduli}

Next, we describe the moduli spaces that permit us to compare the Tate--Morse theory and the Tate--Floer theory. They are packaged under the notion of flow bimodules, for which we refer the reader to Section \ref{sec:abstract} for a precise discussion. We only describe the case of Tate--PSS bimodule. The case of SSP bimodule is completely symmetric. 

Let $H$ be an involution symmetric Hamiltonian and let $J$ be an anti-symmetric almost complex structure as above. Define a Hamiltonian connection on the trivial bundle ${\mb R}\times S^1 \times X$
\beqn
\sigma^{\rm PSS} = d + \chi^{\rm PSS} H_t dt
\eeqn
where $\chi^{\rm PSS}(s)$ is a cut-off function on the cylinder which only depends on $s$, such that $\chi^{\rm PSS}(s) = 0$ when $s \ll -1$ and $\chi^{\rm PSS}(s) = 1$ when $s \gg 1$. Then for each smooth map $u: {\mb R} \times S^1 \to X$, one has the covariant derivative
\beqn
\nabla^{\sigma^{\rm PSS}} u  \in \Gamma( {\mb R}\times S^1, \Lambda^1 \otimes u^* TX).
\eeqn
Using $J$ on $X$, the domain almost complex structure, and the $\tau_X$-invariant Morse function $f_X$, one can consider the following equation for pairs $(x, u)$ where 
\beqn
x: (-\infty, 0] \to X,\ x'(s) + \nabla f_X(x(s)) = 0
\eeqn
and 
\beqn
u: {\mb R}\times S^1 \to X,\ \left( \nabla^{\sigma^{\rm PSS}} u \right)^{0,1} = 0.
\eeqn
They need to satisfy the matching condition
\beqn
x(0) = u(-\infty) \in X,
\eeqn
which is well-defined by removal of singularity of $u$ at $\infty$: any such finite energy solution converges to some $\kappa = x(-\infty) \in {\rm crit} f_X$ and some 1-periodic orbit of $H$ at $+\infty$ where the map $u$ automatically induces a capping. Hence one can use a pair $(\kappa; \gamma)$ to label such solutions. For such a solution $(x, u)$, we simply say that it converges to $\gamma$ at $+\infty$.

Similar to the previous situations, we couple the above parametrized Floer equations with parametrized flow lines in $\wh S^\infty$. Consider $v = (\kappa_v, \mu_v) \in {\rm crit} (\wt f_X)$ and $p = (\gamma_p, \mu_p) \in \tilde {\mc O}(H) \times {\rm crit}(\wt f_{\rm Tate})$. We consider the coupled equation for triples $(y, x, u)$ where $x$ and $u$ are as above and $y: {\mb R}\to \wh S^\infty$ is a negative Tate gradient flow line, such that $(y, x)$ converges to $v$ at $-\infty$ and $(y, u)$ converges to $p$ at $+\infty$.

We compactify the moduli space of such solutions by allowing bubbling and breakings. Notice that bubbling of the map $u$ at $-\infty$ is treated as spherical bubbling, a codimension-two phenomenon. Hence one obtains a compactified moduli space
\beqn
M_{v; p}^{\rm PSS}
\eeqn
which is stratified by the poset
\beqn
A_{v;p}^{\rm PSS} = \Big\{ (v w_1\cdots w_l; r_k r_{k-1} \cdots r_1 p)\ \Big\},
\eeqn
whose elements are labeled by the intermediate asymptotic critical points and the partial order is induced from expansion of words. Notice that there are subspaces of virtual codimension one identified as
\begin{align*}
&\ M_{vw; p}^{\rm PSS} \cong M_{vw}^{\mb M} \times M_{w;p}^{\rm PSS},\ &\ M_{v;qp}^{\rm PSS}\cong M_{ v; q}^{\rm PSS}\times M_{qp}^{\mb F}.
\end{align*}

These PSS moduli spaces also admit similar symmetries. For each element $\lambda \in {\mb Z}$, one has a canonical homeomorphism
\beqn
M_{\lambda v; \lambda p}^{\rm PSS} \cong M_{ v;  p}^{\rm PSS}
\eeqn
Moreover, regarding involution, one has a homeomorphism
\beqn
M_{v^\dagger; p^\dagger}^{\rm PSS} \cong M_{v; p}^{\rm PSS}.
\eeqn

To define chain level objects over the Novikov field, we apply some naive extension of the above notations. We consider the disjoint union of $\Pi$-worth copies of the Morse theory on $X \times \wh S^\infty$ and denote critical points by 
\beqn
v = (\kappa_v,\mu_v, a_v),\ \kappa_v \in {\rm crit} (f_X), \mu_v \in {\rm crit}(f_{\rm Tate}), a_v \in \Pi.
\eeqn
Then the group $\Pi \times {\mb Z}$ acts on those triples. Then the previously defined Tate moduli spaces admit the naive extensions such that for all $\lambda \in \Pi \times {\mb Z}$,  
\beqn
M_{\lambda v\ \lambda w}^{\mb M} = M_{vw}^{\mb M}.
\eeqn
Then the PSS moduli spaces can be extended to pairs $(v,p)$ so that for all $\lambda \in \Pi$, one has 
\beqn
M_{\lambda v; \lambda p}^{\rm PSS} \cong M_{v ; p}^{\rm PSS}.
\eeqn

One can then describe the Tate--SSP moduli spaces in a similar fashion and we use the notation 
\beqn
M_{p; w}^{\rm SSP}
\eeqn
to denote the resulting moduli space.

\subsection{The involutive concatenation and homotopy}

One uses a 1-parameter family of Floer data on the cylinder as follows. Consider a 1-parameter family of Hamiltonian connections
\beqn
\sigma_\zeta = d + H_{\zeta, s, t} dt,\ \zeta \in [0, +\infty)
\eeqn
satisfying the following conditions.
\begin{enumerate}
    \item When $\zeta = 0$, $H_{\zeta, s, t} \equiv 0$.
    
    \item When $\zeta \to +\infty$, $H_{\zeta, s + \zeta, t}$ converges to $\chi^{\rm PSS}(s) H_t$ over compact subsets, $H_{\zeta, s - \zeta, t}$ converges to $\chi^{\rm SSP}(s) H_t$ over compact subsets, where $\chi^{\rm SSP}(s)$ is a cut-off function on the cylinder which only depends on $s$, such that $\chi^{\rm SSP}(s) = 1$ when $s \ll -1$ and $\chi^{\rm SSP}(s) = 0$ when $s \gg 1$, and 
    \beqn
    s \in [-\zeta, \zeta]\Longrightarrow H_{\zeta, s, t} = H_t.
    \eeqn
    
    \item $H_{\zeta, s, -t}(x) = H_{\zeta, s, t}(\tau_X(x))$. 
\end{enumerate}
One can then consider the equation for quadruples $(\zeta, x_-, u, x_+)$ where $\zeta \in [0, +\infty)$, $x_-: (-\infty, 0] \to X$, $x_+: [0, +\infty) \to X$, and $u: {\mb R}\times S^1 \to X$ are smooth maps solving
\beqn
x_\pm'(s) + \nabla f_X ( x_\pm (s)) = 0,
\eeqn
\beqn
\Big( \nabla^{\sigma_\zeta} u\Big)^{0,1} = 0,
\eeqn
subject to the matching condition
\beqn
x_-(0) = u(-\infty),\ x_+(0) = u(+\infty).
\eeqn
For each finite energy solution, the map $u$ represents a class in $\Pi$ and $x_\pm$ are asymptotic to critical points of $f$. Moreover, one can couple the equation with parametrized negative gradient flow equation in $\wh S^\infty$. Then solutions can be labelled by asymptotic limits $v,w \in {\rm crit} (\wt f_X)$ and a class $a \in \Pi$. Moreover, one can compactify the moduli space of such solutions by allowing 1) sphere bubbling, 2) breakings at $\pm\infty$, and 3) $\zeta \to +\infty$. In the last case, the component $u$ will necessarily break along 1-periodic orbits of $H$. 

Then for $v, w \in {\rm crit}(\wt f_X) \times \Pi$, denote such a moduli space by 
\beqn
M_{v; w}^{\rm hmtp}.
\eeqn
We summarize the key features. First, there are the following codimension one strata. When $\zeta = 0$, we denote the resulting moduli space by
\beqn
M_{v;w}^{\rm pearly} \subset M_{v;w}^{\rm hmtp}
\eeqn
which consists of equivalence classes of solutions at $\zeta = 0$. It contains both parametrized Morse trajectories as well as trajectories with spheres inserted. When $\zeta = +\infty$, the boundary is the (likely non-disjoint) union
\beqn
\bigcup_{p} M_{v; p}^{\rm PSS}\times M_{p; w}^{\rm SSP}.
\eeqn
There are also two other virtual codimension-one boundaries corresponding to Morse breakings
\begin{align*}
&\ \bigcup_{z} M_{vz}^{\mb M}\times M_{z; w}^{\rm hmtp},\ &\ \bigcup_{z} M_{v; z}^{\rm hmtp} \times M_{zw}^{\mb M}.
\end{align*}

The $\Pi \times {\mb Z}$ symmetry and involution are similar to previous cases. As before, we will use ${ }^\dagger$ to denote the involution.

\begin{rem}
The name ``pearly'' used here has the same feature as that in \cite{Bai_Xu_Arnold}. However, it is different from the ``pearly'' used in \cite{Bai_Xu_foundation}. In \cite{Bai_Xu_foundation}, one takes a different compactification, which contains configurations with possibly several Morse gradient segments separated by cylindrical bubbles. 
\end{rem}

\section{Proof of Arnold--Givental Conjecture Assuming Transversality}\label{sec:proof}

For the reader's convenience, using the moduli spaces defined in Section \ref{sec:moduli}, we provide an argument for the case when all moduli spaces are cut out transversely emphasizing the algebraic aspects.

We note that grading does not matter in our discussion, so we will omit it from now on.

\subsection{Novikov rings and fields}
In Section \ref{sec:intro}, we provide a sketch of proof by working over the universal Novikov field $\Lambda_{\cl K}$ where $\cl K = \F_2[\theta^{-1},\theta]]$. For the actual proof, however, we shall define all chain level objects over a smaller Novikov field. Define

\begin{defn}[The appropriate Novikov field]
Let $\Lambda$ be the set consisting of formal series 
\beqn
\sum_{i=1}^\infty a_i \theta^{\mu_i} T^{\lambda_i}
\eeqn
where $a_i \in {\mb F}_2$, $\mu_i \in {\mb Z}$, and $\lambda_i \in \Pi$, such that
\beqn
\liminf_{i \to \infty} \mu_i > -\infty,\ \liminf_{i \to \infty} \lambda_i > -\infty
\eeqn
and 
\beqn
\lim_{i \to \infty} \mu_i + \lambda_i = +\infty.
\eeqn
\end{defn}

In other words, we perform the completion combining the valuation in both $\theta$ and $T$ variables.

\subsection{The involutive Floer complex}\label{subsec:chain}

We consider the set of formal sums
\beqn
\wt{CF}(H):= \left\{ \sum_{i=1}^\infty a_i p_i\ \left|\ \begin{array}{l} a_i \in {\mb F}_2,\ p_i = (\gamma_{p_i}, \mu_{p_i}) \in \tilde {\mc O}(H)\times {\rm crit} (f_{\rm Tate}), \\ \displaystyle
\limsup_{i \to \infty} {\mc A}_H(\gamma_{p_i} ) < +\infty,\ \limsup_{i \to \infty} f_{\rm Tate}(\mu_{p_i}) < +\infty,\\
\displaystyle \liminf_{i \to \infty} {\mc A}_H (\gamma_{p_i}) + f_{\rm Tate} (\mu_{p_i}) = -\infty 
\end{array} \right. \right\}
\eeqn
Then $\wt{CF}(H)$ is a free module over $\Lambda$. 

Now we define a differential as follows. For each pair of generators $p$, $q$, the moduli space $M_{p q}^{\mb F}$ is a compact manifold with corners of varying dimensions. Define $n_{pq}^{\mb F} \in {\mb F}_2$ to be the mod 2 count of its zero-dimensional components. Then define
\beqn
\wt{d}_H ( p) = \sum_{q} n_{pq}^{\mb F} q.
\eeqn

\begin{lemma}
The above formal sum is in $\wt{CF}(H)$. 
\end{lemma}

\begin{proof}
We know that $M_{p q}^{\mb F} \neq \emptyset$ only when ${\mc A}_H (\gamma_p) \geq {\mc A}_H(\gamma_q)$ and $f_{\rm Tate}(\mu_p) \geq f_{\rm Tate}(\mu_q)$. Hence one has 
\begin{align*}
&\ \sup_{n_{pq}^{\mb F}\neq 0} {\mc A}_H (\gamma_q) < +\infty,\ &\ \sup_{n_{pq}^{\mb F} \neq 0 } f_{\rm Tate}(\mu_q) < +\infty.
\end{align*}
On the other hand, Gromov compactness ensures that ${\mc A}_H ( \gamma_q ) + f_{\rm Tate}(\mu_q)$ cannot stay bounded. 
\end{proof}

One can also see that for $\lambda \in \Pi$ and $\lambda p$ defined as $(\la\gamma_p, \mu_p),$ for $\Pi$ acts on $\tilde{\cl O}(H) \subset \tilde{\cl L} X$ by deck transformations,
\beqn
\wt d_H( \lambda p) = \lambda \wt d_H(p).
\eeqn

Hence we may define $T^\la p = (-\la)p$ and extend $\wt d_H$ linearly over $\Lambda$. By looking at 1-dimensional moduli spaces, one can see that $\wt d_H^2 = 0$. Hence one obtains a complex $(\wt{CF}(H), \wt d_H)$. 

Now we consider the involution. The involution on generators and moduli spaces induces a ${\mb Z}/2$-action via chain maps on $\wt{CF}(H)$. Hence the ${\mb Z}/2$-fixed part is a subcomplex, which we denote by
$$
\big(\wt{CF}(H)\big)_{{\mb Z}/2} .
$$

For the ease of comparison with Morse theory, we would like to switch back to the larger Novikov field $\Lambda_{\mc K}$. Define the Tate complex by 
\beqn
\wh{CF}(H):= \big(\wt{CF}(H)\big)_{{\mb Z}/2} \otimes \Lambda_{\mc K}.
\eeqn
The resulting cohomology is 
\beqn
\wh{HF}(H)
\eeqn
which is a finite-dimensional vector space over $\Lambda_{\mc K}$.

\subsection{The involutive PSS and SSP maps}\label{subsec:pss}

For each generator $v = (\kappa_v, \mu_v)$ of the Tate-Morse complex, define
\beqn
\wt\Phi^{\rm PSS}(v) = \sum_{p} n_{v; p}^{\rm PSS} p.
\eeqn
Here $n_{v; p}^{\rm PSS}$ is the mod 2 count of the moduli space $M_{v; p}^{\rm PSS}$ (defined in Subsection \ref{subsection_PSS_moduli}). For the same reason as before, this formal sum converges and hence is an element of $\wt{CF}(H)$. Then we have a chain map equivariant under the involution actions
\beqn
\wt\Phi^{\rm PSS}: \wt{CM}(f_X) \to \wt{CF}(H)
\eeqn
which extends to a $\Lambda_{\cl K}$-linear map, hence restricts to a $\Lambda_{\mc K}$-linear chain map
\beqn
\wh\Phi^{\rm PSS}: (\wh{CM}(f_X), \wh d_{f_X}) \to (\wh{CF}(H), \wh d_H).
\eeqn
where we apply base change to $\Lambda_{\cl K}$ on the left hand side. The resulting map on homology is also denoted by
\beqn
\wh\Phi^{\rm PSS}: \wh{HM}(f_X) \to \wh{HF}(H).
\eeqn

Similarly, one has the Tate-version of the SSP map, denoted by 
\beqn
\wh\Phi^{\rm SSP}: \wh{CF}(H) \to \wh{CM}(f_X)
\eeqn
with the induced map on cohomology
\beqn
\wh\Phi^{\rm SSP}: \wh{HF}(H) \to \wh{HM}(f_X).
\eeqn

\subsection{The pearly chain map}

The pearly moduli spaces described above also induce a $\Lambda_{\mc K}$-linear chain map
\beqn
\wh\Phi^{\rm pearly}: \wh{CM}(f_X) \to \wh{CM}(f_X).
\eeqn
In reality, if transversality is already assumed, then one has an additional symmetry on holomorphic spheres that has not been used; one can prove that $\wh\Phi^{\rm pearly}$ has only the classical part which is the identity. However, our transversality assumption should be taken formally; later one will see that after FOP perturbation, the chain map can have nontrivial deformation. However, in general, it is an invertible chain map.

\begin{lemma}\label{lemma:pearl}
    The chain map $\wh\Phi^{\rm pearly}$ satisfies
    $$
    \wh\Phi^{\rm pearly} = \mathrm{Id} + \text{ Terms with positive $T$- or $\theta$-valuation}.
    $$
\end{lemma}
\begin{proof}
    This is a straightforward consequence of the observation that the moduli spaces of the lowest energy are given by constant maps, whose induced count on the nonequivariant component of the Tate--Morse complex produces the identity map on cohomology.
\end{proof}

\subsection{The involutive chain homotopy}

Using counts of the homotopy moduli spaces $M_{v;w}^{\rm hmtp}$, one can define a $\Lambda_{\mc K}$-linear map
\beqn
\Theta: \wh{CM}(f_X) \to \wh{CM}(f_X).
\eeqn
By looking at the boundaries of 1-dimensional homotopy moduli spaces, one obtains that
\beq
\wh d_{f_X} \circ \Theta + \Theta \circ \wh d_{f_X} = \wh \Phi^{\rm pearly} + \wh\Phi^{\rm SSP} \circ \wh\Phi^{\rm PSS}.
\eeq
Passing to homology, one obtains
\beqn
\wh \Phi^{\rm SSP} \circ \wh \Phi^{\rm PSS} = \wh\Phi^{\rm pearly}: \wh{HM}(f_X) \to \wh{HM}(f_X).
\eeqn
As $\wh\Phi^{\rm pearly}$ is invertible, the following proposition follows. 

\begin{prop}\label{prop:lower_bound}
One has 
\beqn
{\rm dim}_{\Lambda_{\mc K}} \Big( \wh{HF}(H) \Big)  \geq {\rm dim}_{\Lambda_{\mc K}} \Big( \wh{HM}(f_X) \Big).
\eeqn
\end{prop}
\begin{proof}
    Using Lemma \ref{lemma:pearl}, we see that $\wh \Phi^{\rm PSS}$ defines an injection on cohomology. The inequality follows as we are working over a field $\Lambda_{\cl K}$.
\end{proof}

\subsection{Equivariant local Floer cohomology}\label{sec:loc-equiv}

We borrow the notation $T$ from Section \ref{sec:intro} to discuss involutions in this subsection. We follow the construction in \cite{BSWX}. Let $\ol O \subset \tilde{\mathcal{O}}(H)$ be a $\Z/2$-orbit. Define \[\widehat{CF}_{\rm loc}(H,\ol O) = \big(\F_2\langle \ol O\rangle \otimes CM(f_{\rm Tate})\big)_{\Z/2},\] where $\Z/2$ acts on the vector space $\F_2\langle \ol O\rangle$ by the linear representation induced from the $\Z/2$ action on $\ol O.$ Consequently, $\widehat{HF}_{\rm loc}(H,\ol O) = H(\widehat{CF}_{\rm loc}(H,\ol O), \hat{d}_{\rm loc})$. It is calculated as follows.

\begin{enumerate}

\item If $\ol{O} = \{\overline{y}\}$ with $T(\overline{y})=\overline{y},$ we calculate \[\widehat{CF}_{\rm loc}(H,\ol O) = \F_2\langle\overline{y}\rangle \otimes CM(f_{\rm Tate})_{\Z/2} = \cl{K}\langle\overline{y}\rangle\] and the local differential $\hat{d}_{\rm loc} = 0.$ Therefore in this case \begin{equation}\label{eq:loc_fixed}\widehat{HF}_{\rm loc}(H,\ol O) \cong \cl{K}.\end{equation}

\item If $\ol O = \{\ol{y}, T(\ol{y})\},$ $\ol{y} \neq T(\ol{y})$ is a free $\Z/2$-orbit, define \[\widehat{CF}_{\rm loc}(H,\ol O) \cong CM(f_{\rm Tate}),\] with the local differential $\hat{d}_{\rm loc} = d_{f_{\rm Tate}}.$ Therefore in this case \begin{equation}\label{eq:loc_free}\widehat{HF}_{\rm loc}(H,\ol O) = 0.\end{equation}
\end{enumerate}

We may observe that once grading is ignored all $\widehat{HF}_{\rm loc}(H,\ol O)$ for capped $\Z/2$-orbits $\ol O$ covering a given $\Z/2$-orbit $O \subset {\mc O}(H)$ are canonically isomorphic and define $\widehat{HF}_{\rm loc}(H,O)$ to be the common isomorphism class. This is the local Tate Floer cohomology appearing in the introduction. However, technically, we find a different approach more convenient.

Indeed, to consider the contribution of all capped orbits $\ol O \in \tilde{\mc O}(H)$ of a given $\Z/2$-orbit $O \subset {\mc O}(H)$ we set \[\widehat{CF}_{\rm loc}(H,O) = \widehat{CF}_{\rm loc}(H,\ol O) \otimes_{\cl K} \Lambda_{\cl K},\] for a choice $\ol O$ of a capped orbit covering $O,$ endowed with the differential $\hat{d}_{\rm loc} \otimes id_{\Lambda_{\cl K}}.$ More invariantly, this can be described by defining $\wt{CF}_{\rm loc}(H,O)$ in the same way as $\wt{CF}(H)$ but requiring that $\gamma_{p_i} \in \pi^{-1}(O)$ where $\pi: \tilde{\mc O}(H) \to {\mc O}(H)$ is the natural surjection. Then as before \[\wh{CF}_{\rm loc}(H,O) = \big(\wt{CF}_{\rm loc}(H,O)\big)_{\Z/2} \otimes \Lambda_{\cl K}.\]

Evidently, this is a free $\Lambda_{\cl K}$-module and \begin{equation}\label{eq:complex_split}\wh{CF}(H) = \bigoplus_{O \in \cl O(H) /(\Z/2)} \wh{CF}_{\rm loc}(H,O).\end{equation} 
Write $\hat{d}_{\rm loc}: \wh{CF}(H) \to \wh{CF}(H)$ for the direct sum of the local differentials on the local complexes $\wh{CF}_{\rm loc}(H,O).$

In the same way as \cite[Lemma 27.6]{Bai_Xu_foundation} (see also \cite[Appendix B]{BSWX}) one sees that the following holds true:

\begin{lemma}\label{lma:d-eq-split}
The equivariant differential $\wh d_H$ on $\wh{CF}(H)$ splits as \[\wh d_H = \wh d_{\rm loc} + D_{\Z/2}\] where $D_{\Z/2}$ is a $\Lambda_{\cl K}$-linear map such that there exists $\delta_H > 0$ such that $D_{\Z/2}$ increases the filtration by at least $\delta_H.$ 
\end{lemma}

\begin{proof}
The differential $\wh d_H$ comes from contributions from nonempty moduli spaces $M_{pq}^{\mb F}$. Notice that if $\gamma_p = \gamma_q^\dagger$, there are no nonconstant Floer trajectories between them as they have equal symplectic action. These contributions are included in the local differential $\wh d_{\rm loc}$. On the other hand, by Gromov compactness, other moduli spaces all have symplectic energy no less than a constant $\delta_H>0$. Hence $D_{{\mb Z}/2}:= \wh d_H - \wh d_{\rm loc}$ increases the energy filtration by at least $\delta_H$. 
\end{proof}

\subsection{Homological perturbation}\label{sec:homological-perturb}

As in the proofs of \cite[Lemmas 27.2 and 27.7]{Bai_Xu_foundation}, Lemma \ref{lma:d-eq-split} together with the homological perturbation procedure implies the following fact.

\begin{prop}\label{prop:perturbed}
There is a $\Lambda_{\cl K}$-linear differential $\wh d^{\rm red}$ on the free $\Lambda_{\cl K}$-module 
\beqn
\wh{CF}{}^{\rm red}(H) =: \bigoplus_{O \in \cl{O}(H)/(\Z/2)} \wh{HF}_{\rm loc}(H,O) 
\eeqn
with 
\beqn
H(\wh{CF}^{\rm red}(H),\wh d^{\rm red}) \cong \wh{HF} (H).
\eeqn
\end{prop}

\begin{proof}
The main ingredient of the proof is the homological perturbation lemma \cite{Markl_perturbation, Huebschman_Kadeishvili}. We first observe that there exist natural $\Lambda_{\cl K}$-linear inclusion and projection operators: \[\pi: \wh{CF}(H) \to \wh{CF}^{\rm red}(H),\] 
\[\sigma: \wh{CF}^{\rm red}(H) \to \wh{CF}(H),\] which do not decrease the filtration and such that \[ \pi \circ \sigma = id_{\wh{CF}^{\rm red}(H)}\] \[ \sigma \circ \pi - id_{\wh{CF}(H)} = \hat{d} \Theta + \Theta \hat{d}\] for a filtration non-decreasing chain homotopy \[\Theta: \wh{CF}(H) \to \wh{CF}(H).\]

The explicit formula from \cite{Markl_perturbation, Huebschman_Kadeishvili}, which converges due to Lemma \ref{lma:d-eq-split} and the non-Archimedean completeness of $\Lambda_{\cl K},$ provides one with a differential \[\hat{d}^{\rm red}: \wh{CF}^{\rm red}(H) \to \wh{CF}^{\rm red}(H)\] and filtration non-decreasing chain homotopy equivalences of complexes \[\pi': \wh{CF}(H) \to \wh{CF}^{\rm red}(H),\] 
\[\sigma': \wh{CF}^{\rm red}(H) \to \wh{CF}(H),\] where $\wh{CF}(H)$ is endowed with the differential $\hat{d}$ and $\wh{CF}^{\rm red}(H)$ is endowed with the differential $\hat{d}^{\rm red}.$ In particular,
\[H(\wh{CF}^{\rm red}(H),\wh d^{\rm red}) \cong \wh{HF} (H)\] as desired.

\end{proof}
\subsection{Concluding the proof}\label{subsec:conclude}

Note that, by \eqref{eq:loc_fixed}, \eqref{eq:loc_free}, Proposition \ref{prop:perturbed}, and Lemma \ref{lemma_Lu}, \[\dim_{\Lambda_{\cl K}} \wh{HF}(H) \leq \dim_{\Lambda_{\cl K}} \wh{CF}^{\rm red}(H) = \#(\cl O(H)^{\Z/2}) \leq \#(\phi(L) \cap L).\] On the other hand Proposition \ref{prop:lower_bound} and Equation \eqref{Morse_localization} imply 
\begin{equation}\label{eqn:inequality}
\dim_{\F_2} H^*(L;\F_2) = \dim_{\cl K} H^*(L;\cl K) = \dim_{\cl K} \wh{HM}(f_X) = \dim_{\Lambda_{\cl K}} \wh{HM}(f_X)\otimes_{\cl K} \Lambda_{\cl K} \leq \dim_{\Lambda_{\cl K}} \wh{HF}(H).
\end{equation}
Combining these chains of inequalities concludes the proof of Theorem \ref{thm:main}.

Of course, the proof presented here does not deal with the subtleties of rigorously defining the relevant counts. In Section \ref{subsec:official-proof}, after introducing some more notions and stating several technical results, we will make the proof rigorous.

\section*{Part II.}

\section{Abstract Notions}
The goal of this section is to introduce abstract notions associated to flow categories following the exposition in \cite{Bai_Xu_foundation} so that we can make the discussions on moduli spaces in Section \ref{sec:moduli} more structured.

\subsection{Regular stratification category with involutions}\label{sec:abstract}

In \cite[Part I]{Bai_Xu_foundation} many formal notions were introduced to package the procedure of global chart construction. We would like to briefly recall the necessary ones which are needed in the current slightly simpler situation. 

\subsubsection{Regular posets}

We use certain kinds of posets (partially ordered sets) to describe the combinatorial data. In \cite{Bai_Xu_foundation}, these are called {\bf regular posets} and {\bf normal posets}, where the latter are more restrictive. 

We do not want to recall the detailed definition here but would like to summarize important properties of them. 

\begin{enumerate}
    \item Regular posets are always countable. 

    \item Given a regular poset $A$, any element $a\in A$ has a well-defined depth, denoted by $\dep(a) \in {\mb Z}_{\geq 0}$, which is monotone decreasing under the partial order, with maximal elements having depth zero.

    \item Given a regular poset $A$,  Alexandrov open subsets are still regular posets. Moreover, for each $\gamma \in A$, the closed set
    \beqn
    \partial^\gamma A:= \{ \beta \leq \gamma \}
    \eeqn
    is regular.

    \item Finite product and finite disjoint union of regular posets are still regular. 
\end{enumerate}
Moreover, regular posets form a distributive monoidal category whose units are singletons (\cite[Proposition 2.9]{Bai_Xu_foundation}), denoted by $\regpos$. A morphism between regular posets is an injective monotone map which preserves the depth function up to a global shift (called the {\bf degree} or {\bf codimension} of the morphism). 

{\bf Normal posets} are regular posets satisfying an additional property: a regular poset $A$ is called normal if for each element $\gamma \in A$ of depth $1$, there exist either one or two maximal elements above it under the partial order. 

Typical examples of normal posets are those indexing the strata of moduli spaces in Floer theory. Let $({\mc P}, \leq)$ be a general countable poset (such as the set of capped orbits of a nondegenerate Hamiltonian, where $p \leq q$ if there is a stable Floer trajectory connecting $p$ to $q$). Then for each pair of elements $p\leq q$, there is an induced poset
\beq\label{typical_poset}
A_{pq}^{\mc P}:= \left\{ \begin{array}{ll} \big\{ pr_1 \cdots r_l q\ |\ p<r_1 < \cdots < r_l < q \big\},\ &\ p< q,\\
\big\{ pp \big\},\ &\ p = q. 
\end{array} \right.
\eeq
It has a unique maximal element $pq$. On the other hand, when discussing bimodules and homotopies, we will consider posets which are not of the form of $A_{pq}^{\mc P}$.

\subsubsection{Regular stratification category}

The main purpose of introducing the formal concept of regular stratification category is to package various steps in the global chart construction in a compact fashion. 

\begin{defn}\cite[Definition 2.10]{Bai_Xu_foundation} \label{defn_regular_stratification_category}
A {\bf regular stratification category} is a distributive monoidal category $\uds{\bf C}$ together with a distributive monoidal functor $\uds{\bf C} \to \regpos$ satisfying the following properties. In this notation, if an object $K$ of $\uds{\bf C}$ is sent to a regular poset $A$, then we say that $K$ is $A$-stratified.

\begin{enumerate}
    \item There is an initial object of $\uds{\bf C}$ which is $\emptyset$-stratified.

    \item For any morphism $f: A \to B$ in $\regpos$ and any $B$-stratified object $K$, there is a canonical ``pullback'' which is an $A$-stratified object $f^* K$, together with a canonical morphism $f^* K \to K$. The pullback needs to satisfy the functoriality condition.
\end{enumerate}
\end{defn}
In particular, if $K$ is $A$-stratified and $\gamma \in A$, then we have an object $\partial^\gamma K$ which is the pullback via the inclusion $\partial^\gamma A \to A$.

Now we introduce two regular stratification categories. The first example of regular stratification category is the category of stratified topological spaces, denoted by $\uds{\bf Top}$. Its objects are locally compact, Hausdorff, and second countable topological spaces $M$ together with a  partition 
\beqn
M = \bigsqcup_{\alpha \in A} M_\alpha
\eeqn
into locally closed subsets satisfying the following conditions.
\begin{enumerate}
    \item The partition is locally finite.

    \item $A$ is a regular poset. 

    \item The canonical map $M \to A$ is continuous with respect to the Alexandrov topology on $A$ and its range is finite.
\end{enumerate}
We can write an object as $(M \to A)$ and call it an ${\mb A}$-stratified space. A morphism from $(M \to A)$ to $(N \to B)$ consists of a continuous map $f: M \to N$ together with a homogeneous (in the sense that it shifts the depth function up to a uniform constant) poset map $A \to B$ such that the diagram 
\beqn
\xymatrix{ M \ar[r]^f \ar[d]  & N \ar[d]\\
 A \ar[r] &  B }
\eeqn
commutes.

\subsubsection{Stratified pseudomanifolds}

\begin{defn}
An $n$-dimensional {\bf pseudomanifold} is a topological space $M$ with a partition $M = {\rm Int} M \sqcup \delta M$ satisfying
\begin{enumerate}
    \item ${\rm Int} M$ is open, dense, and itself an $n$-dimensional topological manifold.

    \item $\delta M$ admits a partition into topological manifolds of dimensions at most $n-2$.
\end{enumerate}
\end{defn}

\begin{defn}
Let $A$ be a countable regular poset. An {\bf $A$-stratified $n$-dimensional pseudomanifold} is an $A$-stratified space $X$ together with partitions $X_\alpha = {\rm Int} X_\alpha \sqcup \delta X_\alpha$ such that
\begin{enumerate}
    \item For each $\alpha \in A$, $X_\alpha = {\rm Int} X_\alpha \sqcup \delta X_\alpha$ is a pseudomanifold of dimension $n - \dep (\alpha)$. 

    \item For $x \in X_\alpha$ and each maximal $\gamma \in A^{\rm max}$ with $\alpha \leq \gamma$, there exists a corner chart
    \beqn
    U_{x, \alpha} \times [0, \epsilon)^{\mb{Face}_\alpha^\gamma } \to \partial^\gamma X
    \eeqn
    which sends ${\rm Int} U_{x, \alpha} \times [0,\epsilon)^{\mb{Face}_\alpha^\gamma}$ resp. $\delta U_{x, \alpha} \times [0,\epsilon)^{\mb{Face}_\alpha^\gamma}$ into ${\rm Int} \partial^\gamma X$ resp. $\delta( \partial^\gamma X)$. Here $\mb{Face}_\alpha^\gamma$ is the set of depth-one elements between $\alpha$ and $\gamma$. 
\end{enumerate}
\end{defn}

Notice that a stratified zero-dimensional pseudomanifold (with boundary) is a discrete collection of points and a stratified 1-dimensional pseudomanifold is a disjoint union of 1-manifolds (with boundary).

\begin{defn}
An embedding of stratified pseudomanifolds from $X_1$ to $X_2$ is an embedding of stratified spaces $f: X_1 \to X_2$ such that $f({\rm Int} X_1) \subset {\rm Int}(X_2)$ and $f(\delta X_1) \subset \delta X_2$. The category of stratified pseudomanifolds, denoted by $\pman$, has objects being stratified pseudomanifolds and morphisms being embeddings. 
\end{defn}

\begin{example}[Standard corner]\label{example_corner}
Let $I$ be a finite set and let 
\beqn
A^I:= \{ I' \subset I\}
\eeqn
which is a normal poset with partial order $I'' \leq I' \Longleftrightarrow I' \subset I''$. Then for $\epsilon>0$, the {\bf standard corner} is the space
\beqn
[0, \epsilon)^I:= \Big\{ (t_i)_{i \in I}\ |\ 0 \leq t_i < \epsilon\Big\}
\eeqn
It is stratified by $A^I$ with each stratum being
\beqn
([0, \epsilon)^I)_{I'} = \Big\{ (t_i)_{i \in I} \ |\ i \in I' \Longleftrightarrow t_i = 0 \Big\}
\eeqn
This is an object of both $\uds{\bf Top}$ and $\pman$.
\end{example}

\subsubsection{Involutions}

To treat the situation related to the Arnold--Givental conjecture, we introduce the abstract notion of involutions on regular stratification categories. 

\begin{defn}
Let $\uds{\bf C}$ be a regular stratification category. An {\bf involution} is a covariant functor 
\beqn
\dagger: \uds{\bf C} \to \uds{\bf C}
\eeqn
satisfying $\dagger^2 = {\rm Id}$. The pair $(\uds{\bf C}, \dagger )$ is called an {\bf involutive regular stratification category}. An {\bf involutive functor} from $(\uds{\bf C}, \dagger)$ to $(\uds{\bf D}, \dagger )$ is a functor which commutes with the involutions.

For a morphism $T: K_1 \to K_2$ of $\uds{\bf C}$, denote by  
\beqn
T^\dagger: K_1^\dagger \to K_2^\dagger
\eeqn
the induced morphism by $\dagger$.
\end{defn}

In the case of $\uds{\bf Top}$ and $\pman$, we will choose the involution to be the identity functor. Indeed, the involution is not the identity only when complex structures are involved. In those cases, involutions will be the operation flipping the complex structure.

\subsection{Flow categories etc. and involutions}

We upgrade the abstract notions of flow categories, bimodules, and homotopies to the involutive case. 

\subsubsection{Involutive flow categories}

\begin{defn}[Flow categories]\label{defn_flow_category}
Let $(\uds{\bf C}, \dagger)$ be an involutive regular stratification category and $\Gamma \subset {\mb R}^k$ be a subgroup. A {\bf filtered $\Gamma$-equivariant involutive Novikov flow category}  consists of the following data.

\begin{enumerate}

\item {\bf (Filtration)} The set of objects ${\rm Ob}{\mb F}$, which is a nonempty countable poset and a monotone decreasing function
\beqn
{\mc A}: {\rm Ob}{\mb F} \to {\mb R}^k.
\eeqn

\item {\bf (Moduli spaces)} For any pair $p \leq q$ of objects, an $A_{pq}^{\mb F}$-stratified object $M_{pq}^{\mb F}$ of $\uds{\bf C}$ where $A_{pq}^{\mb F}$ is an abbreviation of $A_{pq}^{{\rm Ob}{\mb F}}$ (see \eqref{typical_poset}). 

\item {\bf (Structural maps)}  For any triple $p \leq r \leq q$ of objects, a {\bf composition map}
\beqn
\iota_{prq}^{\mb F}: M_{pr}^{\mb F} \times M_{rq}^{\mb F} \to M_{pq}^{\mb F}|_{A_{prq}^{\mb F}}
\eeqn
which is a morphism of $\uds{\bf C}$.

\item {\bf (Involution)} An order-preserving involution $p\mapsto p^\dagger$ on ${\rm Ob}{\mb F}$ and for all pairs $p, q$ of objects isomorphisms
\beqn
\tau_{pq}^{\mb F}: (M_{pq}^{\mb F})^\dagger \cong M_{p^\dagger q^\dagger}^{\mb F}
\eeqn
such that $\tau_{p^\dagger q^\dagger}^{\mb F} \circ (\tau_{pq}^{\mb F})^{\dagger}$ is the identity of $M_{pq}^{\mb F}$.

\item {\bf (Novikov symmetry)} A free $\Gamma$-action on ${\rm Ob}{\mb F}$ and for all $\lambda \in \Gamma$, $p, q \in {\rm Ob}{\mb F}$ isomorphisms
\beqn
T_\lambda^{\mb F}: M_{pq}^{\mb F} \cong M_{\lambda p\ \lambda q}^{\mb F}
\eeqn
such that $T_{\lambda + \lambda'}^{\mb F} = T_\lambda^{\mb F} \circ T_{\lambda'}^{\mb F}$.
\end{enumerate}
These data must satisfy the following conditions.

\begin{enumerate}

\item {\bf (Monoidal property)}  When $p=q$, $M_{pp}^{\mb F}$ is a monoidal unit of $\uds{\bf C}$ and the composition maps $\iota_{ppq}^{\mb F}$ and $\iota_{pqq}^{\mb F}$ are the canonical isomorphisms. 

\item {\bf (Symmetry of filtration)} The filtration ${\mc A}$ is invariant under the involution and $\Gamma$-equivariant, i.e.
\begin{align*}
&\ {\mc A}(p^\dagger) = {\mc A}(p),\ &\ {\mc A}(\lambda p) = {\mc A}(p) - \lambda.
\end{align*}

\item {\bf (Involution is compatible with the $\Gamma$-action)} $(\lambda p)^\dagger = \lambda p^\dagger$ for all $\lambda \in \Gamma$ and $p \in {\rm Ob}{\mb F}$. Moreover, the following diagram commutes.
\beqn
\xymatrix{   (M_{pq})^\dagger \ar[r]^{\tau_{pq}} \ar[d]_{T_\lambda^\dagger} &   M_{p^\dagger q^\dagger}\ar[d]^{T_\lambda} \\
             ( M_{\lambda p\ \lambda q})^\dagger \ar[r]_{\tau_{\lambda p\ \lambda q}} &    M_{\lambda p^\dagger \ \lambda q^\dagger}   }
\eeqn

\item {\bf (Involution is compatible with the structural maps)} For $p\leq r \leq q$, the diagram commutes
\beqn
\xymatrix{   (M_{pr})^\dagger \times (M_{rq})^\dagger \ar[r] \ar[d]^{\tau_{pr} \times \tau_{rq}}    &    (M_{pq}|_{A_{prq}^{\mb F}})^\dagger \ar[d]^{\tau_{pq}|_{{A_{p^\dagger r^\dagger q^\dagger}^{\mb F}}}}\\
     M_{p^\dagger r^\dagger}\times M_{r^\dagger q^\dagger} \ar[r]  &    M_{p^\dagger q^\dagger}|_{A_{p^\dagger r^\dagger q^\dagger}^{\mb F}}}
     \eeqn

     \item {\bf ($\Gamma$-action is compatible with structural maps)} For $p\leq r \leq q$ and $\lambda \in \Gamma$, the diagram commutes
\beqn
\xymatrix{   M_{pr} \times M_{rq} \ar[r] \ar[d]    &    M_{pq}|_{A_{prq}^{\mb F}} \ar[d]\\
     M_{\lambda p\ \lambda r} \times M_{\lambda r\ \lambda q} \ar[r]  &    M_{\lambda p\ \lambda\ q}|_{A_{\lambda p\ \lambda r\ \lambda q}^{\mb F}}}
     \eeqn     

\item {\bf (Associativity)}: for any quadruple $p \leq r \leq s \leq q$ of objects, the following diagram commutes, where the arrows are defined using composition maps:
\beqn
\xymatrix{   &    M_{pr}^{\mb F} \times M_{rs}^{\mb F} \times M_{sq}^{\mb F}  \ar[ld] \ar[rd] & \\
 M_{ps}^{\mb F} \times M_{sq}^{\mb F}  \ar[rd]  &  &  M_{pr}^{\mb F} \times M_{rq}^{\mb F} \ar[ld] \\
& M_{pq}^{\mb F}  & }
\eeqn
\end{enumerate}
\end{defn}

Using the above notions, the discussion in Section \ref{sec:moduli} can be summarized as follows.

\begin{prop}
\begin{enumerate}

\item The Tate-Morse moduli spaces $M_{vw}^{\mb M}$ form a ${\mb Z}$-equivariant involutive Novikov flow category ${\mb M}$ enriched in $(\pman, \dagger)$ whose object set is 
\beqn
{\rm Ob}{\mb M} = {\rm crit} (f_X) \times {\rm crit}(f_{\rm Tate})
\eeqn
with filtration being the function $f_{\rm Tate}$.

\item The Tate-Floer moduli spaces $M_{pq}^{\mb F}$ form a $\Pi\times {\mb Z}$-equivariant involutive Novikov flow category ${\mb F}$ enriched in $(\uds{\bf Top}, \dagger)$ whose object set is 
\beqn
{\rm Ob}{\mb F} = \tilde {\mc O}(H) \times {\rm crit}(f_{\rm Tate})
\eeqn
with filtration being the ${\mb R}^2$-valued function ${\mc A} = ({\mc A}_H, f_{\rm Tate})$.
\end{enumerate}
\end{prop}

\subsubsection{Involutive flow bimodules}

Flow bimodules are abstract packages of moduli spaces considered in situations such as continuation maps, PSS and SSP maps, etc. 

\begin{defn}[Involutive bimodule] \label{defn_bimodule}
Let $(\uds{\bf C}, \dagger)$ be an involutive regular stratification category and $\Gamma \subset {\mb R}^k$ be a subgroup. Let ${\mb F}$ and ${\mb F}'$ be filtered $\Gamma$-equivariant involutive flow categories enriched in $\uds{\bf C}$. A {\bf $\Gamma$-equivariant involutive bimodule} over $({\mb F}; {\mb F}')$, denoted by ${\mb B}$, consists of the following data.
\begin{enumerate}

\item {\bf (Moduli spaces)} For objects $p \in {\rm Ob}{\mb F}$ and $p' \in {\rm Ob} {\mb F}'$, an object $M_{p; p'}^{\mb B}$ of $\uds{\bf C}$.

\item {\bf (Structural maps)} For objects $p\leq q$ in ${\mb F}$, an injective homogeneous poset map of degree $1$ (in the sense that the depth function gets shifted by $1$)
\beqn
A^{{\mb F}}_{pq} \times A_{q; p'}^{\mb B} \to A_{p; p'}^{\mb B}
\eeqn
with image
\beqn
A_{pq;p'}^{\mb B}  \subset A_{p; p'}^{\mb B},
\eeqn
together with a morphism in $\uds{\bf C}$
\beq\label{module_structural_map}
\iota_{pq; p'}^{\mb B}: M_{pq}^{{\mb F}} \times M_{q; p'}^{\mb B} \to M_{p; p'}^{\mb B}|_{A_{pq; p'}^{\mb B}}.
\eeq
Similarly, for a pair of objects $q' \leq  p'$ in ${\mb F}'$, an injective poset map homogeneous of degree $1$
\beqn
A_{p; q'}^{\mb B}\times A_{q'p'}^{{\mb F}'} \to A_{p; p'}^{\mb B}
\eeqn
with image
\beqn
A_{p; q'p'}^{\mb B} \subset A_{p; p'}^{\mb B},
\eeqn
together with a morphism in $\uds{\bf C}$
\beq\label{module_structural_map_2}
\iota_{p; q'p'}^{\mb B}: M_{p; q'}^{\mb B}\times M_{q'p'}^{{\mb F}'} \to M_{p; p'}^{\mb B} |_{A_{p; q'p'}^{\mb B}}.
\eeq

\item {\bf (Novikov symmetry)} For all $\lambda \in \Gamma$ and objects $p, p'$, an isomorphism
\beqn
T_\lambda^{\mb B}: M_{p; p'}^{\mb B} \cong M_{\lambda p; \lambda p'}^{\mb B}
\eeqn
such that $T_{\lambda + \lambda'}^{\mb B} = T_\lambda^{\mb B} \circ T_{\lambda'}^{\mb B}$.

\item {\bf (Involution)} For objects $p, p'$, an isomorphism
\beqn
\tau_{p p'}^{\mb B}: (M_{p; p'}^{\mb B})^\dagger \cong M_{p^\dagger; (p')^\dagger}^{\mb B}.
\eeqn

\end{enumerate}
These objects need to satisfy the following conditions.
\begin{enumerate}

\item {\bf (Monoidal property)} When $p = q$ (recall $M_{pp}^{\mb F}$ is a monoidal unit), the morphism $\iota_{pp; p'}^{\mb B}$ is the natural isomorphism. Similarly, when $p' = q'$, the morphism $\iota_{p; p'p'}^{\mb B}$ is the natural isomorphism.

\item {\bf (Involution is compatible with $\Gamma$-action)} Given $p, p'$ and $\lambda$, the following diagram commutes.
\beqn
\xymatrix{ (M_{p; p'}^{\mb B} )^\dagger  \ar[r] \ar[d] &   M_{p^\dagger; (p')^\dagger }^{\mb B} \ar[d]  \\
        (M_{ \lambda p; \lambda p'}^{\mb B} )^\dagger \ar[r]   &   M_{\lambda p^\dagger; \lambda (p')^\dagger}^{\mb B}    }
\eeqn

\item {\bf (Involution is compatible with structural maps}) The following diagram is commutative.
\beqn
\xymatrix{   (M_{pq}^{\mb F})^\dagger \times ( M_{q; p'}^{\mb B})^\dagger \ar[r] \ar[d]   &     (M_{p; p'}^{\mb B})^\dagger     \ar[d]  &   (M_{p; q'}^{\mb B})^\dagger \times (M_{q'p'}^{{\mb F}'})^\dagger \ar[l] \ar[d]   \\
          M_{p^\dagger q^\dagger}^{\mb F} \times M_{q^\dagger; (p')^\dagger}^{\mb B}  \ar[r]  &   M_{p^\dagger; (p')^\dagger}^{\mb B} &   M_{p^\dagger; (q')^\dagger}^{\mb B}\times M_{(q')^\dagger (p')^\dagger}^{{\mb F}'}  \ar[l]   } 
\eeqn

\item {\bf ($\Gamma$-action is compatible with structural maps)} The following diagram is commutative.
\beqn
\xymatrix{   M_{pq}^{\mb F} \times M_{q; p'}^{\mb B} \ar[r] \ar[d]   &     M_{p; p'}^{\mb B}    \ar[d]  &   M_{p; q'}^{\mb B} \times M_{q'p'}^{{\mb F}'} \ar[l] \ar[d]   \\
          M_{\lambda p\ \lambda q }^{\mb F} \times M_{\lambda q ; \lambda p' }^{\mb B}  \ar[r]  &   M_{\lambda p; \lambda p'}^{\mb B} &   M_{\lambda p; \lambda  q'}^{\mb B}\times M_{\lambda q' \ \lambda p'}^{{\mb F}'}  \ar[l]   } 
\eeqn

\item {\bf (Associativity)} The following associativity properties hold.
\begin{enumerate}
\item For $p \leq r \leq q$ in ${\mb F}$, the following diagram commutes.
\beq\label{module_associativity_1}
\vcenter{\xymatrix{  M_{pr}^{{\mb F}} \times M_{rq}^{{\mb F}} \times M_{q; p'}^{\mb B} \ar[rr] \ar[d]  & & 
M_{p q}^{{\mb F}} \times M_{q; p'}^{\mb B} \ar[d]\\
 M_{p r}^{{\mb F}} \times M_{r; p'}^{\mb B} \ar[rr]  & & M_{p;p'}^{\mb B} }}
\eeq

\item For $q' \leq r' \leq p'$ in ${\mb F}'$, the following diagram commutes.
\beq\label{module_associativity_3}
\vcenter{ \xymatrix{  M_{p ; q'}^{{\mb B}}\times M_{q'r'}^{{\mb F}'} \times M_{r'p'}^{{\mb F}'}  \ar[rr] \ar[d] & & M_{p ; r'}^{{\mb B}} \times M_{r'p'}^{{\mb F}'} \ar[d]\\
M_{p ; q'}^{{\mb B}}\times M_{q'p'}^{{\mb F}'} \ar[rr]  & & M_{p ; p'}^{{\mb B}} }}
\eeq

\item For $p \leq  q $ in ${\mb F}$ and $q' \leq p'$ in ${\mb F}'$, the following diagram commutes.
\beq
\vcenter{ \xymatrix{  M_{p q }^{{\mb F} } \times M_{q; q'}^{{\mb B}} \times M_{q'p'}^{{\mb F}'}  \ar[rr] \ar[d]  & &    M_{p ; q'}^{{\mb B}} \times M_{q'p'}^{{\mb F}'} \ar[d] \\
 M_{p q }^{{\mb F} } \times M_{q; p'}^{{\mb B}}  \ar[rr] & &  M_{p ; p'}^{{\mb B}}   }}
\eeq
\end{enumerate}
\end{enumerate}
\end{defn}

The following proposition is also straightforward using the observations in Section \ref{sec:moduli}.

\begin{prop}
\begin{enumerate}

\item 
The Tate--PSS moduli spaces $M_{p; v}^{\rm PSS}$ resp. Tate--SSP moduli spaces $M_{w; p}^{\rm SSP}$ form a $\Pi \times {\mb Z}$-equivariant involutive bimodule ${\mb B}^{\rm PSS}$ resp. ${\mb B}^{\rm SSP}$ enriched in $(\uds{\bf Top}, \dagger)$ over $({\mb M}; {\mb F})$ resp. over $({\mb F}; {\mb M})$. 

\item The $\zeta = 0$ slice of the homotopy moduli spaces $M_{v;w}^{\rm hmtp}$ forms a $\Pi\times {\mb Z}$-equivariant involutive bimodule ${\mb B}^{\rm pearly}$ over $({\mb M}; {\mb M})$ enriched in $(\uds{\bf Top}, \dagger)$.
\end{enumerate}
\end{prop}

\subsubsection{Bimodule concatenations}

In \cite{Bai_Xu_foundation}, we introduced the abstract notion of concatenations of bimodules. The level of generality of \cite{Bai_Xu_foundation} is needed because the pearly flow category considered there does not have the typical posets \eqref{typical_poset}. In the current situation (as well as in \cite{Bai_Xu_Arnold}), we do not need the abstract formulation. 

Let ${\mb F}$, ${\mb F}'$, ${\mb F}''$ be flow categories enriched in $\uds{\bf C}$ and let ${\mb B}$ resp. ${\mb B}'$ be a bimodule over $({\mb F}; {\mb F}')$ resp. over $({\mb F}'; {\mb F}'')$. We assume that the posets underlying the flow-category moduli spaces, which are
\beqn
A_{pq}^{\mb F},\ A_{p'q'}^{{\mb F}'},\ A_{p''q''}^{\mb{F}''},
\eeqn
are the typical ones \eqref{typical_poset}. We also assume that the posets underlying the bimodule moduli spaces are 
\beqn
\begin{split}
A_{p; p'}^{\mb B} = &\ \Big\{ (pr_1 \cdots r_l; s_k' \cdots s_1' p')\ |\ p<r_1<\cdots <r_l,\ s_k' < \cdots < s_1' < p' \Big\},\\
A_{p'; p''}^{{\mb B}'} = &\ \Big\{ (p'r_1' \cdots r_l'; s_k'' \cdots s_1'' p'')\ |\ p' < r_1' <\cdots < r_l',\ s_k'' < \cdots < s_1'' < p'' \Big\}.
\end{split}
\eeqn
These posets are all normal. Then for each pair of objects $p\in {\rm Ob}{\mb F}$ and $p''\in {\rm Ob}{\mb F}''$, define the set
\beqn
A_{p;p''}^{{\mb B} \circ {\mb B}'}:= \left( \bigsqcup_{p' \in {\rm Ob} \mb{F}'}  A_{p;p'}^{\mb B} \times A_{p';p''}^{{\mb B}'} \right)/ \sim
\eeqn
where the equivalence relation $\sim$ is defined as follows. For any $p', q' \in {\rm Ob}{\mb F}'$, the images of the two maps
\beqn
\xymatrix{ & A_{p;p'}^{\mb B} \times A_{p'q'}^{{\mb F}'} \times A_{q' ;p''}^{{\mb B}'}  \ar[ldd]_{{\rm Id}\times \iota_{p'q'; p''}^{{\mb B}'} } \ar[rdd]^{\iota_{p; p'q'}^{\mb B}\times {\rm Id} } \\ 
& & \\
   A_{p; p'}^{{\mb B}} \times A_{p'; p''}^{{\mb B}'}    & & A_{p; q'}^{{\mb B}} \times A_{q';p''}^{{\mb B}'} }
\eeqn
are identified. The equivalence relation $\sim$ is generated by this identification. The following property of this set was proved in \cite[Section 6.5]{Bai_Xu_foundation}.

\begin{prop} Suppose $p\in {\rm Ob}{\mb F}$ and $p'' \in {\rm Ob}{\mb F}''$.
\begin{enumerate}
    \item There is a natural partial order on $A_{p; p''}^{{\mb B} \circ {\mb B}'}$ making it a normal poset. 

    \item For any $p' \in {\rm Ob}{\mb F}'$, the natural map $A_{p; p'}^{\mb B} \times A_{p'; p''}^{{\mb B}'} \to A_{p; p''}^{{\mb B} \circ {\mb B}'}$ is a codimension zero morphism of $\regpos$. Denote its image by $A_{p; p'; p''}^{{\mb B}\circ {\mb B}'}$. 
\end{enumerate}
\end{prop}

Now if ${\mb F}$, ${\mb F}'$, ${\mb F}''$ and ${\mb B}$, ${\mb B}'$ are all involutive, we can see that poset bimodule is also involutive in a natural way. For example, there is the commutative diagram 
\beqn
\xymatrix{ A_{p; p'}^{\mb B} \times A_{p'; p''}^{{\mb B}'}  \ar[rr]  \ar[d]   &  &    A_{p^\dagger; (p')^{\dagger}}^{{\mb B}} \times A_{(p')^\dagger; (p'')^\dagger}^{{\mb B}'}  \ar[d] \\
           A_{p; p''}^{{\mb B}\circ {\mb B}'} \ar[rr]   &         & A_{p^\dagger; (p'')^\dagger}^{{\mb B} \circ {\mb B}' } }
\eeqn

Now we recall the abstract notion of bimodule concatenations given in \cite{Bai_Xu_foundation}.

\begin{defn}[Bimodule concatenation]\cite[Definition 6.13]{Bai_Xu_foundation} \label{defn_general_concatenation}
Let ${\mb F}, {\mb F}', {\mb F}''$ be flow categories enriched in $\uds {\bf C}$. Let ${\mb B}$ resp. ${\mb B}'$ be a bimodule over $({\mb F}; {\mb F}')$ resp. over $({\mb F}'; {\mb F}'')$. Then a {\bf concatenation} of ${\mb B}$ and ${\mb B}'$, denoted by ${\mb B} \circ {\mb B}'$, consists of the following objects.
\begin{enumerate}
    \item A bimodule over $({\mb F}; {\mb F}'')$.
    
    \item For each pair of composable object pairs $(p; p')$ and $(p'; p'')$ of objects, a codimension zero morphism
    \beq\label{codimension_zero_morphism}
    M_{p; p'}^{\mb B} \times M_{p'; p''}^{{\mb B}'} \to M_{p; p''}^{{\mb B} \circ {\mb B}'}.
    \eeq
    \end{enumerate}
    They are required to satisfy the following conditions.
\begin{enumerate}
    \item For each pair $(p; p'')$, the object $M_{p; p''}^{{\mb B} \circ {\mb B}'}$ is stratified by $A_{p; p''}^{{\mb B} \circ {\mb B}'}$.

    \item The underlying poset maps of the structural maps of ${\mb B} \circ {\mb B}'$ are the same as those in the previously discussed special cases.

    \item The following diagram commutes.
    \beqn
    \xymatrix{ &    M_{p; p'}^{{\mb B}} \times M_{p'q'}^{{\mb F}'} \times M_{q'; p''}^{{\mb B}'} \ar[rd] \ar[ld]   & \\
     M_{p; q'}^{{\mb B}} \times M_{q';p''}^{{\mb B}'}  \ar[rd] & &  M_{p; p'}^{\mb B} \times M_{p'; p''}^{{\mb B}'} \ar[ld] \\
     &  M_{p; p''}^{{\mb B} \circ {\mb B}'}  &   }
    \eeqn

    \item The following diagram commutes.
    \beqn
    \xymatrix{
    M_{pq}^{{\mb F}} \times M_{q;p'}^{\mb B} \times M_{p'; p''}^{{\mb B}'} \ar[r] \ar[d]   &    M_{pq}^{{\mb F}} \times M_{q; p''}^{{\mb B} \circ  {\mb B}'} \ar[d] \\
    M_{p;p'}^{{\mb B}} \times M_{p'; p''}^{{\mb B}'} \ar[r] &   M_{p; p''}^{{\mb B} \circ {\mb B}'} }
    \eeqn

    \item The following diagram commutes.
    \beqn
\xymatrix{
    M_{p; p'}^{\mb B} \times M_{p'; q''}^{{\mb B}'} \times M_{q''p''}^{{\mb F}''} \ar[r] \ar[d]   &  M_{p; q''}^{{\mb B} \circ {\mb B}'} \times M_{q''p''}^{{\mb F}''} \ar[d]\\
    M_{p; p'}^{{\mb B}}\times M_{p'; p''}^{{\mb B}'}  \ar[r] &   M_{p; p''}^{{\mb B} \circ {\mb B}'}}
    \eeqn
    
\end{enumerate}
If $\Gamma \subset {\mb R}^k$ is a subgroup, ${\mb F}, {\mb F}', {\mb F}''$, and ${\mb B}$, ${\mb B}'$ are $\Gamma$-equivariant involutive, then the concatenation is called {\bf $\Gamma$-equivariant and involutive} if ${\mb B} \circ {\mb B}'$ is $\Gamma$-equivariant and involutive and the codimension zero morphisms \eqref{codimension_zero_morphism} respect $\Gamma$-action and involutions.
\end{defn}

Coming back to the discussion in Section \ref{sec:moduli}, the following statement follows from the discussions in \cite[Section 11.3]{Bai_Xu_foundation}.
\begin{prop}
The $\zeta = \infty$ slices of the homotopy moduli spaces $M_{x;y}^{\rm hmtp}$ form a $\Pi \times {\mb Z}$-equivariant involutive concatenation ${\mb B}^{\rm PSS} \circ {\mb B}^{\rm SSP}$ enriched in $\uds{\bf Top}$. 
\end{prop}

\subsubsection{Involutive bimodule homotopies}

\begin{defn}\label{homotopy_modules}
Let $(\uds{\bf C}, \dagger)$ be an involutive regular stratification category and $\Gamma \subset {\mb R}^k$ be a subgroup. Let ${\mb F}$, ${\mb F}'$ be filtered $\Gamma$-equivariant involutive Novikov flow categories enriched in $\uds{\bf C}$ and let ${\mb B}_0$, ${\mb B}_1$ be $\Gamma$-equivariant involutive bimodules over $({\mb F}; {\mb F}')$. A {\bf $\Gamma$-equivariant involutive homotopy} from ${\mb B}_0$ to ${\mb B}_1$, denoted by ${\mb H}$, consists of the following data.
\begin{enumerate}

\item {\bf (Moduli spaces)} For each pair $(p ; p')$ of objects of involved flow categories, an object $M_{p; p'}^{{\mb H}}$ in $\uds{\bf C}$ stratified by a normal poset $A_{p; p'}^{\mb H}$.

\item {\bf (Structural maps)} 
\begin{enumerate}

\item For objects $p\in {\rm Ob}{\mb F}$, $p' \in {\rm Ob}{\mb F}'$, a degree-1 morphism in $\uds{\bf C}$
\beqn
 \iota_{p; p'}^{{\mb B}_0 \to {\mb H}} \sqcup \iota_{p; p'}^{{\mb B}_1\to {\mb H}}: M_{p; p'}^{{\mb B}_0} \sqcup M_{ p; p'}^{{\mb B}_1}  \to M_{p; p'}^{\mb H}.
\eeqn

\item For objects $p\leq q$ in ${\mb F}$ and objects $q' \leq p'$ in ${\rm Ob}{\mb F}'$, morphisms in $\uds{\bf C}$ 
\beq\label{homotopy_composition_1}
\iota_{pq; p'}^{\mb H}: M_{p q}^{{\mb F}} \times M_{q; p'}^{\mb H}  \to M_{p; p'}^{{\mb H}}
\eeq
and 
\beq\label{homotopy_composition_2}
\iota_{p; q'p'}^{\mb H}: M_{p; q'}^{\mb H} \times M_{q'p'}^{{\mb F}'} \to M_{p; p'}^{\mb H}.
\eeq
\end{enumerate}

\item {\bf (Involution)} For each pair of objects $p, p'$, an isomorphism
\beqn
\tau_{p; p'}^{\mb H}: ( M_{p; p'}^{\mb H})^\dagger \cong M_{p^\dagger; (p')^\dagger}^{\mb H}
\eeqn
such that $\tau_{p^\dagger; (p')^\dagger}^{\mb H} \circ (\tau_{p; p'}^{\mb H})^\dagger$ is the identity of $M_{p; p'}^{\mb H}$.

\item {\bf (Novikov symmetry)} For each pair of objects $p, p'$ and $\lambda \in \Gamma$, an isomorphism
\beqn
T_\lambda^{\mb H}: M_{p; p'}^{\mb H} \cong M_{\lambda p; \lambda p'}^{\mb H}
\eeqn
such that $T_{\lambda +\lambda'}^{\mb H} = T_{\lambda}^{\mb H} \circ T_{\lambda'}^{\mb H}$.
\end{enumerate}
These objects need to satisfy the following conditions.
\begin{enumerate}

\item {\bf (Monoidal property)} When $p = q$ resp. $q' = p'
$, the morphism \eqref{homotopy_composition_1} resp. \eqref{homotopy_composition_2} is the natural isomorphism from the monoidal structure of $\uds{\bf C}$.

\item {\bf (Involution is compatible with structural maps)} The following diagrams commute.
\beqn
\xymatrix{  (M_{p; p'}^{{\mb B}_0})^\dagger \sqcup (M_{p; p'}^{{\mb B}_1})^\dagger \ar[r] \ar[d] &    (M_{p; p'}^{\mb H})^\dagger \ar[d] \\
            M_{p^\dagger; (p')^\dagger}^{{\mb B}_0} \sqcup M_{p^\dagger; (p')^\dagger}^{{\mb B}_1} \ar[r]   &    M_{p^\dagger; (p')^\dagger}^{{\mb H}}  }  
\eeqn
\beqn
\xymatrix{  (M_{pq}^{\mb F})^\dagger \times (M_{q; p'}^{\mb H})^\dagger \ar[r] \ar[d] &   (M_{p; p'}^{\mb H})^\dagger \ar[d]  &     (M_{p; q'}^{\mb H})^\dagger \times (M_{q'p'}^{{\mb F}'})^\dagger \ar[l] \ar[d]\\
           M_{p^\dagger q^\dagger}^{\mb F}\times M_{q^\dagger; (p')^\dagger}^{\mb H} \ar[r]  &    M_{p^\dagger; (p')^\dagger}^{\mb H}  &     M_{p^\dagger; (q')^\dagger}^{\mb H}\times M_{(q')^\dagger (p')^\dagger}^{{\mb F}'} \ar[l] }
\eeqn

\item {\bf ($\Gamma$-action is compatible with structural maps)} The following diagrams commute.
\beqn
\xymatrix{  M_{p; p'}^{{\mb B}_0} \sqcup M_{p; p'}^{{\mb B}_1} \ar[r] \ar[d] &    M_{p; p'}^{\mb H} \ar[d] \\
            M_{\lambda p; \lambda p'}^{{\mb B}_0} \sqcup M_{ \lambda p; \lambda  p'}^{{\mb B}_1} \ar[r]   &    M_{\lambda p; \lambda p'}^{{\mb H}}  }  
\eeqn
\beqn
\xymatrix{  M_{pq}^{\mb F}  \times M_{q; p'}^{\mb H} \ar[r] \ar[d] &  M_{p; p'}^{\mb H} \ar[d]  &   M_{p; q'}^{\mb H}  \times M_{q'p'}^{{\mb F}'}  \ar[l] \ar[d]\\
           M_{\lambda p \ \lambda q}^{\mb F}\times M_{\lambda q; \lambda p'}^{\mb H} \ar[r]  &    M_{\lambda p; \lambda p' }^{\mb H}  &     M_{\lambda p; \lambda q'}^{\mb H}\times M_{\lambda q'\ \lambda p'}^{{\mb F}'} \ar[l] }
\eeqn

\item {\bf (Associativity H1)} The following diagram commutes.
\beqn
\xymatrix{  M_{pq}^{{\mb F}} \times M_{q; q'}^{\mb H} \times M_{q'p'}^{{\mb F}'} \ar[r] \ar[d] &   M_{p; q'}^{\mb H} \times M_{q'p'}^{{\mb F}'} \ar[d]
\\
M_{pq}^{{\mb F}} \times M_{q; p'}^{\mb H} \ar[r] & M_{p; p'}^{\mb H} }
\eeqn
Hence we have a well-defined morphism 
\beqn
M_{pq}^{{\mb F}} \times M_{q; q'}^{\mb H} \times M_{q'p'}^{{\mb F}'} \to M_{p; p'}^{\mb H}.
\eeqn

\item {\bf (Associativity H2)} For $\alpha = 0, 1$, the following diagram commutes.
\beqn
\xymatrix{  M_{pq}^{{\mb F}} \times M_{q; q'}^{{\mb B}_\alpha} \times M_{q'p'}^{{\mb F}'}  \ar[r] \ar[d] & M_{p; p'}^{{\mb B}_\alpha} \ar[d]  \\
  M_{p q}^{{\mb F}} \times M_{q; q'}^{\mb H} \times M_{q'p'}^{{\mb F}'} \ar[r] &  M_{p; p'}^{{\mb H}} }
  \eeqn

\item {\bf (Associativity H3)}  For $p\leq r \leq q$ in ${\mb F}$ and $q'\leq r'\leq p'$ in ${\mb F}'$, the following diagram commutes.
\beq\label{homotopy_associativity_1}
\xymatrix{  M_{p r}^{{\mb F}} \times M_{r q}^{{\mb F}} \times M_{q; q'}^{\mb H} \times M_{q'r'}^{{\mb F}'} \times M_{r'p'}^{{\mb F}'} \ar[r] \ar[d] &   M_{p r}^{{\mb F}} \times M_{r; r'}^{\mb H} \times M_{r'p'}^{{\mb F}'} \ar[d] \\
M_{p q}^{{\mb F}} \times M_{q; q'}^{{\mb H}} \times M_{q'p'}^{{\mb F}'} \ar[r] & M_{p; p'}^{\mb H} }
\eeq
\end{enumerate}
\end{defn}

The homotopy discussed in Section \ref{sec:moduli} implies the following statement.
\begin{prop}
The homotopy moduli spaces $M_{x;y}^{\rm hmtp}$ form a $\Pi\times {\mb Z}$-equivariant involutive homotopy from $\mb{B}^{\rm PSS}\circ \mb{B}^{\rm SSP}$ to the pearly bimodule $\mb{B}^{\rm pearly}$.
    \end{prop}

\subsubsection{Lifts of flow categories, bimodules, and homotopies}

The key issue in constructing a Floer chain complex (and various chain maps, homotopies) is to regularize certain moduli spaces. Using the terminology we just introduced, the associated flow categories (and bimodules, homotopies), whose morphism spaces are certain moduli spaces, are enriched in $\uds{\bf Top}$. The regularization requires a lift to a better category which allows manipulations in a differential-topological or algebraic-topological sense. 

Let $(\uds{\bf C}_+, \dagger)$ and $(\uds{\bf C}, \dagger)$ be two involutive regular stratification categories. Let ${\bf F}: \uds{\bf C}_+ \to \uds{\bf C}$ be a covariant functor satisfying the natural compatibility conditions.
\begin{enumerate}
    \item ${\bf F}$ respects the distributive monoidal structure. \vspace{0.2cm}

    \item The diagram commutes.
    \beqn
    \xymatrix{ \uds{\bf C}_+ \ar[r]_{\bf F} \ar@<5pt>@/^1pc/[rr] & \underline{\bf C} \ar[r]               & {\regpos}}
    \eeqn

    \item ${\bf F}$ respects the pullback. Namely, if $M_+$ is an $A$-stratified object of $\uds{\bf C}_+$ and $B \subset A$, then one has the natural isomorphism
    \beqn
    {\bf F}(M_+)|_B \cong {\bf F}(M_+|_B).
    \eeqn

    \item ${\bf F}$ intertwines with the involutions.
\end{enumerate}

\begin{defn}(cf. \cite[Definition 6.6]{Bai_Xu_foundation}) \label{defn:lifting}
Let $(\uds{\bf C}_+, \dagger)$, $(\uds{\bf C}, \dagger)$ and ${\bf F}$ as above. Let $\Gamma \subset {\mb R}^k$ be a subgroup. Suppose ${\mb F}$ is a filtered $\Gamma$-equivariant involutive Novikov flow category enriched in $(\uds{\bf C}, \dagger)$. A {\bf lift} of ${\mb F}$ to $\uds{\bf C}_+$ (with respect to the functor ${\bf F}$) is a $\uds{\bf C}_+$-enriched flow category ${\mb F}_+$ satisfying the following conditions.
\begin{enumerate}
\item The sets of objects are identical: ${\rm Ob} {\mb F}_+ = {\rm Ob} {\mb F}$.

\item ${\bf F}(M_{pq}^{{\mb F}_+}) = M_{pq}^{\mb F}$ and ${\bf F}_* \iota_{prq}^{{\mb F}_+} = \iota_{prq}^{\mb F}$.
\end{enumerate}
A {\bf $\Gamma$-equivariant involutive lift} is a lift ${\mb F}_+$ which is also $\Gamma$-equivariant and involutive with the involutions $\tau_{pq}^{{\mb F}_+}$ and $\Gamma$-action $T_\lambda^{{\mb F}_+}$, such that 
\begin{align*}
&\ {\bf F}_* \tau_{pq}^{{\mb F}_+} = \tau_{pq}^{{\mb F}},\ &\ {\bf F}_* T_\lambda^{{\mb F}_+} = T_\lambda^{{\mb F}}.
\end{align*}
\end{defn}

The case of lifts of bimodules and homotopies can be defined in the same fashion. The reader can find the details in \cite[Section 6.3]{Bai_Xu_foundation}. Moreover, one can extend the above definition to the equivariant and involutive case, which we also omit here. 

\subsubsection{Outercollaring}

There is a very convenient operation called {\bf outercollaring}, originally introduced in \cite{FOOO_Kuranishi}, which is used for inductive constructions in Floer-theoretic situations. For a regularly stratified topological space $Y$, fixing a positive $\epsilon>0$, there is a corresponding stratified space
\beqn
\outer Y
\eeqn
stratified by the same poset but has standard corners near each lower stratum. In \cite[Section 5.2]{Bai_Xu_foundation} this was discussed in more detail and in more abstract situations. 

In this paper we fix the outercollaring width $\epsilon$ throughout the paper and omit it from the notation. Notice that if ${\mb F}$ is an (equivariant involutive Novikov) flow category enriched in $\uds{\bf Top}$, then there is an outercollared (equivariant involutive Novikov) flow category $\outer {\mb F}$, still enriched in $\uds{\bf Top}$. The same applies for bimodules and homotopies. See details in \cite[Section 6.4]{Bai_Xu_foundation}.

\section{Kuranishi Spaces, Derived Orbifolds, and FOP Perturbations}\label{sec:FOP}

In this section, we describe various lifts of the flow categorical notions discussed in Section \ref{sec:abstract}. We recall normally complex derived orbifolds and FOP perturbations, and provide proofs of the necessary ingredients for extending the perturbation scheme in the presence of an involution.

We use $\uds{\bf Kur}$, $\uds{\bf SKur}$, and $\uds{\bf SKur}^{\rm NC}$ to denote the categories of Kuranishi charts in the topological, smooth, and stably complex settings. When the discussion can be unified in the three cases, we simply use $\uds{\bf Kur}$. See below for the detailed definitions.

\subsection{Kuranishi spaces and involutions}
We freely use the notions introduced in \cite[Section 2, Section 3]{Bai_Xu_foundation}, and we refer the reader to \emph{loc.cit.} for a more detailed discussion.

\begin{defn}\label{defn_K_chart}
Let $A$ be a countable regular poset.

\begin{enumerate}
\item An $A$-stratified {\bf (topological) Kuranishi space} (K-space for short) is a quadruple $K = (G, V, E, S)$ where $G$ is a compact Lie group, $V$ is an  $A$-manifold with a continuous $G$-action, $E \to V$ is a $G$-equivariant vector bundle, and $S: V \to E$ is a $G$-equivariant section, such that the $G$-action on $V$ has only finite stabilizers.%

\item A Kuranishi space $K = (G, V, E, S)$ is said to be {\bf smooth} if $V$ is a smooth  $A$-manifold, the $G$-action is smooth, and $E \to V$ is a smooth equivariant vector bundle (we do not impose any smoothness condition on $S$).

\end{enumerate}
\end{defn}

\begin{defn}
A {\bf strict map} of smooth/topological Kuranishi spaces from $K_1 = (G_1, V_1, E_1, S_1)$ to $K_2 = (G_2, V_2, E_2, S_2)$, denoted by $\iota_{21}: K_1 \to K_2$, consists of a Lie group homomorphism $\iota_{21}^G: G_1 \to G_2$ and a commutative diagram
\beqn
\xymatrix{   E_1 \ar[r]^{\iota_{21}^E} \ar[d]  &  E_2 \ar[d] \\
    V_1 \ar@/^1pc/[u]^{S_1} \ar[r]_{\iota_{21}^V} & V_2 \ar@/_1pc/[u]_{S_2} }
\eeqn
where $\iota_{21}^V$ is an equivariant smooth/continuous map and $\iota_{21}^E$ is an equivariant smooth/continuous bundle map.
\end{defn}

Notice that strict maps can be composed. We list a few special kinds of strict maps.

\begin{defn}[Stabilization]
Let $K = (G, V, E, S)$ be an $A$-stratified K-chart. Let $\pi_N: N \to V$ be a $G$-equivariant continuous resp. smooth vector bundle and let $D \subset N$ be a $G$-invariant open (disk) subbundle.\footnote{When we talk about disk bundles, we always assume that they are subsets of specified vector bundles, although this may not be apparent in the notation.} The {\bf stabilization} of $K$ by $D$, denoted by ${\rm Stab}_D (K)$, is the $A$-stratified K-chart 
\beqn
{\rm Stab}_D (K) = (G, D, \pi_N^* E \oplus \pi_N^* N, \pi_N^* S \oplus \tau_N)
\eeqn
where $\tau_N: N \to \pi_N^* N$ is the tautological section restricted along $D$. By abuse of language, we also call the morphism
\beqn
\iota_D: K \to {\rm Stab}_D (K)
\eeqn
induced by the zero section $V \to D$ and the obvious bundle embedding $E \to \pi_N^* E \oplus \pi_N^* N$ a {\bf stabilization morphism}.
\end{defn}

\begin{defn}[Shrinking and open embedding]
\begin{enumerate}

\item Let $K = (G, V, E, S)$ be a K-chart. A {\bf shrinking} of $K$ is a K-chart $K' = (G, V', E', S')$ where $V' \subset V$ is a $G$-invariant open neighborhood of $S^{-1}(0)$, $E' = E|_{V'}$, and $S' = S|_{V'}$. 

\item A morphism $\iota_{21}: K_1 \to K_2$ is called an {\bf open embedding} if the underlying poset map is an isomorphism onto an Alexandrov open subset, the map $V_1 \to V_2$ is an isomorphism onto an open subset, and the bundle map is an isomorphism of vector bundles when restricted to the open subset identified with $V_1$. 
\end{enumerate}
\end{defn}

\begin{defn}[Group enlargement]
Let $K = (G, V, E, S)$ be a   K-chart. A {\bf group enlargement} of $K$ by a Lie group embedding $G \hookrightarrow G'$ is a K-chart
\beqn
G'\times_G K = (G', G'\times_G V, G'\times_G E, S')
\eeqn
where $S': G'\times_G V \to G'\times_G E$ is the section naturally induced from $S$. By abuse of language, the K-chart morphism
\beqn
K \to G'\times_G K
\eeqn
induced from the natural maps $V \to G'\times_G V$ and $E \to G'\times_G E$ is also called a group enlargement. 
\end{defn}

Finally, the morphism which will be used most frequently in this paper is a combination of the above more elementary kinds.

\begin{defn}[Strict morphisms of Kuranishi spaces]\label{chart_embedding}
A {\bf strict morphism} or {\bf strict embedding} of Kuranishi spaces from $K_1$ to $K_2$ is a strict map $\iota_{21} : K_1 \to K_2$ such that there exists a commutative diagram of strict maps
\beq\label{embedding_diagram}
\vcenter{ \xymatrix{  K_3 \ar[r]^{\iota_b}    &   K_4  \ar[d]^{\iota_c}   \\    K_1  \ar[u]^{\iota_a} \ar[r]_{\iota_{21}}    &  K_2      }
}
\eeq
where $\iota_a$ is a stabilization by a disk bundle, $\iota_b$ is a group enlargement, and $\iota_c$ is an open embedding. It implies, in particular, that the group morphism $\iota_{21}^G$, the base map $\iota_{21}^V: V_1 \to V_2$, and the bundle map $\iota_{21}^E: E_1 \to E_2$ are all embeddings. 
\end{defn}

One can check that compositions of embeddings are still embeddings, which ultimately follows from the fact that every $G'$-equivariant vector bundle over $G' / G$ for a group embedding $G \hookrightarrow G'$ is induced from a $G$-representation. Also notice that stabilization maps, open embeddings, and group enlargements are all special cases of embeddings.

\begin{defn}[Category of Kuranishi spaces]
The categories of regularly stratified topological and smooth Kuranishi spaces, denoted by $\uds{\bf Kur}$ and $\uds{\bf S Kur}$, are the categories whose objects are stratified Kuranishi charts in the topological and smooth categories and whose morphisms are strict embeddings. 
\end{defn}

\begin{rem}
In \cite{Bai_Xu_foundation} we had to use different kinds of morphisms. This was because when considering moduli spaces of pairs of pants, there is no easy way to obtain lifts to $\uds{\bf Kur}$ with morphisms being strict morphisms. However, the current situation is similar to that of \cite{Bai_Xu_Arnold}. Hence the above definition suffices. 
\end{rem}

\subsubsection{Rigidified embeddings}

In our geometric construction we can obtain objects and morphisms living in a more rigid category of Kuranishi charts. In short, we only allow the stabilization to be a product bundle coming from a $G$-representation.

\begin{defn}\label{defn_rigid_embedding}
Let $\iota_{21}: K_1 \to K_2$ be a strict embedding. 
\begin{enumerate}
    
\item A {\bf rigidification} (which may not exist) of $\iota_{21}$ consists of a $G_1$-representation $W_{21}$ (viewed as a $G_1$-vector bundle $\pi_{21}: \uds W_{21} \to V_1$), a $G_1$-invariant open neighborhood $D_{21} \subset W_{21}$ of the origin (viewed as a disk bundle over $V_1$), and an open embedding $\theta_{21}: G_2\times_{G_1} {\rm Stab}_{D_{21}}(K_1) \to K_2$ such that the diagram \eqref{embedding_diagram} is realized by 
\beqn
\vcenter{ \xymatrix{  {\rm Stab}_{D_{21}} (K_1) \ar[r]^-{\iota_b}  &  G_2 \times_{G_1} {\rm Stab}_{D_{21}}(K_1) \ar[d]^{\theta_{21}} \\
             K_1 \ar[u]^{ \iota_a} \ar[r]_{\iota_{21}}  &  K_2} }.
\eeqn
We denote the rigidification by $(D_{21} \subset W_{21}, \theta_{21})$.

\item Two rigidifications $(D_{21} \subset W_{21}, \theta_{21})$ and $(D_{21}' \subset W_{21}', \theta_{21}')$ are called {\bf equivalent} if there exists an isomorphism $\tau: W_{21} \cong W_{21}'$ of representations, which induces an isomorphism of K-charts
\beqn
\tau: {\rm Stab}_{D_{21}} K_1 \cong {\rm Stab}_{D_{21}'} K_1
\eeqn
such that after shrinking $D_{21}$ and $D_{21}'$ properly, as K-chart morphisms one has
\beqn
\theta_{21} = \theta_{21}' \circ \tau.
\eeqn

\item A {\bf rigidified embedding} from $K_1$ to $K_2$ is an embedding together with an equivalence class of rigidifications.
\end{enumerate}
\end{defn}

By \cite[Lemma 3.13]{Bai_Xu_foundation}, strict rigidified embeddings can be composed. Hence one obtains a category $\uds{\bf Kur}_{\rm rig}$ or $\uds{\bf SKur}_{\rm rig}$ whose morphisms are strict rigidified embeddings.

\subsubsection{Outer-collaring}

The standard corners explained in Example \ref{example_corner} are also objects in $\uds{\bf Kur}$ where the symmetry group is trivial and the obstruction bundle is trivial. Moreover, one can define the collared version of the categories $\uds{\bf Kur}_{\rm rig}$ and $\uds{\bf SKur}_{\rm rig}$ whose objects have collar neighborhoods near each lower stratum and whose morphisms are constant in collar coordinates. These categories are denoted by 
\beqn
\outer \uds{\bf Kur}_{\rm rig},\ \outer\uds{\bf SKur}_{\rm rig}.
\eeqn

\subsubsection{Stable complex structures and normal complex structures}

Next, we discuss stable and normal complex structures on Kuranishi spaces. Geometrically, such structures arise from index-theoretic considerations of (virtual) tangent bundles of moduli spaces of pseudoholomorphic curves. Such structures are also instrumental in incorporating the FOP perturbation scheme into the definition of integral invariants.

We start with a few basic definitions. 

\begin{defn}\label{defn_stable_isomorphism}
Let $G$ be a Lie group, $V$ be a $G$-space, and $E \to V$ be an equivariant vector bundle. 

\begin{enumerate}

\item Let $F \to V$ be another $G$-equivariant vector bundle. A $G$-equivariant {\bf stable (iso)morphism}  from $E$ to $F$ is an equivalence class of triples $(R^-, R^+, \tau)$ where $R^\pm$ are a pair of real vector spaces (viewed as trivial representations of $G$) and 
\beqn
\tau: \uds R^- \oplus E \to \uds R^+ \oplus F
\eeqn
is a $G$-equivariant (iso)morphism of vector bundles. The equivalence relation is induced by $(R^-, R^+, \tau) \sim (R^- \oplus W, R^+ \oplus W, \tau \oplus {\rm Id}_{\uds W})$ where $W$ is another real vector space. Denote a stable (iso)morphism by 
\beqn
\tau: E \overset{s}{\to} F.
\eeqn

\item If $E, F$ are complex vector bundles, then a stable (iso)morphism $\tau: E \overset{s}{\to} F$ is called {\bf complex} if it is represented by a triple $(R^-, R^+, \tau)$ where $R^\pm$ are complex vector spaces and $\tau$ is a complex vector bundle map. 

\end{enumerate}
\end{defn}

One can verify that stable morphisms can be composed and stable isomorphisms admit inverses as stable isomorphisms.

\begin{defn}\label{defn_stable_complex_structure}
Let $G$ be a Lie group, $V$ be a $G$-space, and $E \to V$ be a $G$-equivariant vector bundle.
\begin{enumerate}

\item A {\bf $G$-invariant stable complex structure} on $E$ consists of an equivalence class of pairs $(\tau, F)$ where $F \to V$ is a  $G$-equivariant complex vector bundle and $\tau: E \overset{s}{\to} F$ is a $G$-equivariant stable isomorphism. The equivalence relation is generated by the following relation: $(\tau_1, F_1)$ and $(\tau_2, F_2)$ are equivalent if there is a commutative diagram
    \beqn
    \xymatrix{ E \ar[r]^{\tau_1} \ar[d]_{{\rm Id}_E}  &  F_1 \ar[d] \\
     E \ar[r]_{\tau_2} & F_2}
    \eeqn
    of stable isomorphisms where the right vertical arrow is a complex stable isomorphism.

\item Suppose $E_1$ and $E_2$ are $G$-equivariant stably complex vector bundles over a $G$-space $V$. A $G$-equivariant stable (iso)morphism $f: E_1 \overset{s}{\to} E_2$ is said to be {\bf stably complex} if there exist representatives $(\tau_1, F_1)$ and $(\tau_2, F_2)$ of their stable complex structures and a commutative diagram of stable (iso)morphisms
\beqn
\xymatrix{  E_1 \ar[r]^{\tau_1} \ar[d]_{f}  & F_1 \ar[d] \\
 E_2 \ar[r]_{\tau_2} & F_2}
\eeqn
where the right vertical arrow is a $G$-equivariant complex stable (iso)morphism.

\item Let $V$ be a smooth $G$-manifold (with boundary and corners). A $G$-equivariant stable complex structure on $V$ is a $G$-equivariant stable complex structure on $TV$. Notice that for each stratum $\partial^\alpha V \subset V$, the stable complex structure on $TV$ induces a stable complex structure on $T(\partial^\alpha V)$ because the normal direction is trivial. 
\end{enumerate}
\end{defn}

\subsubsection{Normally complex Kuranishi spaces}

We briefly recall the definition of normal complex structures on smooth Kuranishi spaces, the category $\uds{\bf SKur}^{\rm NC}$, and its collared and rigidified versions. Details can be found in \cite[Section 3.3.3]{Bai_Xu_foundation}.   

Let $V$ be a $G$-space. Notice that a stable morphism of $G$-equivariant vector bundles $E, F \to V$ (Definition \ref{defn_stable_isomorphism}) induces the following well-defined linear maps. For each $x \in V$, let $G_x \subset G$ be the stabilizer. Then the fibers $E_x$ and $F_x$ split as 
\begin{align*}
&\ E_x = \mathring E_x \oplus \check E_x,\ &\ F_x = \mathring F_x \oplus \check F_x
\end{align*}
where $\mathring E_x$, $\mathring F_x$ are $G_x$-trivial summands and $\check E_x$, $\check F_x$ are $G_x$-nontrivial pieces. Then an equivalence class of stable isomorphisms induces well-defined $G_x$-equivariant maps
\beqn
\check \tau_x: \check E_x \to \check F_x.
\eeqn
(In particular, if $G$ is trivial, the maps $\check \tau_x$ are all trivial.)

\begin{defn}\label{defn_equivariant_NC_structure}
Let $G$ be a Lie group, $V$ be a $G$-space, and $E \to V$ be a $G$-equivariant vector bundle.
\begin{enumerate}
    \item Two stable complex structures on $E$, represented by $(\tau_1, F_1)$ and $(\tau_2, F_2)$, are {\bf normally equivalent}, if there exists a diagram of stable isomorphisms
    \beqn
    \xymatrix{ E \ar[r]^{\tau_1} \ar[d]_{{\rm Id}_E} &  F_1 \ar[d]\\
              E \ar[r]_{\tau_2}  & F_2  }
    \eeqn
    where the right vertical arrow is complex, such that for all $x \in V$, the following induced diagram of linear maps commutes.
    \beqn
    \xymatrix{ \check E_x \ar[r]^{\check \tau_{1, x}} \ar[d] & \check F_{1, x} \ar[d] \\
               \check E_x \ar[r]_{\check \tau_{2, x}} & \check F_{2, x} }
    \eeqn 

    \item A $G$-invariant {\bf normal complex structure} on $E$ is a normal equivalence class of stable complex structures. Notice that a normal complex structure on $E$ equips each $\check E_x$ with a $G_x$-invariant complex structure.

    \item Suppose $E_1, E_2 \to V$ are equipped with $G$-equivariant normal complex structures. A $G$-equivariant stable morphism $\tau: E_1 \overset{s}{\to} E_2$ is called {\bf normally complex} if for each $x \in V$, the induced linear map $\check E_{1, x} \to \check E_{2, x}$ is complex linear.
\end{enumerate}
\end{defn}

Now we can recall the notion of normal complex Kuranishi spaces. 

\begin{defn}\label{defn_NC_Kuranishi}
A {\bf normally complex Kuranishi space} (NC Kuranishi space for short) is a smooth Kuranishi space $K = (G, V, E, S)$ equipped with a $G$-invariant complex structure on $E$ and a $G$-invariant normal complex structure on $TV/{\mf g}$. An {\bf open embedding} of NC Kuranishi spaces from $K_1 = (G_1, V_1, E_1, S_1)$ to $K_2 = (G_2, V_2, E_2, S_2)$ is a smooth open embedding $\iota_{21}: K_1 \to K_2$ such that 1) the vector bundle isomorphism 
\beqn
d\iota_{21}: TV_1/{\mf g}_1 \to TV_2/{\mf g}_2
\eeqn
is normally complex and 2) the vector bundle isomorphism $\iota_{21}^E: E_1 \to E_2$ is complex-linear.
\end{defn}

One can again discuss stabilizations (by complex vector bundles), strict embeddings, conjugations, and embeddings of normally complex Kuranishi spaces. Then one obtains a regular stratification category called the category of normally complex Kuranishi spaces, denoted by 
\beqn
\uds{\bf SKur}^{\rm NC}
\eeqn
and the corresponding collared and rigidified version $ \outer \uds{\bf SKur}^{\rm NC}_{\rm rig}$. 

The next definition is the new notion relevant for this paper. Here, when we say we reverse a bundle complex structure $I: V \to V$, we mean taking the bundle complex structure $-I$.

\begin{defn}[Involution on NC Kuranishi spaces]
Let $K = (G, V, E, S)$ be a smooth NC Kuranishi space. Its {\bf conjugate}, denoted by 
\beqn
K^\dagger = (G^\dagger, V^\dagger, E^\dagger, S^\dagger)
\eeqn
has the same underlying smooth Kuranishi space as $K$ with the NC structure on $V$ and the complex structure on $E$ both reversed. 
\end{defn}

One can easily see that $(\uds{\bf Kur}^{\rm NC}, \dagger)$ forms an involutive regular stratification category. 

\subsection{Derived orbifolds}
Now we recall the counterpart of Kuranishi spaces in the orbifold setting, which should be thought of as being obtained from taking the orbifold quotient of Lie group actions with finite stabilizers. A more detailed discussion can be found in \cite[Section 4]{Bai_Xu_foundation}.

\subsubsection{Stratified orbifolds}

All orbifolds, unless otherwise stated, are assumed to be smooth and effective. They are spaces which are locally modelled by $U / \Gamma$ where $\Gamma$ is a finite group and $U$ is an effective smooth $\Gamma$-manifold. Orbifolds with boundary and corners are locally modelled on $U/ \Gamma$ where $U$ is a $\Gamma$-manifold with boundary and corners. If $A$ is a regular poset, then one can consider the notion of $A$-orbifold. 

\begin{defn}[Derived orbifolds]
Fix a regular poset $A$. An $A$-stratified {\bf derived orbifold} is a triple ${\mc D} = ({\mc U}, {\mc E}, {\mc S})$ where ${\mc U}$ is an effective $A$-orbifold, ${\mc E} \to {\mc U}$ is a smooth orbifold vector bundle, and ${\mc S}: {\mc U} \to {\mc E}$ is a continuous section. A derived orbifold ${\mc D}$ is called {\bf compact} if the coarse space of ${\mc S}^{-1}(0)$ is a compact topological space.
\end{defn}

Similar to the case of Kuranishi spaces, we define embeddings of derived orbifolds as combinations of several more basic morphisms.

\begin{defn} Let ${\mc D}$ be an $A$-stratified derived orbifold. 

\begin{enumerate}

\item A {\bf morphism of derived orbifolds} from an $A_1$-stratified derived orbifold ${\mc D}_1 = ({\mc U}_1, {\mc E}_1, {\mc S}_1)$ to an $A_2$-stratified derived orbifold ${\mc D}_2 = ({\mc U}_2, {\mc E}_2, {\mc S}_2)$ consists of a morphism $i: A_1 \to A_2$ of $\regpos$, a smooth stratified orbifold map $\phi_{21}: {\mc U}_1 \to {\mc U}_2$, and a smooth orbibundle map $\wh\phi_{21}: {\mc E}_1 \to {\mc E}_2$ covering $\phi_{21}$ such that the following diagram commutes.
\beqn
\xymatrix{ {\mc E}_1 \ar[r]^{\wh\phi_{21}} \ar[d]  & {\mc E}_2 \ar[d] \\
           {\mc U}_1 \ar[r]_{\phi_{21}}  \ar@/^1pc/[u]^{{\mc S}_1} & {\mc U}_2 \ar@/_1pc/[u]_{{\mc S}_2} }
\eeqn
Such a morphism induces an orbispace map ${\mc S}_1^{-1}(0) \to {\mc S}_2^{-1}(0)$. 

\item An {\bf open embedding of derived orbifolds} from ${\mc D}_1$ to ${\mc D}_2$ is a morphism such that the poset map $i: A_1 \to A_2$ is an isomorphism onto an Alexandrov open subset, the orbifold map $\phi_{21}$ is an isomorphism onto an open subset, and the bundle map $\wh\phi_{21}$ is an isomorphism over the image of $\phi_{21}$.

\item Given a derived orbifold ${\mc D} = ({\mc U}, {\mc E}, {\mc S})$, for any open subset ${\mc U}' \subset {\mc U}$ with ${\mc S}^{-1}(0) \subset {\mc U}'$, the {\bf shrinking} of ${\mc D}$ to ${\mc U}'$, denoted by ${\mc D}|_{{\mc U}'}$, is the derived orbifold $({\mc U}', {\mc E}|_{{\mc U}'}, {\mc S}|_{{\mc U}'})$. It admits a natural open embedding into ${\mc D}$.

\item A disk bundle over an orbifold is an open neighborhood of the zero section of an orbifold vector bundle. A {\bf stabilization} of a derived orbifold ${\mc D} = ({\mc U}, {\mc E}, {\mc S})$ by a disk bundle ${\mc N} \to {\mc U}$ contained in a vector bundle $\pi_{\mc F}: {\mc F} \to {\mc U}$ is the derived orbifold 
\beqn
{\rm Stab}_{\mc N} {\mc D} = ({\mc N}, \pi_{\mc N}^* {\mc E} \oplus \pi_{\mc N}^* {\mc F}, \pi_{\mc N}^* {\mc S} \oplus \tau_{\mc F})
\eeqn
where $\tau_{\mc F}$ is the tautological section. We call the natural morphism
\beqn
{\mc D} \to {\rm Stab}_{\mc N}( {\mc D} )
\eeqn
the {\bf stabilization morphism} by ${\mc N}$.

\item An {\bf embedding of derived orbifolds} from ${\mc D}_1 = ({\mc U}_1, {\mc E}_1, {\mc S}_1)$ to ${\mc D}_2 = ({\mc U}_2, {\mc E}_2, {\mc S}_2)$ is a morphism ${\bm \phi}_{21}: {\mc D}_1 \to {\mc D}_2$ such that  there exists a disk bundle ${\mc N}_{21} \to {\mc U}_1$ and an open embedding $\theta_{21}: {\rm Stab}_{{\mc N}_{21}}( {\mc D}_1) \to {\mc D}_2$ extending $\phi_{21}$. 
\end{enumerate}
\end{defn}

\begin{defn}
The category of stratified derived orbifolds, denoted by $\uds{\bf dOrb}$, has objects being stratified derived orbifolds and morphisms being embeddings of derived orbifolds. 
\end{defn}

Then one can easily verify that the category $\uds{\bf dOrb}$ is a regular stratification category. One can see that there are functors
\beqn
\uds{\bf S Kur}  \to \uds{\bf dOrb} \to \uds{\bf Top}.
\eeqn
The first arrow sends $K = (G, V, E, S)$ to the quotient ${\mc D} = (V/G, E/G, S/G)$. The second arrow is induced by passing to the zero locus ${\mc S}^{-1}(0)$ of a derived orbifold ${\mc D} = ({\mc U}, {\mc E}, {\mc S})$.

\subsubsection{Rigidified embeddings}

We would like to introduce an analogue of rigidified embeddings for orbifolds. We need to relax the triviality requirement. A flat orbifold vector bundle is an orbifold vector bundle together with a flat connection. A key property of flat vector bundles is that if ${\mc E} \to {\mc U}$ is flat and ${\mc F} \to {\mc E}$ is flat, then ${\mc F}$ is canonically a flat bundle over ${\mc U}$. %

\begin{defn}\label{defn_rigid_embedding_D}
Let $\phi_{21}: {\mc D}_1 \to {\mc D}_2   $ be an embedding of derived orbifolds. 
\begin{enumerate}

\item A {\bf rigidification} of $\phi_{21}$ consists of a disk bundle ${\mc N}_{21}$ contained in a flat vector bundle ${\mc F}_{21} \to {\mc U}_1$ and an open embedding
\beqn
\psi_{21}: {\rm Stab}_{{\mc N}_{21}}({\mc D}_1) \to {\mc D}_2
\eeqn
which extends $\phi_{21}$. 

\item Two rigidifications $({\mc N}_{21} \subset {\mc F}_{21}, \psi_{21})$, $({\mc N}_{21}'\subset {\mc F}_{21}', \psi_{21}')$ are equivalent if there is a flat bundle isomorphism $\tau: {\mc F}_{21} \cong {\mc F}_{21}'$ which identifies ${\mc N}_{21}$ with ${\mc N}_{21}'$ after an appropriate shrinking such that 
\beqn
\psi_{21}' = \psi_{21} \circ \tau.
\eeqn

\item A {\bf rigidified embedding} of derived orbifolds is an embedding together with an equivalence class of rigidifications. 
\end{enumerate}
\end{defn}

One can see that rigid derived orbifold embeddings can be composed. Therefore we can define a category of derived orbifolds with rigid embeddings as morphisms. This category is denoted by 
\beqn
\uds{\bf dOrb}_{\rm rig}.
\eeqn
Note that there is a functor by taking the orbifold quotient:
\beqn
\uds{\bf SKur}_{\rm rig} \to \uds{\bf dOrb}_{\rm rig}.
\eeqn

\subsubsection{Normal complex structures on derived orbifolds}

Normal complex (NC for short) structures are essential for constructing FOP transverse perturbations. We recall the definitions here. 

We start with a few basic notions. 
\begin{enumerate}

\item Suppose $G$ acts on a set $A$. A subgroup $H \subset G$ is called an $A$-essential subgroup, denoted by 
\beqn
H \subset_A G
\eeqn
if there exists $a \in A$ whose stabilizer is $H$. 

\item Given a real representation $V$ of a finite group $G$, the {\bf basic decomposition} of $V$ is the splitting
\beqn
V = V_G \oplus \check V_G
\eeqn
where $V_G$ is the direct sum of trivial subrepresentations and $\check V_G$ is the direct sum of nontrivial irreducible subrepresentations. An {\bf NC structure} on $V$ is a $G$-invariant complex structure on $\check V_G$. A real representation $V$ of $G$ with an NC structure is called an {\bf NC representation} of $G$.

\item Let $V$ be an NC representation of $G$, its {\bf conjugate} $V^\dagger$ is the same real representation with the reversed NC structure. 

\item An {\bf NC triple} is a triple $(G, V, W)$ where $G$ is a finite group and $V$, $W$ are two NC representations of $G$. The {\bf conjugate} of an NC triple $(G, V, W)$ is the triple $(G, V^\dagger, W^\dagger)$.
\end{enumerate}

\begin{defn}
Let $U$ be a $G$-manifold and $E \to U$ be a $G$-equivariant vector bundle. An {\bf NC structure} on $E$ consists of, for each $U$-essential subgroup $H \subset_U G$, a $H$-invariant complex structure $I^{\check E_H}$ on the subbundle
\beqn
\check E_H \subset E|_{U_H}
\eeqn
satisfying the following conditions.
\begin{enumerate}
    \item For each $g\in G$ which conjugates $H$ to $H':= g H g^{-1}$, the bundle isomorphism $\check E_H \to \check E_{g H g^{-1}}$ sends $I^{\check E_H}$ to $I^{\check E_{H'}}$.

    \item For each pair of $U$-essential subgroups $K \subset H \subset G$, one has the $K$-equivariant decomposition  
    \beqn
    \check E_H \cong \check E_K \oplus (\check E_H \cap E_K|_{U_H}).
    \eeqn
    We require that the restriction of $I^{\check E_H}$ to $\check E_K$ coincides with $I^{\check E_K}$.
\end{enumerate}

\end{defn}

\begin{defn}
Let ${\mc U}$ be an effective orbifold and ${\mc E} \to {\mc U}$ be an orbifold vector bundle. An {\bf NC structure} on ${\mc E}$ consists of, for each bundle chart $\hat C = (G, U, E)$ of ${\mc E}$, an NC structure ${\bm I}^E$ on the vector bundle $E$, such that chart embeddings respect those complex structures. An {\bf NC structure} on ${\mc U}$ is by definition an NC structure on $T{\mc U}$.

A {\bf normally complex derived orbifold} is a derived orbifold ${\mc D} = ({\mc U}, {\mc E}, {\mc S})$ together with NC structures on ${\mc U}$ and ${\mc E}$.
\end{defn}

\begin{defn}[Embeddings of NC derived orbifolds] \hfill
\begin{enumerate}

\item An {\bf open embedding of normally complex derived orbifolds} from ${\mc D}_1$ to ${\mc D}_2$ is an open embedding of derived orbifolds which respects normal complex structures. 

\item A stabilization of an NC derived orbifold is a stabilization by a disk bundle which is contained in a complex orbifold vector bundle.

\item An embedding of NC derived orbifolds $\iota_{21}: {\mc D}_1 \to {\mc D}_2$ is an embedding such that there exists a stabilization of ${\mc D}_1$ by a disk bundle ${\mc N}_{21} \to {\mc U}_1$ and an open embedding from ${\rm Stab}_{{\mc N}_{21}} ({\mc D}_1) \to {\mc D}_2$ which extends $\iota_{21}$. 

\item A {\bf rigidified embedding} of NC derived orbifolds from ${\mc D}_1$ to ${\mc D}_2$ is a rigidified embedding whose rigidification can be represented by a disk bundle in a rigid complex vector bundle. Two rigidifications are equivalent if they (after appropriate shrinking to smaller disk bundles) differ by an isomorphism of rigid complex vector bundles. 
\end{enumerate}
\end{defn}

Moreover, one can define an involution, still denoted by $\dagger$, on $\uds{\bf dOrb}^{\rm NC}$ by reversing the NC structures on both the base and the obstruction bundle. Therefore, $(\uds{\bf dOrb}^{\rm NC}, \dagger)$ is an involutive regular stratification category. If we forget the NC structure, then the corresponding involution on $\uds{\bf dOrb}$ is simply the identity. 

The following statement is straightforward to check. 

\begin{lemma}
There are natural involutive functors 
\begin{align*}
&\ (\uds{\bf SKur}^{\rm NC}, \dagger) \to (\uds{\bf dOrb}^{{\rm NC}}, \dagger),\ &\ (\uds{\bf SKur}_{\rm rig }^{\rm NC}, \dagger) \to (\uds{\bf dOrb}^{{\rm NC}}_{\rm rig}, \dagger)
\end{align*}
\end{lemma}

\subsection{Main theorems about regularization}

Now we can state the main technical result of this paper concerning constructing regularizations of the moduli spaces.

\begin{thm}\label{thm_flow_category_lift}
The outercollaring of the Tate--Floer flow category ${\mb F}$ admits a $\Pi \times {\mb Z}$-equivariant involutive lift, denoted by $\tilde {\mb F}$, to $(\outer \uds{\bf dOrb}_{\rm rig}^{\rm NC}, \dagger)$. We call it an involutive {\bf AMS lift}. 
\end{thm}

\begin{thm}\label{thm_PSS_lift}
Given an involutive AMS lift $\tilde {\mb F}$ of the outercollaring of the Tate--Floer flow category provided by Theorem \ref{thm_flow_category_lift}, there exists a $\Pi \times {\mb Z}$-equivariant and involutive lift of the outercollaring of the Tate--PSS bimodule ${\mb B}^{\rm PSS}$ resp. the Tate--SSP bimodule, denoted by $\wt{\mb B}^{\rm PSS}$ resp. $\wt{\mb B}^{\rm SSP}$ to $\outer \uds{\bf dOrb}_{\rm rig}^{\rm NC}$, which is an equivariant involutive bimodule over $(\outer {\mb M}; \tilde {\mb F})$ resp. over $(\tilde {\mb F}; \outer {\mb M})$.
\end{thm}

\begin{thm}\label{thm_pearly_lift}
There is a $\Pi \times {\mb Z}$-equivariant involutive lift of the outercollaring of the pearly bimodule $\outer \mb{B}^{\rm pearly}$ to $\outer \uds{\bf dOrb}_{\rm rig}^{\rm NC}$ as a bimodule over $(\outer {\mb M}; \outer {\mb M})$.
\end{thm}

\begin{thm}\label{thm_homotopy_lift}
Given lifts as provided by Theorem \ref{thm_flow_category_lift}, Theorem \ref{thm_PSS_lift}, Theorem \ref{thm_pearly_lift}, there exists a $\Pi\times {\mb Z}$-equivariant involutive concatenation $\wt{\mb B}^{\rm PSS} \circ \wt{\mb B}^{\rm SSP}$ as a lift of the canonical concatenation $\outer {\mb B}^{\rm PSS} \circ \outer {\mb B}^{\rm SSP}$. Moreover, there exists a $\Pi\times {\mb Z}$-equivariant involutive lift of the outercollaring of the homotopy ${\mb H}$ as a homotopy from $\wt{\mb B}^{\rm PSS} \circ \wt{\mb B}^{\rm SSP}$ to $\tilde {\mb B}^{\rm pearly}$.
\end{thm}

Their proofs occupy Sections \ref{section6}--\ref{section7}. 

\subsection{FOP perturbations and involutions}

To prove the Arnold--Givental conjecture, a key step is to apply FOP perturbations (cf. \cite{Bai_Xu_2022}) in an involution-symmetric way. The notion of FOP transversality was originally proposed by Fukaya--Ono \cite{Fukaya_Ono_integer} and further studied by Parker \cite{BParker_integer}. In \cite{Bai_Xu_2022} the rigorous foundation was established. We recall the ``blackbox'' theorem regarding such perturbations.

\begin{thm}\label{thm:FOP}\cite[Main Theorem]{Bai_Xu_2022} 
Let ${\mc U}$ be an NC orbifold and ${\mc E} \to {\mc U}$ be an NC orbifold vector bundle. Then there exists a $C^0$-dense subset
\beqn
\Gamma^{\rm FOP}({\mc U}, {\mc E}) \subset \Gamma({\mc U}, {\mc E})
\eeqn
whose elements are called {\bf FOP transverse sections} satisfying the following conditions.
\begin{enumerate}
    \item {\bf (Locality)} If ${\mc U}' \subset {\mc U}$ is an open subset, then the restriction of an FOP transverse section over ${\mc U}$ to ${\mc U}'$ is also FOP transverse.

    \item {\bf (Extension)} If $C \subset {\mc U}$ is a closed subset and ${\mc S}_0$ is an FOP transverse section defined near $C$, then there exists an FOP transverse section ${\mc S}$ over ${\mc U}$ which agrees with ${\mc S}_0$ near $C$.

    \item {\bf (Stabilization)} If ${\mc S}$ is an FOP transverse section of ${\mc E}$ and $\pi_{\mc F}: {\mc F} \to {\mc U}$ is a flat complex orbifold vector bundle, then the stabilization 
    \beqn
    \pi_{\mc F}^* {\mc S} \oplus \tau_{\mc F}: {\mc F} \to \pi_{\mc F}^* {\mc E} \oplus \pi_{\mc F}^* {\mc F}
    \eeqn
    is an FOP transverse section.

    \item {\bf (Product)} If ${\mc S}_i: {\mc U}_i \to {\mc E}_i$, $i = 1, 2$ are FOP transverse sections, then the product ${\mc S}_1 \times {\mc S}_2$ is also an FOP transverse section.

    \item {\bf (Classical Transversality)} If ${\mc U}$ is a manifold, then ${\mc S}$ is FOP transverse if and only if it is transverse in the classical sense. 

    \item {\bf (Pseudocycle)} For any FOP transverse section ${\mc S}: {\mc U} \to {\mc E}$, the closure 
    \beqn
    \ov{ {\mc S}^{-1}(0) \cap {\mc U}_{\rm mfd}} \subset {\mc U}
    \eeqn
    is a pseudomanifold of dimension ${\rm dim} {\mc U} - {\rm rank} {\mc E}$.
\end{enumerate}
\end{thm}

Then there is a regular stratification category $\uds{\bf dOrb}^{\rm FOP}$ introduced in \cite[Section 8]{Bai_Xu_foundation} whose objects are quadruples $({\mc U}, {\mc E}, {\mc S}_0, {\mc S})$ where $({\mc U}, {\mc E}, {\mc S}_0)$ is an NC derived orbifold and ${\mc S}: {\mc U} \to {\mc E}$ is an FOP transverse perturbation of ${\mc S}_0$. There are also collared and rigidified versions. Moreover, there are two natural functors
\beq\label{FOP_diagram}
\vcenter{ \xymatrix{  &   \uds{\bf dOrb}^{\rm FOP} \ar[ldd]_-{({\mc U}, {\mc E}, {\mc S}_0, {\mc S}) \mapsto ({\mc U}, {\mc E}, {\mc S}_0)} \ar[rdd]^-{({\mc U}, {\mc E}, {\mc S}_0, {\mc S}) \mapsto \ov{ {\mc S}^{-1}(0) \cap {\mc U}_{\rm free}} }  & \\
& & \\
\uds{\bf dOrb}^{\rm NC}  & &   \pman} }
\eeq
between regular stratification categories. Both of them can be upgraded to the collared and rigidified versions. 

Recall that both $\uds{\bf dOrb}^{\rm NC}$ and $\pman$ are involutive, where the former is equipped with the conjugation action on the normal complex structure, while the latter is equipped with the trivial involution. We need an extension of the FOP transversality arguments in the involutive-symmetric setting. The following important result will be proven momentarily.

\begin{thm}\label{thm_FOP_involution}
Let ${\mc U}^\dagger$ resp. ${\mc E}^\dagger$ be equipped with the opposite NC structure. Then as subspaces of $\Gamma({\mc U}, {\mc E})$, one has 
\beqn
\Gamma^{\rm FOP}({\mc U}, {\mc E}) = \Gamma^{\rm FOP}({\mc U}^\dagger, {\mc E}^\dagger).
\eeqn
Hence the functor $({\mc U}, {\mc E}, {\mc S}_0, {\mc S}) \mapsto ({\mc U}^\dagger, {\mc E}^\dagger, {\mc S}_0, {\mc S})$ is an involution on $\uds{\bf dOrb}^{\rm FOP}$ and the functors in \eqref{FOP_diagram} are both involutive.
\end{thm}

\subsubsection{Proof of Theorem \ref{thm_FOP_involution}}

In fact one only needs to prove this involution property for a local model. Recall that an {\bf NC triple} is a triple $(G, V, W)$ where $G$ is a finite group and $V$, $W$ are two faithful and normally complex representations of $G$. We have their basic decompositions
\begin{align*}
    &\ V = V_G \oplus \check V_G,\ &\ W = W_G \oplus \check W_G.
\end{align*}
One has a special space of polynomial maps
\beqn
\wh{\rm Poly}{}_G^d(V, W) = {\rm Poly}_G^d(\check V_G, \check W_G) \oplus W_G.
\eeqn
Its elements can be viewed as $G$-equivariant maps $P: V \to W$ which are independent of the $V_G$-direction and whose $W_G$-components are constants. Hence we regard
\beqn
\wh{\rm Poly}{}_G^d(V, W) \subset C_G^\infty(V, W).
\eeqn
Notice that the ambient space $C_G^\infty(V, W)$ does not depend on the NC structure. The space of {\bf NC maps} is the $C_G^\infty(V, {\mb R})$-submodule
\beqn
C_G^{\rm NC}(V, W):= \varinjlim_{d} C_G^\infty(V, {\mb R}) \Big( \wh{\rm Poly}{}_G^d(V, W) \Big) \subset C_G^\infty(V, W).
\eeqn

We need a particular kind of presentation of an NC map. Fix a sufficiently large $d$. The {\bf graph space} associated to $(G, V, W)$ and $d$ is 
\beqn
M^d(G, V, W):= V \times \wh{\rm Poly}{}_G^d(V, W) \subset V \times C_G^\infty(V, W)=: M^{\rm smooth}(G, V, W).
\eeqn
There is a natural evaluation map 
\beqn
\ev: M^d (G, V, W) \to W,\ (v, P) \mapsto P (v)
\eeqn
whose zero locus is denoted by
\beqn
Z^d(G, V, W).
\eeqn
By a main technical result of \cite[Section 3.2]{Bai_Xu_2022}, the set $Z^d(G, V, W)$ admits a canonical Whitney stratification ${\mf Z}^d(G, V, W)$, which is determined by the following criterion. Notice that $G$ acts nontrivially on the $V$ factor of $M^d(G, V, W)$.

\begin{defn}
The {\bf canonical Whitney stratification} on $Z^d(G, V, W)$ is the minimal (which is also unique) one among all Whitney stratifications subject to the following condition. For each $V$-essential subgroup $H\subset_V G$, let
\beqn
M^d(G, V, W)_H^* = \Big\{ (v, P) \ |\ G_v = H \Big\}.
\eeqn
Then each stratum of $Z^d(G, V, W)$ is contained in $M^d(G, V, W)_H^*$ for some $H \subset_V G$. 
\end{defn}

In \cite{Bai_Xu_2022} the first and the fourth authors proved the existence of canonical Whitney stratifications (following methods of Parker \cite{BParker_integer}) along with their properties.

\begin{defn}
Let $f \in C_G^{\rm NC}(V, W)$ be an NC map.
\begin{enumerate}

\item A {\bf lift} of $f$ is a $G$-invariant smooth map
\beqn
{\mf p}: V \to \wh{\rm Poly}{}_G^d(V, W)
\eeqn
such that 
\beqn
f(v) = {\mf p}(v)(v) = \ev(v, {\mf p}(v))
\eeqn

\item $f$ is said to be {\bf FOP transverse} if it admits a lift ${\mf p}$ for a sufficiently large $d$, such that the graph of ${\mf p}$, which is 
\beqn
{\rm graph}({\mf p}):= \Big\{(v, {\mf p}(v))\ |\ v \in V \Big\} \subset M^d(G, V, W)
\eeqn
is transverse to each stratum of the canonical Whitney stratification of $Z^d(G, V, W)$.
\end{enumerate}
\end{defn}

Now we start to prove that after flipping the NC structure, the notion of FOP transversality is preserved. The first observation is that being complex is preserved by flipping twice.

\begin{lemma}\label{lemma532}
As subspaces of $C_G^\infty(V, W)$, one has
\beqn
\wh{\rm Poly}{}_G^d(V, W) = \wh{\rm Poly}{}_G^d(V^\dagger, W^\dagger).
\eeqn
\end{lemma}

\begin{proof}
It suffices to consider the case when $V_G = W_G = 0$. Hence $V$ and $W$ are complex representations. Then $P \in {\rm Poly}{}_G^d(V, W)$ if and only if $P$ is a real equivariant polynomial map of degree $\leq d$ and is holomorphic. Since $P: V \to W$ is holomorphic if and only if $P: V^\dagger \to W^\dagger$ is holomorphic, the two spaces agree. 
\end{proof}

As a consequence, one has 
\beq\label{NC_map_invariance}
C_G^{\rm NC}(V, W) = C_G^{\rm NC}(V^\dagger, W^\dagger)
\eeq
and 
\begin{align*}
&\ M^d(G, V, W) = M^d(G, V^\dagger, W^\dagger),\ &\ Z^d(G, V, W) = Z^d(G, V^\dagger, W^\dagger).
\end{align*}

The next result follows from the properties of the canonical Whitney stratification and justifies the name ``canonical.''

\begin{lemma}\label{lemma_stratification_invariance}
As stratifications on the same set $Z^d(G, V, W)$, one has
\beqn
{\mf Z}^d(G, V, W) = {\mf Z}^d(G, V^\dagger, W^\dagger).
\eeqn
\end{lemma}

\begin{proof}
The canonical Whitney stratification is the minimal Whitney stratification on $Z^d(G, V, W)$ which respects the $G$-action on $M^d(G, V, W)$. The minimality is a concept in $C^\infty$ category, which does not depend on the complex analytic structure (cf. \cite[Appendix B]{Bai_Xu_foundation}). In particular, since the pullback of the canonical Whitney stratification preserves the minimality and G-invariance, it actually agrees with the canonical Whitney stratification. 
\end{proof}

\begin{cor}
Given $f\in C_G^\infty(V, W)$, it is FOP transverse as an NC map from $V$ to $W$ if and only if it is FOP transverse as an NC map from $V^\dagger$ to $W^\dagger$.
\end{cor}

\begin{proof}
By Lemma \ref{lemma532} and \eqref{NC_map_invariance}, $f\in C_G^{\rm NC}(V, W)$ if and only if $f \in C_G^{\rm NC}(V^\dagger, W^\dagger)$. Moreover, if ${\mf p}: V \to \wh{\rm Poly}{}_G^d(V, W)$ is a lift of $f$, ${\mf p}$ can also be regarded as a map ${\mf p}: V \to \wh{\rm Poly}{}_G^d(V^\dagger, W^\dagger)$. Their graphs are the same submanifold of $M^d(G, V, W)$. By Lemma \ref{lemma_stratification_invariance}, its transversality against ${\mf Z}^d(G, V, W)$ is equivalent to its transversality against ${\mf Z}^d(G, V^\dagger, W^\dagger)$. 
\end{proof}

This finishes the proof of Theorem \ref{thm_FOP_involution}.

\subsection{FOP perturbations on flow categories, bimodules, and homotopies}

Now we state and prove the involutive version of \cite[Theorem N]{Bai_Xu_foundation} on the basis of Theorem \ref{thm_flow_category_lift}, Theorem \ref{thm_PSS_lift}, Theorem \ref{thm_pearly_lift}, and Theorem \ref{thm_homotopy_lift}.

\begin{thm}\label{thm_FOP_lift}
\begin{enumerate}

\item Let $\tilde {\mb F}$ be an involutive AMS lift of the outercollaring of ${\mb F}$ provided in Theorem \ref{thm_flow_category_lift}. Then there is a $\Pi \times {\mb Z}$-equivariant involutive lift of $\tilde {\mb F}$ to $\outer \uds{\bf dOrb}_{\rm rig}^{\rm FOP}$, denoted by $\mathring {\mb F}$. 

\item Notice that $\outer {\mb M}$ is a lift of itself to $\outer \uds{\bf dOrb}_{\rm rig}^{\rm FOP}$. Given $\mathring {\mb F}$ as above and an AMS lift $\tilde {\mb B}^{\rm PSS}$,  there exists a $\Pi \times {\mb Z}$-equivariant involutive lift   $\mathring{\mb B}^{\rm PSS}$ to $\outer \uds{\bf dOrb}_{\rm rig}^{\rm FOP}$ as a bimodule over $(\outer {\mb M}; \mathring {\mb F})$. The same remains true for the SSP bimodule.

\item There is a $\Pi \times {\mb Z}$-equivariant involutive lift of $\tilde{\mb B}^{\rm pearly}$ to $\outer \uds{\bf dOrb}_{\rm rig}^{\rm NC}$, denoted by $\mathring {\mb B}^{\rm pearly}$.

\item The involutive concatenation $\tilde {\mb B}^{\rm PSS} \circ \tilde {\mb B}^{\rm SSP}$ provided by Theorem \ref{thm_homotopy_lift} admits a $\Pi \times {\mb Z}$-equivariant involutive lift to $\outer \uds{\bf dOrb}_{\rm rig}^{\rm FOP}$ which is also a concatenation $ \mathring {\mb B}^{\rm PSS} \circ \mathring {\mb B}^{\rm SSP}$. 

\item There exists a $\Pi \times {\mb Z}$-equivariant involutive lift of $\tilde {\mb H}$ to $\outer \uds{\bf dOrb}_{\rm rig}^{\rm NC}$, written as $\mathring{\mb H}$, as a homotopy from $\mathring {\mb B}^{\rm PSS} \circ \mathring {\mb B}^{\rm SSP}$ to $\mathring {\mb B}^{\rm pearly}$. 
\end{enumerate}
\end{thm}

\begin{proof}
The content of Theorem \ref{thm_FOP_lift} is the existence of coherent FOP perturbations. The proof follows the induction of the proof of \cite[Theorem N]{Bai_Xu_foundation}. We do induction on both the energy ${\mc A}_H(\gamma_p) - {\mc A}_H(\gamma_q)$ and the difference $f_{\rm Tate}(\mu_p) - f_{\rm Tate}(\mu_q)$. Each step of the induction, we take care of a whole $\Pi \times {\mb Z} \times {\mb Z}/2$-orbit of pairs of objects $(p, q)$. As the involution is free, the construction on half of pairs in the same orbit (which is a $\Pi \times {\mb Z}$-orbit) can be done independently (while maintaining the $\Pi \times {\mb Z}$-symmetry). For the other half, one uses the involution to translate the perturbations directly. The key result, Theorem \ref{thm_FOP_involution}, guarantees that the translations via involution remain FOP transverse. 
\end{proof}

\subsection{Proof of the Arnold--Givental conjecture}\label{subsec:official-proof}

Now we can prove Theorem \ref{thm:main}. We explain how to rigorously define the objects introduced in Section \ref{sec:proof} using the statements from above. The strategy is to apply the functor 
\beqn
\outer \uds{\bf dOrb}_{\rm rig}^{\rm FOP} \to \pman
\eeqn
(defined by taking the closure of the isotropy-free part of the zero locus of the perturbation) to the flow categories etc. provided by Theorem \ref{thm_FOP_lift}. One then obtains flow categories, bimodules, homotopies enriched in $\pman$. Then, using the general abstract formulation of \cite[Section 7]{Bai_Xu_foundation}, we obtain the following objects.

\begin{itemize}
\item We can construct the chain complex 
\beqn
\wt{CF}(H)
\eeqn
of $\Lambda^{\Pi \times {\mb Z}}$-modules from $\mathring {\mb F}$ using the results in \cite[Section 7.4.1]{Bai_Xu_foundation};
\item Using the bimodules $\mathring{\mb B}^{\rm PSS}$ and $\mathring{\mb B}^{\rm SSP}$ and the construction in \cite[Section 7.4.1]{Bai_Xu_foundation}, we obtain chain maps
\begin{align*}
&\ \wt\Phi^{\rm PSS}: \wt{CM}(f_X) \to \wt{CF}(H),\ &\ \wt\Phi^{\rm SSP}: \wt{CF}(H) \to \wt{CM}(f_X).
\end{align*}
\item Using the bimodule $\mathring {\mb B}^{\rm pearly}$ and \cite[Section 7.4.1]{Bai_Xu_foundation}, we can define a chain map
$$\wt\Phi^{\rm pearly}: \wt{CM}(f_X) \to \wt{CM}(f_X).$$
\item Using the bimodule homotopy $\mathring{\mb H}$ and the results in \cite[Section 7.4.2]{Bai_Xu_foundation}, we obtain a chain homotopy between $\wt\Phi^{\rm SSP} \circ \wt\Phi^{\rm PSS}$ and $\wt\Phi^{\rm pearly}$ as chain maps of $\wt{CM}(f_X)$. 
\end{itemize}
As the constructions are invariant under the ${\mb Z}/2$-action, we can take the ${\mb Z}/2$-invariant parts. These chain-level objects satisfy the same properties as the case discussed in Section \ref{sec:proof} where transversality was assumed. In particular, the involution symmetry induces a ${\mb Z}/2$-action on $\wt{CF}(H)$. Define
\beqn
\wh{CF}(H):= \wt{CF}(H)_{{\mb Z}/2} \otimes \Lambda_{\mc K}
\eeqn
and call the resulting cohomology the {\bf Tate-Floer cohomology} of $H$. The chain maps $\wt\Phi^{\rm PSS}$, $\wt\Phi^{\rm SSP}$, and $\wt\Phi^{\rm pearly}$ all induce $\Lambda_{\mc K}$-linear maps between Tate cohomology groups
\beqn
\xymatrix{  \wh{HM}(f_X) \ar[r]^{\wh\Phi^{\rm PSS}}   \ar@/_2.0pc/[rr]_{\wh\Phi^{\rm pearly}} &     \wh{HF}(H) \ar[r]^{\wh\Phi^{\rm SSP}} &   \wh{HM}(f_X)}
\eeqn
after taking the ${\mb Z}/2$-invariant component. They are called the {\bf Tate-PSS}, {\bf Tate-SSP}, {\bf Tate-pearly} maps. Using the chain homotopy induced by $\mathring{\mb H}$, we see that the above diagram commutes.

Moreover, Lemma \ref{lemma:pearl} holds following the reasoning in the proof of \cite[Theorem 25.12]{Bai_Xu_foundation}, which is a consequence of the fact that, for moduli spaces that are already transversely cut out, the count induced by FOP perturbations agrees with the classical count (cf. Theorem \ref{thm:FOP} (5)). So the algebraic fact Proposition \ref{prop:lower_bound}, which asserts that ${\rm dim}_{\Lambda_{\mc K}} \Big( \wh{HF}(H) \Big)  \geq {\rm dim}_{\Lambda_{\mc K}} \Big( \wh{HM}(f_X) \Big),$ holds.

Finally, the homological perturbation argument in Section \ref{subsec:conclude} carries over because we only used energy considerations in Section \ref{sec:loc-equiv} and in the proof of Proposition \ref{prop_perturbation}. This allows us to deduce
\[\dim_{\Lambda_{\cl K}} \wh{HF}(H) \leq \dim_{\Lambda_{\cl K}} \wh{CF}^{\rm red}(H) = \#(\cl O(H)^{\Z/2}) \leq \#(\phi(L) \cap L).\]
So the proof concludes using the chain of inequalities \eqref{eqn:inequality}. \qed

\section{Construction of Global Charts}\label{section6}

The goal of this section is to work toward proving Theorem \ref{thm_flow_category_lift}, Theorem \ref{thm_PSS_lift}, Theorem \ref{thm_pearly_lift}, and Theorem \ref{thm_homotopy_lift} by constructing suitable global Kuranishi charts of the moduli spaces. Due to the evident resemblance with the discussions in \cite[Section]{Bai_Xu_foundation}, we will only highlight the new ingredients in the involutive-symmetric setting. We will lift the outercollaring of the Tate-Floer flow category ${\mb F}$ to an intermediate category $\outer \uds{\bf S^{\rm rel} Kur}_{\rm rig}$, called the category of relatively smooth Kuranishi spaces (collared and rigidified). This notion was already defined in \cite{Bai_Xu_foundation}. We will recall that in due course.

\subsection{The category of prestable cylinders} 
 
\begin{defn}\label{defn_curve_1}
An {\bf equivariant family of prestable cylinders} is a triple ${\mc C} = ({\mc G}, B, C)$ where ${\mc G}$ is a (not necessarily reductive) complex Lie group containing a maximal compact subgroup $G\subset {\mc G}$, $B$ is a stratified smooth ${\mc G}$-manifold, $C$ is a stratified ${\mc G}$-space, with a ${\mc G}$-equivariant map $C \to B$, and the structure of a prestable cylinder on each fiber $C_\phi \subset C$ (see Definition \ref{defn_cylinder}). This triple needs to satisfy the following conditions.
\begin{enumerate}
    \item For any $g \in {\mc G}$ and $\phi \in B$, the map $g: C_\phi \to C_{g\phi}$ is an isomorphism of prestable cylinders.

    \item Let $\mathring C \subset C$ be the complement of fiberwise nodes. Then $\mathring C$ is a stratified ${\mc G}$-manifold and the map $\mathring C \to B$ is a submersion.
\end{enumerate}
The {\bf conjugate} of ${\mc C} = ({\mc G}, B, C)$ is the triple ${\mc C}^\dagger = ({\mc G}^\dagger, B^\dagger, C^\dagger)$ with ${\mc G}^\dagger$ being the same Lie group as ${\mc G}$ with the reversed complex structure (with the same maximal compact subgroup $G$), $B^\dagger = B$, and $C^\dagger$ being the same space as $C$ with fibrewise structure of prestable cylinders on $C_\phi$ being conjugated to $C_\phi^\dagger$ (see Definition \ref{defn_cylinder_conjugate}).
\end{defn}

We can also add additional structures to an equivariant family of prestable cylinders, such as a collection of marked points, which are precisely ${\mc G}$-equivariant sections of $C \to B$ avoiding nodes and each other. We recall the definition of morphisms of the category of equivariant families of prestable cylinders. 

\begin{defn}\label{defn_curve_2}(\cite[Definition 14.7]{Bai_Xu_foundation}) Let ${\mc C} = ({\mc G}, B, C)$ be an equivariant family of prestable cylinders.
\begin{enumerate}

\item A {\bf group enlargement} of ${\mc C} = ({\mc G}, B, C)$ is an equivariant family of prestable cylinders ${\mc C}' = ({\mc G}', B', C')$ where ${\mc G} \to {\mc G}'$ is a complex Lie group embedding which sends the compact group $G$ into $G'$ with $B' = {\mc G}' \times_{\mc G} B$ and $C' = {\mc G}' \times_{\mc G} C$. We denote the group enlargement as such by 
\beqn
{\mc G}'\times_{\mc G} {\mc C}.
\eeqn

\item A {\bf strict morphism} from ${\mc C}_1 = ({\mc G}_1, B_1, C_1)$ to ${\mc C}_2 = ({\mc G}_2, B_2, C_2)$, denoted by $\zeta_{21}: {\mc C}_1 \to {\mc C}_2$, consists of a complex Lie group embedding $\zeta_{21}^{{\mc G}}: {\mc G}_1 \to {\mc G}_2$ which restricts to a group homomorphism $\zeta_{21}^G: G_1 \to G_2$ and an equivariant commutative diagram
\beqn
\xymatrix{ C_1 \ar[r]^{\zeta_{21}^C}  \ar[d] & C_2 \ar[d] \\
 B_1 \ar[r]_{\zeta_{21}^B} & B_2 }
\eeqn
satisfying 1) the map $\zeta_{21}^B: B_1 \to B_2$ is a stratified smooth map; 2) the map $\zeta_{21}^C: C_1 \to C_2$ is continuous such that the restricted map $\mathring C_1 \to \mathring C_2$ is smooth and such that for each $\phi_1 \in B_1$ sent to $\phi_2 \in B_2$, the induced map $\mathring C_{\phi_1} \to \mathring C_{\phi_2}$ is an isomorphism of prestable cylinders. Notice that a strict morphism ${\mc C}_1 \to {\mc C}_2$ induces a strict morphism from ${\mc G}_2 \times_{{\mc G}_1} {\mc C}_1$ into ${\mc C}_2$. 

\item A strict morphism is called a {\bf strict embedding} if the induced map ${\mc G}_2\times_{{\mc G}_1} {\mc C}_1 \to {\mc C}_2$ is an isomorphism.



\item The category of equivariant families of cylinders, denoted by $\uds{\bf Cylinder}$, is the category whose objects are equivariant families of cylinders and whose morphisms are strict embeddings.

\end{enumerate}
\end{defn}

The category $\uds{\bf Cylinder}$ is monoidal: given two objects ${\mc C}_i = ({\mc G}_i, B_i, C_i)$, the product ${\mc C}_1 \times {\mc C}_2$ is $({\mc G}_1 \times {\mc G}_2, B_1 \times B_2, C_1 \sqcup C_2)$ where the fiber of $C_1 \sqcup C_2$ over $(\phi_1, \phi_2) \in B_1\times B_2$ is the disjoint union $C_{\phi_1}\sqcup C_{\phi_2}$. The initial object, which is $\emptyset$-stratified (see Definition \ref{defn_regular_stratification_category}), is the object with ${\mc G}$ being the trivial group, $B$ being the singleton, and $C$ being empty. 

Notice that the category $\uds{\bf Cylinder}$ also allows outercollaring. Moreover, the conjugation action ${\mc C} \mapsto {\mc C}^\dagger$ specified in Definition \ref{defn_curve_1} defines an involution on $\uds{\bf Cylinder}$, making it an involutive regular stratification category.

\subsubsection{Rigidifications}

Similar to the cases of $\uds{\bf dOrb}$ or $\uds{\bf Kur}$, there are also rigidified versions of the category $\uds{\bf Cylinder}$. We recall its definition from \cite{Bai_Xu_foundation}.

\begin{defn}(\cite[Definition 14.9]{Bai_Xu_foundation}) \label{defn_curve_rigidification}
Let $\zeta_{21}: ({\mc G}_1, B_1, C_1) \to ({\mc G}_2, B_2, C_2)$ be a strict embedding in $\uds{\bf Cylinder}$. A {\bf rigidification} of $\zeta_{21}$ consists of an orthogonal representation $Q_{21}$ of $G_1$ and
a germ of $\zeta_{21}^G$-equivariant maps
\beqn
\theta_{21}: B_1 \times Q_{21}^\epsilon \to B_2
\eeqn
(where $G_1$ acts diagonally on the domain) whose restriction to $B_1 \times \{0\}$ coincides with $\zeta_{21}^B$, such that the map
\beqn
\begin{split}
G_2 \times_{G_1} ( B_1 \times Q_{21}^\epsilon) \to &\ B_2,\\
[g_2, (b_1, \eta_{21})] \mapsto &\ g_2 \theta_{21} (b_1, \eta_{21})
\end{split}
\eeqn
is a diffeomorphism onto an open neighborhood of the image of $\zeta_{21}^B$. A {\bf strict rigidified embedding} is a strict embedding together with a rigidification.
\end{defn}

\begin{lemma}(\cite[Lemma 14.10]{Bai_Xu_foundation})
Strict rigidified embeddings can be composed.
\end{lemma}

\begin{defn}
$\uds{\bf Cylinder}_{\rm rig}$ has morphisms being strict rigidified embeddings.
\end{defn}

Notice that we can naturally define the collared version of $\uds{\bf Cylinder}_{\rm rig}$, denoted by 
\beqn
\outer \uds{\bf Cylinder}_{\rm rig}.
\eeqn
The conjugation defines an involution on $\outer \uds{\bf Cylinder}_{\rm rig}$.

\subsubsection{Stable complex structure}

We recall the definition of $\uds{\bf Cylinder}^{\mb C}$ of equivariant families of prestable cylinders whose bases have equivariant stable complex structures given in \cite{Bai_Xu_foundation}. 

\begin{defn}(\cite[Definition 14.12]{Bai_Xu_foundation}) 
A {\bf stable complex structure} on an equivariant family of prestable cylinders ${\mc C} = ({\mc G}, B, C)$ is a $G$-equivariant stable complex structure of the tangent bundle $TB$ (see Definition \ref{defn_stable_complex_structure}). Recall that it consists of a $G$-equivariant complex vector bundle $I \to B$ and a $G$-equivariant stable isomorphism
    \beqn
    \tau: TB \overset{s}{\to} I.
    \eeqn
\end{defn}

To define morphisms of stably complex objects, note the following fact. Let $\zeta_{21}: ({\mc G}_1, B_1, C_1) \to ({\mc G}_2, B_2, C_2)$ be a strict embedding of equivariant families of curves. Notice that there are ${\mc G}_2$-equivariant vector bundles
    \begin{align*}
    &\ {\mc G}_2\times_{{\mc G}_1} TB_1,\ &\ \pi^* T({\mc G}_2/{\mc G}_1)
    \end{align*}
    where $\pi: {\mc G}_2\times_{{\mc G}_1} B_1 \to {\mc G}_2/{\mc G}_1$ temporarily denotes the natural projection. 
    Via $\zeta_{21}$, one can see that there is a  ${\mc G}_2$-equivariant exact sequence
    \beqn
    \xymatrix{  0 \ar[r] & {\mc G}_2 \times_{{\mc G}_1} TB_1 \ar[r] &  TB_2 \ar[r] &  \pi^* T({\mc G}_2/{\mc G}_1) \ar[r] & 0     }.
    \eeqn
    It naturally splits along the embedding image of $\zeta_{21}^B$ via the ${\mc G}_2$-action. Hence one can make the following definition.

\begin{defn}
A strict embedding of equivariant families of prestable cylinders with stably complex structures from $({\mc G}_1, B_1, C_1, I_1, \tau_1)$ to $({\mc G}_2, B_2, C_2, I_2, \tau_2)$ is a strict embedding $\zeta_{21}$ from $({\mc G}_1, B_1, C_1)$ to $({\mc G}_2, B_2, C_2)$ such that the natural bundle isomorphism 
\beqn
TB_2|_{\zeta_{21}^B(B_1)} \cong \zeta_{21}^B (TB_1) \oplus \pi^* T({\mc G}_2/{\mc G}_1)
\eeqn
respects the stable complex structures. The category $\uds{\bf Cylinder}^{\mb C}$ consists of objects being equivariant families of curves with stable complex structures whose morphisms are unitary conjugacy classes of strict embeddings. 
\end{defn}

One can also define rigidifications for strict embeddings in $\uds{\bf Cylinder}^{\mb C}$. The details are provided in \cite[Section 14.2]{Bai_Xu_foundation} and hence omitted here. This provides a regular stratification category $\uds{\bf Cylinder}_{\rm rig}^{\mb C}$. 

Notice that the category $\outer\uds{\bf Cylinder}^{\mb C}$ admits a natural involution. For any object ${\mc C} = ({\mc G}, B, C)$, we simply take ${\mc C}^\dagger$ to be $({\mc G}^\dagger, B^\dagger, C^\dagger)$ where the underlying real objects are the same, while ${\mc G}^\dagger$ has the opposite complex structure, $B^\dagger$ has the opposite stable complex structure, and $C^\dagger$ has the opposite fibrewise complex structure (with opposite parametrization of the cylindrical components).

\subsection{Construction of the domain flow categories}

We will construct a flow category denoted by $\mb{Dom}$ enriched in $\outer\uds{\bf Cylinder}_{\rm rig}^{\mb C}$ with objects being the same as the Tate-Floer flow category ${\mb F}$. 

Recall that using the first and fourth authors' approach to the integral Arnold conjecture, specifically \cite[Proof of Theorem A]{Bai_Xu_Arnold} and the discussions after \cite[Hypothesis 4.1]{Bai_Xu_Arnold}, we can reduce the proof to the case when the symplectic form $\omega$ represents an integral cohomology class and the Hamiltonian function $H_t$ satisfies that all critical values of its action functional are integers. Of course, one can also appeal to \cite[Proposition 15.4]{Bai_Xu_foundation} or \cite[Lemma 16]{Rezchikov_Arnold} to integralize the action so that the proof works in general. 

Henceforth, without loss of generality, we assume that for each capped 1-periodic orbit $\gamma$ of $H$, the symplectic action ${\mc A}_H(\gamma)$ is an integer. Then for each pair of objects $p = (\gamma_p, \mu_p)$, $q = (\gamma_q, \mu_q)$, denote
\beqn
d_{pq} =  {\mc A}_H(\gamma_p) - {\mc A}_H(\gamma_q) \in {\mb Z}.
\eeqn

Also, note that the constructions just need to be carried out for regularizing the ``curve" part of the moduli spaces in the Tate--Floer flow category, since  the ``Morse" part of the moduli spaces are already regular. Therefore, in the following, we will be focused on explaining how to construct the global Kuranishi charts for the stable Floer trajectories, even though it should be thought of as performing the construction in the coupled setting.

\subsubsection{Some Lie groups and homogeneous spaces}

We will consider variants of $U(d)$ and $GL(d)$. Inside $PGL(d+1)$, there is the subgroup
\beqn
{\mc G}_d: = \left\{ \left[ \begin{array}{cc} * & 0 \\ *  & 1 \end{array}\right] \in PGL(d+1) \right\}
\eeqn
The unitary group $G_d:= U(d)$ embeds into ${\mc G}_d$ by 
\beqn
U(d) \ni g \mapsto \left[ \begin{array}{cc} g & 0 \\ 0 & 1 \end{array} \right] \in {\mc G}_d.
\eeqn
One needs to consider certain specific group embeddings for different $d$'s. The embedding
\beqn
G_{d_1} \times G_{d_2} = U(d_1) \times U(d_2) \to G_{d_1 + d_2} = U(d_1 + d_2)
\eeqn
is obvious by putting the nontrivial blocks together. For the corresponding complex Lie group ${\mc G}_d$, there are corresponding embeddings. For example, for $d_1 = d_2 = 2$, the embedding ${\mc G}_2 \times {\mc G}_2 \to {\mc G}_4$ reads 
\beqn
\left( \left[ \begin{array}{ccc} a_{11} & a_{12} & 0 \\ a_{21} & a_{22} & 0 \\ a_{31} & a_{32} & 1 \end{array}\right], \left[ \begin{array}{ccc} b_{11} & b_{12} & 0 \\ b_{21} & b_{22} & 0 \\  b_{31} & b_{32} & 1 \end{array}\right] \right) \mapsto \left[ \begin{array}{ccccc} a_{11} & a_{12} & 0 & 0 & 0 \\ a_{21} & a_{22} & 0 & 0 & 0 \\
0 & 0 & b_{11} & b_{12} & 0 \\ 0 & 0 & b_{21} & b_{22} & 0 \\ a_{31} & a_{32} & b_{31} & b_{32} & 1 \end{array}\right].
\eeqn

If we denote by ${\mc G}_d^\dagger$ the same Lie group as ${\mc G}_d$ with the opposite complex structure, then the standard complex conjugation provides a holomorphic isomorphism
\beqn
{\mc G}_d^\dagger \cong {\mc G}_d
\eeqn
which respects the map ${\mc G}_{d_1}\times {\mc G}_{d_2} \to {\mc G}_{d_1 + d_2}$. 

\subsubsection{Moduli spaces of cylinders in projective spaces}

We describe a system of moduli spaces of curves in projective spaces coupled with Tate flow lines. This is similar to the situation considered in equivariant Floer theory in \cite[Section 25]{Bai_Xu_foundation}. 

Fix a pair of objects $p, q \in {\rm Ob}{\mb F}$. Recall that $p = (\gamma_p, \mu_p)$ where $\gamma_p$ is a capped 1-periodic orbit  of $H$ and $\mu_p$ is a critical point of $f_{\rm Tate}$; similar for $q$. We describe an object ${\mc C}_{pq} = ({\mc G}_{pq}, B_{pq}, C_{pq})$ in $\uds{\bf Cylinder}$. Denote
\beqn
{\mc G}_{pq}:= {\mc G}_{d_{pq}}.
\eeqn
Consider the moduli space
\beqn
{\mc M}_{pq}(\mb{CP}^{d_{pq}})
\eeqn
of pairs $(y, u)$ where $y: {\mb R} \to \wh S^\infty$ is a negative gradient flow line from $\mu_p$ to $\mu_q$ and $u: {\mb R}\times S^1 \to \mb{CP}^{d_{pq}}$ is a holomorphic map of degree $d_{pq}$, modulo simultaneous translations. This space has a natural compactification by allowing (synchronized) breakings and sphere bubbling. Denote the compactification by 
\beqn
\ov{\mc M}_{pq}(\mb{CP}^{d_{pq}}).
\eeqn
As ${\mc G}_{pq}$ acts on $\mb{CP}^{d_{pq}}$, this compactified moduli space has an induced action. Let 
\beqn
\ov{\mc M}_{pq}^+ ({\mb P}^{d_{pq}}) \subset 
\ov{\mc M}_{pq}(\mb{CP}^{d_{pq}})
\eeqn
be the subset satisfying the constraint:
\beq\label{positive_constraint}
u (z_+) = [0, \ldots, 0, 1]
\eeq
which is invariant under ${\mc G}_{d_{pq}}$, where $z_+ = \{\infty\} \in \mathbb{P}^1 \supseteq {\mb R}\times S^1$. Let 
\beqn
B_{pq} \subset \ov{\mc M}_{pq}^+ (\mb{CP}^{d_{pq}})
\eeqn
be the open subset corresponding to pairs $(y, u)$ where the image of $u$ is not contained in any hyperplane of $\mb{CP}^{d_{pq}}$. This is a ${\mc G}_{pq}$-invariant subset. It is easy to calculate the dimension of $B_{pq}$, which is 
\beqn
{\rm dim}_{\mb R} B_{pq} = 2d_{pq} (d_{pq} +1) + {\rm index}(\mu_p) - {\rm index}(\mu_q) - 1.
\eeqn

Further, let $C_{pq} \to B_{pq}$ be the universal curve; notice that fibers also carry a canonical structure of prestable cylinders. The ${\mc G}_{pq}$-action relates different fibres on the same orbit in $B_{pq}$ via isomorphism of prestable cylinders. Moreover, the moduli space $B_{pq}$ is stratified by the poset $A_{pq}^{\mb F}$. Therefore, one obtains an object
\beqn
{\mc C}_{pq} = ({\mc G}_{pq}, B_{pq}, C_{pq}) \in {\rm Ob} \uds{\bf Cylinder}.
\eeqn

\subsubsection{Involution} 

Recall that $\uds{\bf Cylinder}$ is equipped with a canonical involution. For each pair $(p, q)$, we can define a canonical isomorphism in $\uds{\bf Cylinder}$ 
\beqn
\tau_{pq}^{\mc C}: ({\mc C}_{pq})^\dagger \to {\mc C}_{p^\dagger q^\dagger}.
\eeqn
Of course we define ${\mc G}_{pq}^\dagger \to {\mc G}_{p^\dagger q^\dagger}$ to be the complex conjugation. On the other hand, we can treat $B_{pq}$ as a subset of a Tate-Floer moduli space with target being $\mb{CP}^{d_{pq}}$ and Hamiltonian being zero. The standard complex structure is anti-symmetric with respect to the standard complex conjugation on $\mb{CP}^{d_{pq}}$. Hence using the homeomorphism \eqref{Tate-Floer_involution}, one obtains a homeomorphism
\beq\label{base_involution}
B_{pq} \cong B_{p^\dagger q^\dagger}.
\eeq
One can then check that this is equivariant over ${\mc G}_{pq}^\dagger \cong {\mc G}_{p^\dagger q^\dagger}$. Moreover, the corresponding fibres in the universal curves are inverted as well.

\subsubsection{Concatenation of cylinders}

In \cite{Bai_Xu_Arnold} \cite{Bai_Xu_foundation} details have been provided for constructing ``concatenations,'' namely, the structural maps required for a flow category enriched in $\uds{\bf Cylinder}$. We will not repeat the details but only state the result.

\begin{prop}
There exist codimension 1 morphisms
\beqn
\zeta_{prq}: \outer {\mc C}_{pr}\times \outer {\mc C}_{rq} \to \outer {\mc C}_{pq}
\eeqn
of $\outer \uds{\bf Cylinder}_{\rm rig}$ such that together with the objects ${\mc C}_{pq}$, they form a $\Pi \times {\mb Z}$-equivariant involutive Novikov flow category $\mb{Dom}$ enriched in $\outer \uds{\bf Cylinder}_{\rm rig}$. 
\end{prop}

\subsubsection{Stable complex structure}

As the first step towards constructing the stable complex structure for the AMS construction, we have the following ``natural'' stable complex structure on the monotone flow category of curves. 

\begin{prop}\label{prop_domain_stable_complex}
The flow category $\mb{Dom}$ (enriched in $\outer \uds{\bf Cylinder}_{\rm rig}$) admits a natural involutive lift to $\outer \uds{\bf Cylinder}_{\rm rig}^{\mb{C}}$.
\end{prop}

\begin{proof}
For the same reason as in \cite{Bai_Xu_foundation}, the stable complex structure is pulled back from the complex structure of the smooth part of the stable map moduli $\ov{\mc M}{}_{0,2}(\mb{CP}^d, d)$. One can easily see that such complex structures are invariant under involutions as the complex conjugation on $\mb{CP}^d$ and sphere involution also induce an involution on $\ov{\mc M}{}_{0,2}(\mb{CP}^d, d)$ which reverses the complex structure. 
\end{proof}

\subsection{The AMS thickening}

\subsubsection{Framed maps}

We briefly recall the definitions about framed maps and group reductions defined in \cite[Section 15.2]{Bai_Xu_foundation}. We do not intend to consider the general case but will work directly over the flow category $\mb{Dom}$ whose objects are the same as those of the Tate-Floer flow category. For each pair of objects $p, q \in {\rm Ob}{\mb F}$, one can consider the topological space
\beqn
{\rm Map}_{pq}:= \Big\{ (\phi, u)\ |\ \phi \in B_{pq}, u: C_\phi \to X \ {\rm satisfies}\ (\ast) \Big\};
\eeqn
here $(\ast)$ means the map $u$ is componentwise smooth, converges exponentially fast (with all derivatives) to the correct 1-periodic orbits at cylindrical nodes and satisfies the correct topological condition. Then for each $(\phi, u)\in {\rm Map}_{pq}$, the integral symplectic form $\omega$ is pulled back to a 2-form $\Omega_{\phi, u} \in \Omega^2(C_\phi)$ whose componentwise integrals are integers and which decays exponentially near cylindrical ends. Hence there is a unique Hermitian holomorphic line bundle $L_{\phi, u} \to C_\phi$ (up to isomorphism) whose curvature form is $-2\pi {\bf i} \Omega_{\phi, u}$. As $u$ satisfies the ``correct'' topological condition, the line bundle is positive exactly on components of $C_\phi$ with positive degrees in $\mb{CP}^{d_{pq}}$. Hence $H^0(L_{\phi, u})$ has a fixed dimension $d_{pq} + 1$. A {\bf framing} on $(\phi, u)$ is a framing $F = (f_0, \ldots, f_{d_{pq}})$ of $H^0(L_{\phi, u})$ satisfying
\beqn
f_0(z_+) = \cdots = f_{d_{pq}-1}(z_+) = 0.
\eeqn
Define $Y_{pq}$ to be the set of triples $(\phi, u, F)$ such that $(\phi, u) \in {\rm Map}_{pq}$ and $F$ is a framing on $(\phi, u)$. 

Notice that each holomorphic framing $F$ also induces a holomorphic map $\phi_F: C_\phi \to \mb{CP}^{d_{pq}}$. Indeed this induces a map
\beqn
Y_{pq} \to B_{pq}.
\eeqn

Notice that ${\mc G}_{pq}$ acts on $Y_{pq}$ in two different ways: the first is to reparametrize the map via the ${\mc G}_{pq}$-action on $B_{pq}$; the second is to transform the framing linearly. Denote the first by ${\mc G}_{pq}^L$ and the second by ${\mc G}_{pq}^R$. Notice that the second action is free, hence $Y_{pq} \to {\rm Map}_{pq}$ is a ${\mc G}_{pq}^L$-equivariant principal ${\mc G}_{pq}^R$-bundle.

The spaces $Y_{pq}$ admits an involution symmetry. We define a bijection 
\beqn
Y_{pq} \cong Y_{p^\dagger q^\dagger}
\eeqn
as follows. Recall by \eqref{base_involution}, one has a correspondence 
\beqn
B_{pq} \ni \phi \mapsto \phi^\dagger \in B_{p^\dagger q^\dagger}
\eeqn
as well as an anti-holomorphic identification $C_\phi \cong C_{\phi^\dagger}$. Moreover, the involution $\tau_X$ induces a map
\beqn
u^\dagger: C_{\phi^\dagger} \to X.
\eeqn
One hence has a conjugate-linear isomorphism
\beqn
\xymatrix{   L_{\phi, u} \ar[r] \ar[d]  &  L_{\phi^\dagger, u^\dagger} \ar[d]\\
             C_\phi \ar[r] &  C_{\phi^\dagger} }
\eeqn
inducing a conjugate linear isomorphism $H^0(L_{\phi, u}) \cong H^0(L_{\phi^\dagger, u^\dagger})$. Therefore, the framing $F$ over $(\phi, u)$ is also transferred to a framing $F^\dagger$. Together with the group-level involution, such a correspondence $Y_{pq} \cong Y_{p^\dagger q^\dagger}$ is also equivariant with respect to the group isomorphism ${\mc G}_{pq}^L \times {\mc G}_{pq}^R \cong {\mc G}_{p^\dagger q^\dagger}^L \times {\mc G}_{p^\dagger q^\dagger}^R$. 

A group reduction over $Y_{pq}$ is a continuous equivariant map 
\beqn
Y_{pq} \to {\mc G}_{pq}^R/ G_{pq}^R
\eeqn
which is ${\mc G}_{pq}^L$-invariant and ${\mc G}_{pq}^R$-equivariant. We would like to make the group reduction as part of the Kuranish section, hence one needs to ``linearize'' the target by identifying the quotient ${\mc G}_{pq}^R/G_{pq}^R$ with an orthogonal representation of $G_{pq}$ denoted by $Q_{pq}$. The choices we need to make include the group reductions as well as the identifications ${\mc G}_{pq}^R/G_{pq}^R \cong Q_{pq}$. Notice that there are natural maps
\begin{align*}
&\ {\mc G}_{pr}^R/ G_{pr}^R \times {\mc G}_{rq}^R/ G_{rq}^R \to {\mc G}_{pq}^R/ G_{pq}^R,\ &\ Q_{pr} \oplus Q_{rq} \to Q_{pq}
\end{align*}
and linearizations ${\mc G}_{pq}^R/G_{pq}^R \cong Q_{pq}$ compatible with these natural maps.

\begin{prop}\label{prop_group_reduction}
There exists a collection of linearized group reductions
\beqn
\lambda_{pq}: Y_{pq} \to Q_{pq}
\eeqn
satisfying the following conditions.
\begin{enumerate}
\item $\lambda_{pq}$ is collared.

\item $\lambda_{pq}$ respects rigidified embeddings (see precise meaning of this property in \cite[Definition 15.7]{Bai_Xu_foundation}).

\item $\lambda_{pq}$ respects structural maps among $Y_{pq}$. 

\item $\lambda_{pq}$ respects the involution. Namely there are isomorphisms
\beqn
Q_{pq} \cong Q_{p^\dagger q^\dagger}
\eeqn
equivariant with respect to the group isomorphism (the complex conjugation) $G_{pq} \cong G_{p^\dagger q^\dagger}$ such that the diagram commutes.
\beqn
\xymatrix{ Y_{pq} \ar[r]^{\lambda_{pq}} \ar[d]  &  Q_{pq}  \ar[d] \\
           Y_{p^\dagger q^\dagger} \ar[r]_{\lambda_{p^\dagger q^\dagger}}    &  Q_{p^\dagger q^\dagger}   }
\eeqn
\end{enumerate}
\end{prop}

\begin{proof}
\cite[Proposition 15.9]{Bai_Xu_foundation} treated the non-involutive case, based on an induction. The involutive case is an easy extension: at each inductive step, we treat a ${\mb Z}/2$-orbit of pairs of objects $(p, q) \neq (p^\dagger, q^\dagger)$ at the same time. 
\end{proof}

\subsubsection{Involution-symmetric thickening data}

Fix an equivariant family ${\mc C} = ({\mc G}, B, C)$ and the almost complex manifold $(X, J)$. Recall $\mathring C \subset C$ is the complement of nodes and markings. Then one has a complex vector bundle
\beqn
\Lambda^{0,1}_{\mathring C/ B} \otimes TX \to \mathring C \times X.
\eeqn 
Notice that we can regard this complex vector bundle as a real subbundle 
\beqn
\Lambda^{0,1}_{\mathring C/ B} \otimes TX \subset \Lambda^1_{\mathring C /B} \otimes TX = {\rm Hom}_{\mb R}( T\mathring C/B, TX)
\eeqn
over $\mathring C \times X$ where the latter does not depend on either the domain complex structure or the target almost complex structure. Then one can see that $d\tau_X$ together with the domain conjugation induces an isomorphism
\beqn
\xymatrix{ \Lambda_{\mathring C/B}^{0,1}\otimes TX \ar[d] \ar[r]^{d\tau_X}  &       \Lambda_{\mathring C^\dagger/B^\dagger}^{0,1}\otimes TX \ar[d]\\
\mathring C \times X  \ar[r]   &   \mathring C^\dagger \times X }.
\eeqn
Notice that this bundle isomorphism induces a complex-linear isomorphism between spaces of sections with compact supports
\beqn
\Gamma_c( \mathring C \times X, \Lambda_{\mathring C/B}^{0,1} \otimes TX) \cong \Gamma_c( \mathring C^\dagger \times X, \Lambda_{\mathring C^\dagger/B^\dagger}^{0,1} \otimes TX).
\eeqn

\begin{defn}\label{defn_thickening_datum}
A {\bf thickening datum} on an equivariant family of curves ${\mc C} = ({\mc G}, B, C)$ consists of a pair $(W, \nu)$ where $W$ is a finite-dimensional complex unitary representation\footnote{We restrict to complex representations in order to simplify the discussion of normal complex structures.} of the maximal compact subgroup $G \subset {\mc G}$ and 
\beqn
\nu: W \to \Gamma_c( \mathring C \times X, \Lambda^{0,1}_{\mathring C/B} \otimes TX)
\eeqn
is a $G$-equivariant complex linear map. Here $\Gamma_c$ stands for smooth sections with compact support.
\end{defn}

Notice that a thickening datum $(W, \nu)$ over ${\mc C} = ({\mc G}, B, C)$ induces a {\bf conjugate} $(W^\dagger, \nu^\dagger)$ described as follows. $W^\dagger$ has the same underlying space as $W$ with reversed complex structure, and 
\begin{equation}\label{eqn:thicken-conjugate}
\begin{split}
\nu^\dagger: W^\dagger \to &\ \Gamma_c( \mathring C^\dagger \times X, \Lambda_{\mathring C^\dagger/ B^\dagger}^{0,1}\otimes TX),\\
\nu^\dagger(e)(z, x) = &\ d\tau_X(\nu(e)(z, \tau_X(x))).
\end{split}
\end{equation}

Notice that given a thickening datum $(W, \nu)$ on ${\mc C}$, for each $\phi \in B$ and a smooth map $u: C_\phi \to X$, there is an induced complex linear map
\beqn
\nu_{\phi, u}: W \to \Omega^{0,1}_c( \mathring C_\phi, u^* TX)
\eeqn
given by restricting $\nu (e)$ to the graph of $u$ in $\mathring C_\phi \times X$.

\begin{defn}
The {\bf category of curves-with-thickening data}, denoted by $\uds{\bf Thick}$ is defined as follows. The objects are tuples
\beqn
({\mc G}, B, C, W, \nu )
\eeqn
where $({\mc G}, B, C)$ is an equivariant family of curves and $(W, \nu )$ is a thickening datum on ${\mc C}$. The morphisms are {\bf strict embeddings}, which are defined as follows. A strict embedding consists of a strict embedding 
\beqn
\zeta_{21} = (\zeta_{21}^{\mc G}, \zeta_{21}^B, \zeta_{21}^C): ({\mc G}_1, B_1, C_1) \to ({\mc G}_2, B_2, C_2)
\eeqn
(see Definition \ref{defn_curve_2}) and a  $\zeta_{21}^G$-equivariant isometric linear embedding $\zeta_{21}^W: W_1 \to W_2$ such that 1) the following diagram commutes;
\beqn
\xymatrix{    W_1 \ar[r]^-{\nu_1} \ar[d]_{\zeta_{21}^W}  &  \Gamma_c( \mathring C_1 \times X, \Lambda_{\mathring C_1/B_1}^{0,1} \otimes TX)  \ar[d]\\
  W_2 \ar[r]_-{\nu_2}  &   \Gamma_c( \mathring C_2 \times X, \Lambda_{\mathring C_2/B_2}^{0,1} \otimes TX) }
\eeqn
(here the right vertical arrow is induced from the morphism $\zeta_{21}$) and 2) the restriction of $\nu_2$ to the orthogonal complement $W_1^\bot \subset W_2$ vanishes. 
\end{defn}

The category $\uds{\bf Thick}$ is a regular stratification category with a forgetful functor $\uds{\bf Thick} \to \uds{\bf Cylinder}$. Moreover, there is an involution defined as follows. Given an object $({\mc G}, B, C, W, \nu) \in {\rm Ob}\uds{\bf Thick}$, define
\beqn
({\mc G}, B, C, W, \nu)^\dagger = ({\mc G}^\dagger, B^\dagger, C^\dagger, W^\dagger, \nu^\dagger)
\eeqn
where ${\mc G}^\dagger$, $B^\dagger$, and $C^\dagger$ are defined previously. In addition, here $W^\dagger$ has the same underlying space as $W$ with the reversed complex structure (and the same real inner product) and $\nu^\dagger$ is defined as in \eqref{eqn:thicken-conjugate}. 

One can also define the collared and rigidified versions of the category $\uds{\bf Thick}$, denoted by $\outer \uds{\bf Thick}_{\rm rig}$. We do not repeat the details which can be found in \cite[Section 15.3]{Bai_Xu_foundation}. Notice that $\outer \uds{\bf Thick}_{\rm rig}$ is still an involutive regular stratification category.

\begin{defn}
A {\bf thickening datum} on a flow category/bimodule/homotopy enriched in $\uds{\bf Cylinder}$ is a lift in $\uds{\bf Thick}$. If the flow category/bimodule/homotopy is enriched in $\outer \uds{\bf Cylinder}_{\rm rig}$, then a perturbation is called {\bf collared and rigidified} if it is a lift in $\outer\uds{\bf Thick}_{\rm rig}$.
\end{defn}

For the domain flow category $\mb{Dom}$ enriched in $\outer \uds{\bf Cylinder}$, one can also consider an involutive thickening datum.

\begin{defn}
A thickening datum on $\mb{Dom}$ is called {\bf involutive} if for each pair of objects $p, q \in {\rm Ob}{\mb F}$, one has the isomorphism of $\uds{\bf Thick}$
\beqn
({\mc G}_{pq}^\dagger, B_{pq}^\dagger, C_{pq}^\dagger, W_{pq}^\dagger, \nu_{pq}^\dagger) \cong ({\mc G}_{p^\dagger q^\dagger}, B_{p^\dagger q^\dagger}, C_{p^\dagger q^\dagger}, W_{p^\dagger q^\dagger}, \nu_{p^\dagger q^\dagger} )
\eeqn
where the underlying isomorphism ${\mc C}_{pq}^\dagger \cong {\mc C}_{p^\dagger q^\dagger}$ is the one constructed before. These isomorphisms need to respect structural maps and $\Pi \times {\mb Z}$-actions.
\end{defn}

Now we can describe the AMS lift of the Floer flow categories, bimodules, and homotopies to the category of topological Kuranishi spaces.

\subsubsection{The thickenings}

We describe the thickened moduli spaces as candidates of the Kuranishi regularizations of the moduli spaces. We spell out the detail in the flow category case. Return to the domain flow category $\outer \mb{Dom} $ enriched in $\outer\uds{\bf Cylinder}_{\rm rig}$ whose morphism spaces are equivariant families of curves ${\mc C}_{pq} $. An involutive thickening datum of $\outer \mb{Dom} $ is denoted by $\mb{W} $ which consists of thickening data
\beqn
\nu_{pq} : W_{pq} \to \Gamma_c ( \mathring C_{pq}  \times X, \Lambda_{\mathring C_{pq} / B_{pq} }^{0,1}\otimes TX).
\eeqn

\begin{defn}\label{defn_thickened_moduli}
Given any collared, rigidified, and involutive thickening datum $\mb{W}$ on the flow category $\mb{Dom}$ and group reductions, for each pair $p, q$ of ${\mb F}$, define the following items.

\begin{enumerate}

\item The {\bf thickened moduli space} is
\beqn
V_{pq}: = \left\{ (\phi, u, F, e)\ \left|\ \begin{array}{c} (\phi, u, F) \in Y_{pq},\ e \in W_{pq},\\
\ov\partial_H u + \nu_{pq} (e) = 0,\ \phi = \phi_F \in \outer B_{pq}. \end{array} \right. \right\}.
\eeqn

\item The $G_{pq}$-action on $V_{pq}$ is the product of the diagonal $G_{pq}$-action on $Y_{pq}^{\epsilon}$ space of framed maps and the action on $W_{pq}$.

\item The {\bf obstruction bundle} is the trivial bundle
\beqn
E_{pq}:= \uds Q_{pq} \oplus 
\uds W_{pq} \to V_{pq}.\footnote{The piece $Q_{pq}$ is not yet complex. However, we will further stabilize by another copy of it.}
\eeqn

\item The {\bf Kuranishi section} is (given by) the equivariant map
\beqn
S_{pq} (\phi, u,  F, e) = (\lambda_{pq} (\phi, u, F), e)
\eeqn
where $\lambda_{pq}$ is the linearized group reduction provided by Proposition \ref{prop_group_reduction}.

\end{enumerate}
\end{defn}

The following statement is \cite[Theorem 15.15]{Bai_Xu_foundation}.

\begin{thm}\label{thm_footprint}
The natural ``footprint'' map
\beqn
 (S_{pq})^{-1}(0)/ G_{pq} \to \outer M_{pq}^{\mb F}
\eeqn
which sends the orbit of $(\phi, u, F, 0)$ to the equivalence class of the stable Tate-Floer trajectory is an isomorphism of orbispaces.
\end{thm}

In our setting, the above abstract quadruples $(G_{pq}, V_{pq}, E_{pq}, S_{pq})$ also admit involutions. 

\begin{lemma}
There are equivariant isomorphisms 
\beqn
\xymatrix{        E_{pq}  \ar[r] \ar[d]   &  E_{p^\dagger q^\dagger} \ar[d] \\
V_{pq}  \ar[r]  \ar@/^1.0pc/[u]^{S_{pq}} &  V_{p^\dagger q^\dagger} \ar@/_1.0pc/[u]_{S_{p^\dagger q^\dagger}} }
\eeqn
\end{lemma}

\begin{proof}
Given a point $(\phi, u, F, e) \in V_{pq}$ where $\phi \in B_{pq}$, one has a biholomorphism $C_\phi^\dagger \cong C_{\phi^\dagger}$ where $\phi^\dagger \in B_{p^\dagger q^\dagger}$. The involution on $Y_{pq}$ induces $u^\dagger$ which fits into the diagram
\beqn
\xymatrix{  C_\phi \ar[r]^{u} \ar[d] &  X \ar[d]^{\tau_X} \\
            C_{\phi^\dagger} \ar[r]_{u^\dagger}  &  X }.
\eeqn
The framing $F$ also has a natural conjugate $F^\dagger$ over $(\phi^\dagger, u^\dagger)$. Notice that if $\phi_F = \phi$, then it also holds
\beqn
\phi_{F^\dagger} = \phi^\dagger \in B_{p^\dagger q^\dagger}.
\eeqn
Lastly, $e\in W_{pq}$ corresponds canonically to $e^\dagger \in W_{p^\dagger q^\dagger}$ via the involution $W_{pq}^\dagger \cong W_{p^\dagger q^\dagger}$ contained in the data. 

We can check that $(\phi^\dagger, u^\dagger, F^\dagger, e^\dagger)\in V_{p^\dagger q^\dagger}$. Indeed, at any $z \in C_\phi$ which is sent to $z^\dagger \in C_{\phi^\dagger}$, one has 
\beqn
\Big( \ov\partial_{H} u^\dagger + \nu_{p^\dagger q^\dagger}(e^\dagger) \Big)(z^\dagger)  = d\tau_X (\ov\partial_H u (z)) + d\tau_X (\nu_{pq}(e)(z) ).
\eeqn
Hence such an involution is well-defined. One can verify the equivariance easily. 
\end{proof}

\subsubsection{Transversality}

Next we deal with transversality. For this we need to set up a certain Fredholm problem. For each pair of objects $p, q$, define
\beqn
V_{pq}':= \Big\{ (\phi, u, e) |\ \phi \in \outer B_{pq}, e \in W_{pq}, u: C_\phi \to X \text{ satisfies }\ov\partial_H u + \nu_{pq}^{\mb F} (e) = 0 \Big\}.
\eeqn

\begin{lemma}\cite[Lemma 15.17]{Bai_Xu_foundation}
There is a canonical $G_{pq}$-equivariant isomorphism $V_{pq}' \cong V_{pq}$.
\end{lemma}

For a fixed $\phi$, let $V_{pq}'(\phi) \subset V_{pq}'$ be the preimage of $\phi$. Then $V_{pq}'(\phi)$ is the zero locus of a smooth Fredholm section defined on a certain Banach manifold of maps from $C_\phi$ to $X$; the concrete choice of the Banach manifold depends on certain Sobolev parameters. 

\begin{defn}
We say that $(\phi, u, e) \in V_{pq}'(\phi)$ is {\bf transverse} if the linearization of the corresponding Fredholm section is surjective (which is independent of the Sobolev parameters). We say that the perturbation $\mb{W}$ is {\bf transverse} if all points of $V_{pq}'(\phi)$ for all $p,q$ are transverse. 
\end{defn}

We can choose the thickening data appropriately so that transversality holds.

\begin{prop}\label{prop_perturbation} There exists an involutive and  transverse thickening datum $\mb{W}$ on the flow category $ \mb{Dom}$, i.e., a transverse lift to $\outer \uds{\bf Thick}_{\rm rig}$.
\end{prop} 

\begin{proof}
For the non-involutive case, the construction is provided in the proof of \cite[Proposition 15.19]{Bai_Xu_foundation}. In the involutive case, this can also be done in the same inductive procedure, however each time for a ${\mb Z}/2$-orbit of pairs of objects. The point is that transversality at any point in the ${\mb Z}/2$-orbit implies transversality at every other point in the same orbit.
\end{proof}

Transversality implies the regularity of the thickened moduli spaces. 

\begin{prop}\label{prop_topological_lift}
Suppose we choose transverse thickening data as claimed in Proposition \ref{prop_perturbation}. Then for each $p, q \in {\rm Ob}{\mb F}$, the thickened moduli space $V_{pq}$ is an $A_{pq}^{\mb F}$-stratified topological $G_{pq}$-manifold near $S_{pq}^{-1}(0)$. As a consequence, one obtains a $\Pi \times {\mb Z}$-equivariant involutive lift of the outercollaring of the Tate-Floer flow category to $\outer \uds{\bf Kur}_{\rm rig}$, which is denoted by $\hat{\mb F}$.
\end{prop}

The construction also gives automatically an involutive lift to a finer category 
\beqn
\outer \uds{\bf S^{\rm rel} Kur}_{\rm rig},
\eeqn
the collared and rigidified version of the category of relatively smooth Kuranishi spaces. Indeed, the map 
\beqn
\pi: V_{pq} \to B_{pq},\ (\phi, u, F, e) \mapsto \phi
\eeqn
has smooth fibres whose smooth structures vary continuously over the base (see relevant notions defined in detail in \cite{Swaminathan_relative}, see also \cite{Hirschi_Swaminathan}). This category is involutive with respect to the identity functor. See also detailed discussions in \cite[Section 15.4]{Bai_Xu_foundation}. We also denote this lift by $\hat{\mb F}$.

\subsection{The cases of PSS, SSP bimodules etc.}

We will not go through all the details about the global chart construction of the PSS/SSP bimodule, pearly bimodule, and the homotopy moduli spaces, which are essential for proving Theorem \ref{thm_PSS_lift}, Theorem \ref{thm_pearly_lift}, and Theorem \ref{thm_homotopy_lift}. In general, the details just resemble those in \cite[Section 12, Section 25]{Bai_Xu_foundation} where the involved free action changes from one which preserves all structures to an involution which reverses many structures. The key feature that allows the extension is that involution symmetry can be imposed everywhere. We only state our results below.

\begin{thm}\label{thm_module_topological_lift}
Let $\hat{\mb F}$ be the lift of the outercollaring of the Tate-Floer flow category to $\outer\uds{\bf S^{\rm rel} Kur}_{\rm rig}$. Notice that the outercollaring of the Tate-Morse flow category $\outer {\mb M}$ is also its own lift to $\outer \uds{\bf S^{\rm rel} Kur}_{\rm rig}$. 
\begin{enumerate}

    \item There exists a  $\Pi \times {\mb Z}$-equivariant involutive lift $\hat {\mb B}^{\rm PSS}$ resp. $\hat {\mb B}^{\rm SSP}$ of the outercollaring of the Tate-PSS bimodule resp. Tate-SSP bimodule to $\outer \uds{\bf S^{\rm rel} Kur}_{\rm rig}$ as a bimodule over $(\outer {\mb M}; \hat{\mb F})$ resp. over $(\hat {\mb F}; \outer {\mb M})$. 

    \item Given $\hat {\mb B}^{\rm PSS}$ and $\hat{\mb B}^{\rm SSP}$ as above, there exists a $\Pi \times {\mb Z}$-equivariant involutive concatenation, denoted by $\hat{\mb B}^{\rm PSS} \circ \hat{\mb B}^{\rm SSP}$. 

    \item There exists a $\Pi \times {\mb Z}$-equivariant involutive lift $\hat {\mb B}^{\rm pearly}$ of the outercollaring of the Tate-pearly bimodule to $\outer \uds{\bf S^{\rm rel} Kur}_{\rm rig}$ as a bimodule over $(\outer {\mb M}; \outer {\mb M})$. 

    \item Given the constructions above, there exists a $\Pi \times {\mb Z}$ involutive lift $\hat{\mb H}$ of the outercollaring of the homotopy ${\mb H}$ to $\outer \uds{\bf S^{\rm rel} Kur}_{\rm rig}$ as a homotopy from $\hat {\mb B}^{\rm pearly}$ to $\hat{\mb B}^{\rm PSS} \circ \hat {\mb B}^{\rm SSP}$. 

\end{enumerate}
\end{thm}

\section{Smoothing and NC structures}\label{sec:nc-smoothing}\label{section7}
In this final section, we complete the proof of Theorem \ref{thm_PSS_lift}, Theorem \ref{thm_pearly_lift}, and Theorem \ref{thm_homotopy_lift} by discussing how to construct normal complex structures and smoothings of the relevant moduli spaces.

To this end, we introduce yet another new category 
\beqn
\uds{\bf S^{\rm rel} Kur}^{\rm NC}
\eeqn
whose objects are relative smooth Kuranishi spaces (cf. \cite[Section 15.4.1]{Bai_Xu_foundation}) $K = (G, V/B, E, S)$ with vertical NC structures and stable complex structures on $TB$. This category admits a natural involution, which takes the real Lie group $G$ to itself, $V/B$ to itself with the opposite vertical NC structure and opposite stable complex structure on $TB$, denoted by $V^\dagger/B^\dagger$, and which takes $E$ to $E^\dagger$ equipped with the opposite complex structure. Notice that there is an involutive functor
\beqn
(\uds{\bf S^{\rm rel}Kur}^{\rm NC},\dagger) \to (\uds{\bf S^{\rm rel} Kur}^{\rm NC}, \dagger).
\eeqn
There are also the collared and rigidified versions of the new category and the above functor, which we do not write explicitly here. 

Recall that we have obtained a lift of the outercollaring of the involutive Floer category to $\outer \uds{\bf S^{\rm rel} Kur}_{\rm rig}$. Below is the main technical result of this section generalizing \cite[Theorem 16.1]{Bai_Xu_foundation}. 

\begin{thm}\label{thm71}
There exists a $\Pi\times {\mb Z}$-equivariant involutive lift of $\hat {\mb F}$ to $\outer \uds{\bf S^{\rm rel} Kur}_{\rm rig}^{\rm NC}$.
\end{thm}

\begin{proof}
See Subsection \ref{subsection173}.
\end{proof}

We do not state corresponding theorems about the cases of PSS/SSP bimodules, the pearly bimodule, and the homotopy. The method provided below for the flow category case together with the detailed treatment of \cite[Section 16]{Bai_Xu_foundation} are sufficient for proving the corresponding results in the involutive case.

\subsection{The situation for a single moduli space}
We work towards the proof of Theorem \ref{thm71}.

We first adapt the construction of \cite[Section 11]{Abouzaid_Blumberg} in the case of a single moduli space. Consider a pair of objects $p < q$ such that there is no intermediate object $r$ with $p < r < q$ (this is then also true for $p^\dagger$ and $q^\dagger$).  Let $K_{pq} = (G_{pq}, V_{pq}, E_{pq}, S_{pq} )$ be an AMS global chart of $M_{pq}^{\mb F}$, viewed as an object of $\outer \uds{\bf S^{\rm rel} Kur}_{\rm rig}$. We can ignore $E_{pq}$ and $S_{pq}$ for the discussion but should remember the relative smooth structure on $V_{pq}/B_{pq}$ as well as the isomorphism 
\beqn
V_{pq}/B_{pq} \cong V_{p^\dagger q^\dagger}/ B_{ p^\dagger q^\dagger}.
\eeqn
For each $(\phi, u, F, e) \in V_{pq}$, denote by $(\phi^\dagger, u^\dagger, F^\dagger, e^\dagger)\in V_{p^\dagger q^\dagger}$ its conjugate. We will also drop the outercollaring from notations temporarily. 

We first specify a family of complex linear Cauchy--Riemann operators for points in $V_{pq}$. 

Fix a $J$-linear connection $\nabla^{TX}$ such that
\begin{enumerate}

\item $\tau_X$ preserves the connection $\nabla^{TX}$.

\item The holonomy of $\nabla^{TX}$ along each 1-periodic orbit of $H_t$ is trivial. 
\end{enumerate}

Choose $\tau > 2$.

We first look at an individual smooth map $u: {\mb R} \times S^1 \to X$ which converges exponentially to 1-periodic orbits of $H$. Define the space
\beqn
\wt W^{1,\tau}(u^* TX) \subset W^{1,\tau}_{\rm loc}(u^* TX)
\eeqn
to be the space of sections which differ from $\nabla^{TX}$-parallel sections along the limiting periodic orbits by $W^{1,\tau}$-small terms near infinity. Then the $J$-linear connection $\nabla^{TX}$ induces a complex-linear Cauchy--Riemann operator
\beqn
D_u^{\mb C}: \wt W^{1,\tau}(u^*TX) \to L^\tau (\Lambda^{0,1} \otimes u^* TX).
\eeqn
As $u$ converges exponentially to periodic orbits, it is easy to see that $D_u^{\mb C}$ is Fredholm. Moreover, regarding the involution, one has the following commutative diagram
\beqn
\xymatrix{  \wt W^{1,\tau}(u^* TX) \ar[r] \ar[d]_{D_u^{\mb C}}  &      \wt W^{1,\tau}((u^\dagger)^* TX)  \ar[d]^{D_{u^\dagger}^{\mb C}} \\
             L^\tau( \Lambda^{0,1} \otimes u^* TX)  \ar[r]   &  L^\tau( \Lambda^{0,1} \otimes (u^\dagger)^* TX)   }
\eeqn
where the horizontal arrows (which are conjugate linear) are induced from $d\tau_X$ on the target and the domain conjugation.

The operator $D_u^{\mb C}$ and the linearization of the Floer equation can be connected by a particular 1-parameter family of Fredholm operators modulo certain finite-dimensional corrections. Choose a cut-off function 
\beqn
\chi: {\rm Dom}(u) \cong {\mb R} \times S^1 \to [0, 1]
\eeqn
which only depends on the ${\mb R}$-variable such that $\chi(s) = 0$ for $s \ll 0$ and $\chi(s) = 1$ for $s \gg 0$. For each object $p = (\gamma_p, \mu_p)$, let $u_p: {\mb R} \times S^1 \to X$ be the corresponding constant solution to the Floer equation (which does not depend on the capping) at $\gamma_p$. Then there is an associated linearized operator
\beqn
D_{u_p}(\xi) = \partial_s \xi + J \left( \nabla_t \xi - \nabla_\xi X_{H_t}(\uds\gamma_p) \right).
\eeqn
Let 
\beqn
\wt W^{1,\tau}({\mb R}\times S^1, u_p^* TX) \subset W^{1,\tau}_{\rm loc}({\mb R}\times S^1, u_p^* TX)
\eeqn
be the subspace of $W^{1,\tau}_{\rm loc}$-sections which satisfies the same asymptotic condition near $-\infty$ as $\wt W^{1,\tau}(u^* TX)$ and which is $W^{1,\tau}$-small near $+\infty$. Then consider the interpolation
\beq\label{operator_dp}
\begin{split}
D_p: \wt W^{1,\tau} ({\mb R}\times S^1, u_p^* TX) \to &\  L^\tau ({\mb R}\times S^1, \Lambda^{0,1}\otimes u_p^* TX)\\
\xi\mapsto &\  ( 1- \chi) (\nabla^{TX} \xi)^{0,1} + \chi D_{u_p} ( \xi).
\end{split}
\eeq
Then $D_p$ is a real-linear Fredholm operator. Notice that from the definition, $D_p$ only depends on the underlying 1-periodic orbit but not the capping. Then for $p$ and $p^\dagger$, one has the following commutative diagram
\beqn
\xymatrix{  \wt W^{1,\tau} ({\mb R}\times S^1, u_p^* TX) \ar[r] \ar[d]_{D_p}  & \wt W^{1,\tau} ({\mb R}\times S^1, u_{p^\dagger}^* TX)  \ar[d]^{D_{p^\dagger}} \\  L^\tau ({\mb R}\times S^1, \Lambda^{0,1}\otimes u_p^* TX)  \ar[r]  &   L^\tau ({\mb R}\times S^1, \Lambda^{0,1}\otimes u_{p^\dagger}^* TX)    }
\eeqn

We now choose, for each object $p= (\gamma_p, \mu_p)$, a finite-dimensional complex vector space $R_p^+$, a complex linear map
\beq\label{eqn_nup}
\nu_p: R_p^+ \to C_0^\infty( {\mb R}\times S^1, \Lambda^{0,1} \otimes u_p^* TX)
\eeq
such that the map
\beqn
\nu_p \oplus D_p: R_p^+ \oplus \wt W^{1,\tau} ({\mb R}\times S^1, u^* TX) \to L^\tau ({\mb R}\times S^1, \Lambda^{0,1}\otimes u_p^* TX)
\eeqn
is surjective. Denote
\beqn
R_p^-:= {\rm Ker}( \nu_p \oplus D_p) \subset R_p^+ \oplus \wt W^{1,\tau}({\mb R}\times S^1, u^* TX).
\eeqn
Notice that the pair $(R_p^-, R_p^+)$ does not depend on the Sobolev exponent $\tau$. We could choose $R_p^+$, $\nu_p$ together with $R_{p^\dagger}$, $\nu_{p^\dagger}$ with a conjugate linear isomorphism $R_p^+ \cong R_{p^\dagger}^+$ such that $\nu_p$ and $\nu_{p^\dagger}$ are compatible in the natural way. 

Now we would like to construct interpolations between $D_u$ and $D_u^{\mb C}$, up to factors coming from $D_p$ and $D_q$, in a 1-parameter family and verify that the involution can be incorporated into the argument.

More precisely, we would like to construct an interpolation  between 
\beqn
D_p \oplus D_u: \wt W^{1,\tau }(u_p^* TX) \oplus W^{1,\tau }(u^* TX) \to L^\tau (\Lambda^{0,1} \otimes u_p^* TX) \oplus L^\tau ( \Lambda^{0,1} \otimes u^* TX)
\eeqn
and 
\beqn
D_u^{\mb C} \wt \oplus D_q: \wt W^{1, \tau}(u^* TX) \wt \oplus \wt W^{1,\tau }(u_q^* TX) \to L^\tau(\Lambda^{0,1}\otimes u^* TX) \oplus  L^\tau (\Lambda^{0,1}\otimes u_q^* TX).
\eeqn
Here $\wt \oplus$ means the subset of the direct sum satisfying the matching condition. 

To build the interpolation, consider the family of cut-off functions
\beqn
\chi_\rho: {\mb R} \to [0, 1],\ \chi_\rho (s) = \chi( s- \rho).
\eeqn
Then for all $\rho \in {\mb R}$, define
\beq\label{family_CR_operator}
\begin{split}
D_{u,\rho}: \wt W^{1,\tau} (u^* TX) \to &\ L^\tau (\Lambda^{0,1}\otimes u^* TX),\\
\xi \mapsto &\ (1 - \chi_\rho) (\nabla^{TX} \xi)_{J_0}^{0,1} + \chi_\rho D_u (\xi).
\end{split}
\eeq
Notice that even if $u$ has bubble components, these operators are still defined. Then intuitively, we have
\begin{align*}
&\ \lim_{\rho \to -\infty} D_{u, \rho} \approx D_p \oplus D_u,\ &\ \lim_{\rho \to +\infty} D_{u, \rho } \approx D_u^{\mb C} \wt \oplus D_q.
\end{align*}
This is the underlying idea of interpolating the linearized operator $D_u$ and a complex linear operator $D_u^{\mb C}$, up to the two operators $D_p$ and $D_q$ which are independent of $u$.

To better formulate this construction over the whole Kuranishi chart $V_{pq}$, we need to introduce an auxiliary space. Denote
\beqn
\wt B_{pq}:= B_{pq} \times [-\infty, +\infty],
\eeqn
viewed as a product of stratified spaces where $[-\infty, +\infty]$ has one open stratum and two boundary strata. We define an auxiliary universal curve $\wt C_{pq} \to \wt B_{pq}$ whose fiber $\wt C_{\phi, \rho}$ at $(\phi, \rho)$ with $\phi \in B_{pq}$ and $\rho \in (-\infty, +\infty)$ is still $C_\phi$, while for $\rho = -\infty$ resp. $\rho = +\infty$, $\wt C_{\phi, \rho}$ has the additional component of the domain of the map $u_p$ resp. the domain of $u_q$. One can then define the additional space $\wt V_{pq}$ which is identical to $V_{pq} \times [-\infty, +\infty]$. However, for each point $(\wt\phi, \wt u, F, e) \in \wt V_{pq}$, we regard $u$ as a map $\wt u: \wt C_{\phi, \rho} \to X$ where if $\rho = \pm\infty$, $\wt u$ maps the extra cylindrical component to the constant loop at the negative or the positive end. Then $\wt V_{pq} \to \wt B_{pq}$ is still relatively smooth. The operator $D_{u,  \rho}$ can be viewed as a Cauchy--Riemann operator $D_{\wt u}$ acting on the bundle $\wt u^* TX$ (and extended to suitable Sobolev completions). We remark that this setup facilitates the gluing construction: when $\rho \in \pm\infty$, if the operator $D_{u, \rho}$ is surjective, then each vector in the kernel can be glued to an element in the kernel at a nearby point.  

With the above viewpoint understood, using basic gluing construction, one obtains the following technical lemma. Let 
\beqn
\mathring {\mc E}_{pq}^0 \to \wt V_{pq}
\eeqn
be the infinite-dimensional vector bundle whose fiber over $\wt x = (\wt \phi, \wt u, F, e)$ is the space
\beqn
\mathring {\mc E}_{\wt x}^0 = \Omega_c^{0,1}( \wt C_{\phi, \rho}, \wt u^* TX).
\eeqn
Notice that when $\rho = \pm \infty$, one has the splitting
\beq\label{eqn172}
\mathring {\mc E}_{\wt x}^0 = \Omega_c^{0,1}( {\mb R}\times S^1, u_p^* TX) \oplus \Omega_c^{0,1}( C_\phi, u^* TX)
\eeq
or
\beq\label{eqn173}
\mathring {\mc E}_{\wt x}^0 = \Omega_c^{0,1}( C_\phi, u^* TX) \oplus \Omega_c^{0,1}({\mb R}\times S^1, u_q^* TX).
\eeq

\begin{lemma}\label{lemma175}
There exist a unitary representation $W_{pq}$ of $G_{pq}$ and a continuous $G_{pq}$-equivariant bundle map over $\wt V_{pq}  $
\beqn
\wt\nu_{pq}: \uds R_p^+ \oplus \uds W_{pq} \oplus \uds R_q^+ 
\to \mathring {\mc E}_{pq}^0
\eeqn
satisfying the following conditions. Let $\wt x = (\wt \phi, \wt u, F, e) \in \wt V_{pq}$ where $\wt \phi = (\phi, \rho ) \in \wt B_{pq}$.
\begin{enumerate}

    \item When $\rho = -\infty$, with respect to the splitting \eqref{eqn172}, $\wt\nu_{pq}$ has the block form
    \beqn
    \wt \nu_{pq} = \left[ \begin{array}{ccc} \nu_p &  0 & 0 \\
                     0 & 0 & 0 \end{array}\right]
    \eeqn
    where $\nu_p$ is specified in \eqref{eqn_nup}.
    
    \item When $\rho = +\infty$, with respect to the splitting \eqref{eqn173}, $\wt\nu_{pq}$ has the block form
    \beqn
    \wt\nu_{pq} = \left[ \begin{array}{ccc} 0 &  \nu_{\wt x} & 0 \\
     0 & 0 &  \nu_q \end{array}\right]
     \eeqn
     where
    \beqn
    \nu_{\wt x}: W_{pq} \to \Omega_c^{0,1}(C_\phi, u^* TX)
    \eeqn
    is a complex-linear map.
    
    \item For each $\wt x \in \wt V_{pq}$, the image of $\wt\nu_{pq}$ at $\wt x$ is transverse to the image of the operator $D_{\wt u}$. 
\end{enumerate}
\end{lemma}

\begin{proof}
This follows from basic Fredholm theory and part of the gluing argument.
\end{proof}

\subsubsection{An extended index bundle}

Now assume $W_{pq}$ and $\wt\nu_{pq}$ are given as in Lemma \ref{lemma175}. Consider the family of linear maps parametrized by $\wt x = (\wt \phi, \wt u, F, e) \in \wt V_{pq}$:
\beqn
\wt \nu_{pq}|_{{\wt x}} \oplus D_{\wt u}: \big( R_p^+ \oplus W_{pq} \oplus R_q^+ \big) \oplus \wt W^{1,\tau}(\wt u^* TX) \to L^ \tau (\Lambda^{0,1}\otimes \wt u^* TX).
\eeqn

\begin{prop}\label{prop1711}
The union of kernels of $\wt \nu_{pq}|_{{\wt x}} \oplus D_{\wt u}$ forms a continuous \footnote{In fact it is a relative smooth vector bundle.} vector bundle 
\beqn
\wt I_{pq}  \to \wt V_{pq}.
\eeqn
\end{prop}

\begin{proof}
The proof is based on the typical gluing construction. However, one needs to maintain the linear structure in the gluing procedure. The pregluing is constructed using parallel transport on the target manifold $X$ and cut-paste on the domain, which are linear in the input. The use of the implicit function theorem can also respect the linear feature.
\end{proof}

Hence the vector bundles on the two boundary pieces, $\partial^\pm \wt V_{pq}:= V_{pq}\times \{\pm\infty\} \cong V_{pq}$ are (non-canonically) isomorphic. The family of kernels selects a homotopy class of isomorphisms. By the specific characterization of $\wt\nu_{pq}$ given in Lemma \ref{lemma175}, one can see that 
\beqn
\wt I_{pq} |_{\partial^- \wt V_{pq}} \cong \uds R_p^- \oplus T^{\rm vt} V_{pq} \oplus \uds W_{pq} \oplus \uds R_q^+;
\eeqn
and
\beqn
\wt I_{pq} |_{\partial^+ \wt V_{pq}} \cong \uds R_p^+ \oplus I_{pq}^{\mb C} \oplus \uds R_q^-.
\eeqn
Here $I_{pq}^{\mb C}$ is formed by the kernels of the complex-linear  operator 
\beqn
\wt\nu_{pq}|_{{\wt x}} \oplus D_u^{\mb C}: W_{pq} \oplus \wt W^{1,\tau}(C_\phi, u^* TX) \to L^\tau( C_\phi, \Lambda^{0, 1} \otimes u^* TX)
\eeqn
hence is a complex vector bundle.

We can then choose a continuous bundle isomorphism
\beqn
\uds R_p^- \oplus T^{\rm vt} V_{pq} \oplus \uds W_{pq} \oplus \uds R_q^+ \cong \uds R_p^+ \oplus I_{pq}^{\mb C} \oplus \uds R_q^-.
\eeqn
Notice that $G_{pq}$ acts trivially on the vector spaces $R_p^\pm$ and $R_q^\pm$. As $W_{pq}$ is complex, one can see that this is a stable isomorphism from $T^{\rm vt} V_{pq}$ to $I_{pq}^{\mb C}$ and hence induces an NC structure on $T^{\rm vt} V_{pq}$.

Again, one can treat $(p, q)$ and $(p^\dagger, q^\dagger)$ simultaneously to respect the involution. We omit the details. However, they will be covered by Lemma \ref{lemma_index_interpolator} proved below.

\subsection{Proof of Theorem \ref{thm71}}\label{subsection173}

\subsubsection{An extended Kuranishi flow category}

We still follow the approach of \cite{Bai_Xu_foundation} to extend the previous construction to all moduli spaces involved in  ${\mb F}$, while maintaining the involution symmetry. Recall that for each pair of objects $p\leq q$, in the AMS construction there is a $G_{pq}$-manifold $B_{pq}$ parametrizing certain curves in $\mb{CP}^{d_{pq}}$. Consider the moduli space
\beqn
\ov{\mc M}{}_{pq, +} (\mb{CP}^{d_{pq}})
\eeqn
as the analogue of $\ov{\mc M}_{pq}(\mb{CP}^d)$ which contains configurations with an additional marked point lying on the line $t = 0$ of cylindrical domains. There is the forgetful map with generic fiber ${\mb R}$
\beqn
\ov{\mc M}{}_{pq, +}(\mb{CP}^{d_{pq}}) \to \ov{\mc M}{}_{pq}( \mb{CP}^{d_{pq}}).
\eeqn
Let $\wt B_{pq}$ be the preimage of $B_{pq}$; let $\wt \phi$ denote a typical element of $\wt B_{pq}$ with the underlying point $\phi \in B_{pq}$. We use the convention that this interior marked point does not go to sphere bubbles but only stays in cylindrical components; namely, it can coincide with a node. Let 
\beqn
\wt C_{pq} \to \wt B_{pq}
\eeqn
be the pullback of $C_{pq} \to B_{pq}$,  which still gives an equivariant family of curves
\beqn
\wt {\mc C}_{pq}:= ({\mc G}_{pq}, \wt B_{pq}, \wt C_{pq}).
\eeqn
Let $C_{\wt \phi} \subset \wt C_{pq}$ be the fiber over $\wt \phi \in \wt B_{pq}$. Then there is a natural map 
\beqn
C_{\wt \phi} \to C_\phi
\eeqn
which collapses at most one constant cylindrical component. Notice that for each $\wt \phi \in \wt B_{pq}$, there is a special cylindrical component
\beqn
C_{\wt \phi}^\star \subset C_{\wt \phi}
\eeqn
which contains the marked point, which we call the {\bf cylinder of interpolation}.

Moreover, notice that there are two disjoint closed subsets 
\begin{align*}
&\ \partial^- \wt B_{pq},\ &\ \partial^+ \wt B_{pq}
\end{align*}
corresponding to configurations where the marked point is at the very left resp. very right of the curve. Notice that one has canonical identifications
\beqn
\partial^- \wt B_{pq} \cong B_{pq} \cong \partial^+ \wt B_{pq}.
\eeqn

Notice that $\wt B_{pq}$ and $\wt C_{pq}$ can also be organized as morphism spaces of a flow category. However, because they are objects in a special category that does not have an obvious symmetric monoidal structure, we will describe the flow category structure in a more {\it ad hoc} way, without specifying an underlying category of stratified objects. Define a ``product'' $\wt B_{pr} \boxtimes \wt B_{rq}$ via the diagram
\beqn
\vcenter{ \xymatrix{ \wt B_{pr} \boxtimes \wt B_{rq} \ar[r] \ar[d]  &   \wt B_{pq} \ar[d] \\
          B_{pr}\times B_{rq} \ar[r]_-{\zeta_{prq}^B}   &  B_{pq} } },
\eeqn
which we require to be a pullback. The ``product" universal curves can also be defined as a pullback, which admits a domain map
\beqn
\wt C_{pr} \boxtimes \wt C_{rq} \to \wt B_{pr} \boxtimes \wt B_{rq}.
\eeqn

Now suppose we have obtained a relatively smooth AMS lift  of ${\mb F}$ enriched in $\outer \uds{\bf S^{\rm rel} Kur}_{\rm rig}$. For each pair of $p \leq q$, let 
\beqn
\wt V_{pq} \to \wt B_{pq}
\eeqn
be the pullback of $V_{pq} \to B_{pq}$. Then one can define similarly $\wt V_{prq}:= \wt V_{pr} \boxtimes \wt V_{rq}$ and corresponding maps
\beqn
\wt V_{prq} \to \partial^{prq} \wt V_{pq}.
\eeqn

We remark that all these constructions are canonical and do not involve additional choices. Hence the $\Pi \times {\mb Z}$-symmetry and involution are canonically induced.

\subsubsection{The construction of the normal complex structure}\hfill

---{\bf The family of Cauchy--Riemann operators}---Then we would like to build a family of Fredholm operators indexed by points $\wt x = (\wt \phi, \wt u, F, e) \in \wt V_{pq}$ while maintaining the involution symmetry. Fix $\tau > 2$. The construction appears to depend on $\tau$ but the resulting normal complex structure does not. Notice that for each $\wt x \in \wt V_{pq}$, there is a pair of Banach spaces
\begin{align*}
&\ {\mc E}_{\wt x}^1:= \wt W^{1,\tau}(C_{\wt \phi}, \wt u^* TX),\ &\ {\mc E}_{\wt x}^0:= L^{\tau}(C_{\wt \phi}, \Lambda^{0,1} \otimes \wt u^* TX).
\end{align*}
We regard them as ``vector bundles'' over $\wt V_{pq}$ only in a set-theoretic sense, denoted by 
\begin{align*}
&\ {\mc E}_{pq}^1\to \wt V_{pq},\ &\ {\mc E}_{pq}^0 \to \wt V_{pq}.
\end{align*} 

We give an {\it ad hoc} construction of a family of Cauchy--Riemann operators between these spaces. First we construct a family of cut-off functions.

\begin{lemma}\label{lemma_cut_off_function}
There exists a collection of $G_{pq}$-invariant cut-off functions
\beqn
\chi_{pq}: \wt C_{pq} \to [0, 1]
\eeqn
for all pairs of objects $p \leq q$ of ${\mb F}$ satisfying the following conditions.
\begin{enumerate}

    \item $\chi_{pq} \equiv 1$ on components to the right of $C_{\wt \phi}^\star$ and $ \chi_{pq} \equiv 0$ on components to the left of $C_{\wt \phi}^\star$.

    \item On $C_{\wt \phi}^\star$, $\chi_{pq} = 0$ for $s \ll 0$ and $ \chi_{pq} = 1$ for $s \gg 0$ where $s$ is a cylindrical coordinate.

    \item $\chi_{pq}$ is collared.

    \item $\chi_{pq}$ respects the flow category structure. Namely, whenever $p< r < q$ and $\wt \phi \in \wt B_{pr} \boxtimes \wt B_{rq}$, if the marked point is on the $pr$-section of $C_{\wt\phi}$  then on this part $\chi_{pq}$ coincides with $\chi_{pr}$; if the marked point is on the $rq$-section then on this section $\chi_{pq}$ coincides with $\chi_{rq}$.

    \item $\chi_{pq}$ respects the rigidification. $\chi_{pq}$ does not vary with the normal direction near $\wt C_{pr}\boxtimes \wt C_{rq}$ inside $\partial^{prq} \wt C_{pq}$. 

    \item These functions respect the involution. Namely, under the identification $\wt C_{pq} \cong \wt C_{p^\dagger q^\dagger}$, one has $\chi_{pq} = \chi_{p^\dagger q^\dagger}$.

\end{enumerate}
\end{lemma}

\begin{proof}
This is a minor extension of \cite[Lemma 16.7]{Bai_Xu_foundation} to the case involving involutions. As the involution action is free, there is no trouble to maintain such a symmetry. 
\end{proof}

Now for each $\wt x = (\wt \phi, \wt u, F, e) \in \wt V_{pq}$, consider the Cauchy--Riemann operator
\beqn
D_{\wt x}: {\mc E}_{\wt x}^1 \to {\mc E}_{\wt x}^0
\eeqn
given by the same formula as \eqref{family_CR_operator}. We would like to consider the index bundles. Let
\beqn
\mathring {\mc E}_{\wt x}^0 \subset {\mc E}_{\wt x}^0
\eeqn
be the subspace of compactly supported smooth sections and let 
\beqn
\mathring {\mc E}_{pq}^0 \to \wt V_{pq}
\eeqn
denote the corresponding ``subbundle.''

\begin{defn}(cf. \cite[Definition 16.8]{Bai_Xu_foundation}) \label{defn_bundle_interpolator}
An {\bf involutive index bundle interpolator} over the AMS lift consists of the following items.
\begin{enumerate}

\item For each object $p \in {\rm Ob}{\mb F}$, a finite-dimensional complex vector space $R_p^+$ and a complex-linear map
\beqn
\nu_p: R_p^+ \to \Omega_c^{0,1} ({\mb R}\times S^1, u_p^* TX)
\eeqn
which is transverse to the image of $D_p$ (see \eqref{operator_dp}). We require that $R_p^+$ and $\nu_p$ only depend on the $\Pi\times {\mb Z}$-orbit, i.e., the underlying uncapped orbit $\uds\gamma_p$. Denote
\beqn
R_p^- = {\rm ker} (\nu_p \oplus D_p) \subset R_p^+ \oplus  \wt W^{1,\tau}({\mb R}\times S^1, u_p^* TX),
\eeqn
which is only a real vector space.

\item For each $p \leq q$, a finite-dimensional Hermitian representation $W_{pq}$ of $G_{pq}$ and a $G_{pq}$-equivariant complex-linear bundle map over $\wt V_{pq}$
\beqn
\wt\nu_{pq}: \uds R_p^+ \oplus \uds W_{pq} \oplus \uds R_q^+ \to \mathring {\mc E}_{pq}^0.
\eeqn

\item For each triple $p < r < q$, a complex-linear map
\beqn
\wt\iota_{prq}: W_{pr} \oplus R_r^+ \oplus W_{rq} \to W_{pq}
\eeqn
which is equivariant with respect to the group homomorphism $G_{pr}\times G_{rq} \to G_{pq}$.

\item Two collections of complex-linear isomorphisms
\begin{align}
&\ (R_p^+)^\dagger \cong R_{p^\dagger}^+,\ &\  (W_{pq})^\dagger \cong W_{p^\dagger q^\dagger}.
\end{align}
\end{enumerate}
They are required to satisfy the following conditions.
\begin{enumerate}
    \item {\bf (The original index bundle)} For each $\wt x  = (\wt \phi, \wt u, F, e) \in \partial^- \wt V_{pq}$, with respect to the natural decomposition
    \beqn
    \mathring {\mc E}_{\wt x}^0 \cong \Omega^{0,1}_c({\mb R}\times S^1, u_p^* TX) \oplus \Omega^{0,1}_c(C_\phi, u^* TX)
    \eeqn
    we have
    \beqn
    \wt\nu_{pq} = \left[ \begin{array}{ccc}  \nu_p & 0 & 0  \\  0 &  0  & 0 \end{array}\right]
    \eeqn

    \item {\bf (The complex index bundle)} For each $\wt x  = (\wt \phi, \wt u, F, e) \in \partial^+ \wt V_{pq}$, with respect to the natural decomposition
    \beqn
    \mathring {\mc E}_{\wt x}^0 \cong 
    \Omega_c^{0,1}( C_\phi, u^* TX) \oplus \Omega_c^{0,1}({\mb R}\times S^1, u_q^* TX)
    \eeqn
    we have
    \beqn
    \wt\nu_{pq} = \left[ \begin{array}{ccc} 0  &  \nu_{\wt x} &  0 \\
          0 & 0 & \nu_q \end{array}\right].
    \eeqn
    Here $\nu_{\wt x}: W_{pq} \to \Omega_c^{0,1}(C_\phi, u^* TX)$ is a complex-linear map.

    \item {\bf (Flow category structure I)} The above conditions imply that $\wt\nu_{pr}$ and $\wt\nu_{rq}$ determine a bundle map 
    \beqn
    \wt\nu_{prq}: (\uds R_p^+ \oplus \uds W_{pr} \oplus \uds R_r^+ \oplus \uds W_{rq} \oplus \uds R_q^+ ) |_{\wt V_{prq}} \to  \mathring {\mc E}_{pr}^0 \boxplus \mathring {\mc E}_{rq}^0.
    \eeqn
    We require that the following diagram commutes.
    \beqn
    \xymatrix{ \big( \uds R_p^+ \oplus \uds W_{pr} \oplus \uds R_r^+ \oplus \uds W_{rq} \oplus \uds R_q^+ \big)|_{\wt V_{prq}} \ar[rr]^-{\wt\nu_{prq}} \ar[dd]_-{{\rm Id}_{R_p^+} \oplus \wt\iota_{prq} \oplus {\rm Id}_{R_q^+}}   &  &  \mathring {\mc E}_{pr}^0 \boxplus \mathring {\mc E}_{rq}^0 \ar[dd] 
    \\ & &  \\ \uds R_p^+ \oplus \uds W_{pq} \oplus \uds R_q^+|_{ \partial^{prq} \wt V_{pq}}  \ar[rr]_-{\wt\nu_{pq}} & &  \mathring {\mc E}_{pq}^0|_{\partial^{prq} \wt V_{pq} }}
    \eeqn

    \item {\bf (Flow category structure II)} Whenever $p<r < s < q$, the following diagram commutes. 
    \beqn
    \xymatrix{   W_{pr} \oplus R_r^+ \oplus W_{rs} \oplus R_s^+ \oplus W_{sq} \ar[rrr]^-{\wt\iota_{prs} \oplus {\rm Id}_{R_s^+} \oplus {\rm Id}_{W_{sq}}} \ar[d]_{{\rm Id}_{W_{pr} \oplus {\rm Id}_{R_r^+} \oplus \wt\iota_{rsq}}}   & &   & W_{ps} \oplus R_s^+ \oplus W_{sq} \ar[d]^{\wt\iota_{psq}} \\
           W_{pr} \oplus R_r^+ \oplus W_{rq} \ar[rrr]_-{\wt\iota_{prq}} &  & & W_{pq} }
    \eeqn

    \item {\bf (Transversality)} For each $\wt x \in \wt V_{pq}$, the image of $\wt \nu_{pq}$, contained in $\mathring {\mc E}_{\wt x}^0 \subset {\mc E}_{\wt x}^0$, is transverse to the image of $D_{\wt x}$ (which is a condition true for all $\tau>2$). 

    \item {\bf (Collar)} The maps $\wt\nu_{pq}$ satisfy the natural collaring condition. Namely, near each stratum $\partial^\alpha \wt V_{pq}$, $\wt\nu_{pq}$ is independent of the collar coordinates.

    \item {\bf (Rigidification)} $\wt \nu_{pq}$ does not vary in the normal direction to $\wt V_{prq}$  specified by the rigidification.

    \item {\bf (Novikov equivariance)} These objects are equivariant with respect to the natural $\Pi \times {\mb Z}$-action.

    \item {\bf (Involution invariance)} The diagrams commute.
    \beqn
    \xymatrix{ R_p^+ \ar[rr]^-{\nu_p} \ar[d]   &    &   \Omega_c^{0,1}({\mb R}\times S^1, u_p^*TX)  \ar[d]\\
               R_{p^\dagger}^+ \ar[rr]_-{\nu_{p^\dagger}} &  & \Omega_c^{0,1}( {\mb R}\times S^1, u_{p^\dagger}^* TX)         }
    \eeqn
    \beqn
    \xymatrix{    (\uds R_p^+)^\dagger \oplus (\uds W_{pq})^\dagger \oplus (\uds R_q^+)^\dagger \ar[rr]^-{\wt\nu_{pq}}  \ar[d] &  &    \mathring {\mc E}_{pq}^0 \ar[d]\\
                \uds R_{p^\dagger}^+ \oplus \uds W_{p^\dagger q^\dagger} \oplus \uds R_{q^\dagger}^+ \ar[rr]_-{\wt \nu_{p^\dagger q^\dagger}} &    &   \mathring {\mc E}_{p^\dagger q^\dagger}^0 }.
    \eeqn
    \beqn
    \xymatrix{    W_{pr} \oplus R_r^+ \oplus W_{rq}  \ar[rr]^-{\wt\iota_{prq}} \ar[d]   &  &   W_{pq} \ar[d]\\
                 W_{p^\dagger r^\dagger} \oplus R_{r^\dagger}^+ \oplus W_{r^\dagger q^\dagger}  \ar[rr]_-{\wt\iota_{p^\dagger r^\dagger q^\dagger}}   &   &   W_{p^\dagger q^\dagger} }
    \eeqn
\end{enumerate}
\end{defn}

\begin{lemma}\label{lemma_index_interpolator}
There exists an involutive index bundle interpolator.
\end{lemma}

\begin{proof}
The inductive proof in the non-involutive case was given in \cite[Lemma 16.9]{Bai_Xu_foundation}. For the involutive case, it is the same argument because all the auxiliary data respect the symmetry.  
\end{proof}

Once the involutive index bundle interpolator is chosen, one automatically obtains a system of vector bundles over the spaces $\wt V_{pq}$. Indeed, for each $\wt x\in \wt V_{pq}$, denote
\beqn
\wt I_{pq}|_{\wt x}:= {\rm ker} ( \wt \nu_{pq}|_{\wt x} \oplus D_{\wt x}) \subset R_p^+ \oplus W_{pq} \oplus R_q^+ \oplus  {\mc E}_{pq}^1|_{\wt x}.
\eeqn
They form a topological vector bundle $\wt I_{pq} \to \wt V_{pq}$. We consider its restriction to $\partial^\pm \wt V_{pq}$. By the first two conditions of Definition \ref{defn_bundle_interpolator}, one has $G_{pq}$-equivariant identifications
\beqn
\wt I_{pq}|_{\partial^- \wt V_{pq}} \cong \uds R_p^- \oplus T^{\rm vt} V_{pq} \oplus \uds W_{pq} \oplus \uds R_q^+
\eeqn
and 
\beqn
\wt I_{pq}|_{\partial^+ \wt V_{pq}} \cong \uds R_p^+ \oplus I_{pq}^{\mb C} \oplus \uds R_q^-.
\eeqn
Here $I_{pq}^{\mb C}$ has fiber over $\wt x$ being
\beqn
{\rm ker} ( \nu_{\wt x} \oplus D_u^{\mb C})
\eeqn
hence a $G_{pq}$-equivariant complex vector bundle. Notice that $\partial^- \wt V_{pq} \cong V_{pq} \cong \partial^+ \wt V_{pq}$ canonically and the inclusions of $V_{pq}$ into $\wt V_{pq}$ from both sides are $G_{pq}$-equivariant homotopy equivalences. Hence the two restriction bundles are equivariantly isomorphic (without a canonical isomorphism).

Notice that the involution symmetry induces the equivariant identification (over $G_{pq} \cong G_{p^\dagger q^\dagger}$) 
\beqn
\wt I_{pq} \cong \wt I_{p^\dagger q^\dagger}
\eeqn
such that the restriction $I_{pq}^{\mb C} \cong I_{p^\dagger q^\dagger}^{\mb C}$ is conjugate linear. 

To construct a normal complex structure on the flow category, one must choose such bundle isomorphisms in a coherent way. Notice that the product $\wt V_{prq} \cong \wt V_{pr}\boxtimes \wt V_{rq}$ is the union
\beqn
\wt V_{pr}\times \partial^- \wt V_{rq} \underset{\partial^+ \wt V_{pr}\times \partial^- \wt V_{rq}}{\cup} \partial^+ \wt V_{pr} \times \wt V_{rq}.
\eeqn
Hence $\wt I_{pr}$ and $\wt I_{rq}$ determine a vector bundle
\beqn
\wt I_{prq} \to \wt V_{prq}
\eeqn
which is $\wt I_{pr} \oplus \uds R_r^- \oplus T^{\rm vt} V_{rq} \oplus \uds W_{rq} \oplus \uds R_q^+$ on the first piece and which is $\uds R_p^+ \oplus I_{pr}^{\mb C} \oplus \wt I_{rq}$ on the second piece.

\begin{lemma}\label{lemma1713}
There exists a system of $G_{pq}$-equivariant real vector bundle isomorphisms
\beqn
\wt I_{pq}|_{\partial^- \wt V_{pq}} \cong \wt I_{pq}|_{\partial^+ \wt V_{pq}}
\eeqn
satisfying the following conditions.

\begin{enumerate}

\item Whenever $p<r < q$, the following diagram commutes.
\beqn
\xymatrix{  \wt I_{prq}|_{\partial^- \wt V_{pr}\times \partial^- \wt V_{rq}}      \ar[r] \ar[d]  &  \wt I_{prq}|_{\partial^+ \wt V_{pr} \times \partial^- \wt V_{rq}}  \ar[r] & \wt I_{prq}|_{\partial^+ \wt V_{pr} \times \partial^+ \wt V_{rq}} \ar[d]\\
\wt I_{pq}|_{\partial^- \wt V_{prq}} \ar[rr] & & \wt I_{pq}|_{\partial^+ \wt V_{prq}}} 
\eeqn

\item The isomorphisms are collared and respect the rigidification.

\item The isomorphisms respect the involution in the following way. Namely, the diagram commutes.
\beqn
\xymatrix{   \wt I_{pq}|_{\partial^- \wt V_{pq}}  \ar[r] \ar[d]  &   \wt I_{pq}|_{\partial^+ \wt V_{pq}} \ar[d]\\
           \wt I_{p^\dagger q^\dagger}|_{\partial^- \wt V_{p^\dagger q^\dagger}} \ar[r]    &     \wt I_{p^\dagger q^\dagger} |_{\partial^+ \wt V_{p^\dagger q^\dagger}}     }
\eeqn

\end{enumerate}
\end{lemma}

\begin{proof}
Induction and homotopy triviality of the interpolating interval $[0,1]$. The involution symmetry can be preserved through the induction procedure. 
\end{proof}

Now as $I_{pq}^{\mb C}$ is complex and $R_p^\pm$, $R_q^\pm$ are trivial representations of $G_{pq}$, the isomorphism induces a $G_{pq}$-equivariant NC structure on $T^{\rm vt} V_{pq}$ (after taking stabilization by $W_{pq}$). The compatibility condition given in Lemma \ref{lemma1713} implies that the structural maps of the flow category respect the vertical NC structures. Therefore, one obtains a lift to $\outer \uds {\bf S^{\rm rel} Kur}_{\rm rig}^{\rm NC}$. The involution symmetry preserved in each step of the construction implies that the lift is also involutive. This finishes the proof of Theorem \ref{thm71}.

\subsection{Stable smoothing respecting the involution}

We will finish proving Theorem \ref{thm_flow_category_lift}, Theorem \ref{thm_PSS_lift}, Theorem \ref{thm_pearly_lift}, and Theorem \ref{thm_homotopy_lift}. At this moment, the flow categories etc., after outercollaring, have been lifted to the category $\outer \uds{\bf S^{\rm rel} Kur}_{\rm rig}^{\rm NC}$ respecting both the $\Pi \times {\mb Z}$-symmetry and the involution. Recall we also have the category $\outer \uds{\bf SKur}_{\rm rig}^{\rm NC}$, the category of smooth normally complex Kuranishi spaces. There is no direct functor relating to $\outer \uds{\bf S^{\rm rel} Kur}_{\rm rig}^{\rm NC}$. However, the relative structure allows us to apply the so-called stable smoothing. 

To systematically apply stable smoothing, one needs to use another intermediate category $\uds{\bf Kur}^+$ whose objects are Kuranishi spaces $K = (G, V, E, S)$ together with a finite-dimensional unitary representation $R$ of $G$ (see details in \cite[Section 17.1]{Bai_Xu_foundation}). It is again an involutive regular stratification category where the involution reverses the complex structure on $R$. There is also a corresponding category $\outer \uds{\bf S^{\rm rel} Kur}_{\rm rig}^+$. Notice that besides the forgetful functor, there is another functor
\beqn
\uds{\bf Kur}_{\rm rig}^+ \to \uds{\bf Kur}_{\rm rig}
\eeqn
by taking the stabilization of $(G, V, E, S)$ by the trivial bundle $R$.

\begin{defn}\cite[Definition 17.4]{Bai_Xu_foundation} A {\bf smoothing} of a flow category/bimodule/homotopy enriched in $\uds{\bf Kur}_{\rm rig}$ is a lift to $\uds{\bf SKur}_{\rm rig}$. A {\bf stable smoothing} of a flow category/bimodule/homotopy enriched in $\uds{\bf Kur}_{\rm rig}$ consists of a lift to $\uds{\bf Kur}_{\rm rig}^+$ and a smoothing of the stabilization.     
\end{defn}

We generalize the stable smoothing theorem of \cite{Bai_Xu_foundation} to the involutive case as follows.

\begin{thm}(cf. \cite[Theorem 17.5]{Bai_Xu_foundation}) Let $\hat {\mb F}$ be a relatively smooth AMS lift of the outercollaring of ${\mb F}$. Then there exists a stable smoothing which is a $\Pi \times {\mb Z}$-equivariant involutive lift of $\hat{\mb F}$.   
\end{thm}

\begin{proof}
The details are routine extensions of corresponding parts of \cite[Section 17]{Bai_Xu_foundation}. The involution does not introduce any additional difficulty because the action is free, which allows us to carry out the same induction procedure. 
\end{proof}

Moreover, the stable smoothing involves, for each Kuranishi space $K_{pq}$, a stabilization by a unitary representation $R_{pq}$. The complex structure is not needed for the smoothing. However, to obtain an NC structure on the tangent bundle after stabilization, we choose to use unitary representations. Then we choose to have isomorphisms
\beq\label{rpq_involution}
(R_{pq})^\dagger \cong R_{p^\dagger q^\dagger}
\eeq
of unitary representations.

Hence we have obtained an involutive lift of the outercollaring of ${\mb F}$ to $\outer \uds{\bf SKur}_{\rm rig}$. In the same way as proving \cite[Theorem 17.16]{Bai_Xu_foundation}, we can lift further to $\outer \uds{\bf SKur}_{\rm rig}^{\rm NC}$. Recall that the obstruction bundle $E_{pq}$ originally contains a copy of $Q_{pq}$ which is only a real representation. The proof of \cite[Theorem 17.16]{Bai_Xu_foundation} includes the step of doing an additional stabilization by another copy of $Q_{pq}$. So the obstruction bundles, after necessary stabilizations, become complex. By taking group quotient, this provides a flow category enriched in $\outer \uds{\bf dOrb}_{\rm rig}^{\rm NC}$. This finishes the proof of Theorem \ref{thm_flow_category_lift}.

The stable smoothing for the case of PSS/SSP bimodules can be carried out without further difficulties. In fact, the involution-equivariant situation is very similar, though not included, in \cite[Section 25.4]{Bai_Xu_foundation}. There is no difference in stable smoothing; however, to have an NC structure, one needs to choose the unitary representations to satisfy an analogue of \eqref{rpq_involution}. Then one can obtain a proof of Theorem \ref{thm_PSS_lift}. 

On the other hand, Theorem \ref{thm_pearly_lift} has no direct counterpart in \cite{Bai_Xu_2022}, which only treated the non-involutive case, which was called the ``deformed Morse-to-Morse bimodule'' (also in \cite{Bai_Xu_Arnold}, which was called the pearly bimodule). However, the construction given in \cite[Section 21.2]{Bai_Xu_foundation} can be generalized to the current involutive case. 

Lastly, even though the homotopy case stated in Theorem \ref{thm_homotopy_lift} has no equivariant counterpart proved in \cite{Bai_Xu_foundation}, the non-involutive case was covered by \cite[Theorem T]{Bai_Xu_foundation}. The involutive extension is step-by-step but straightforward repetition of the same argument to respect the involutions.

\bibliographystyle{amsalpha}
\bibliography{reference}

\end{document}